\documentclass{mcom-l}

\usepackage{amsmath,amssymb,mathabx,amsfonts,mathtools,bbm,algorithmic}
\def\supp{{\rm supp}}
\DeclarePairedDelimiter\ceil{\lceil}{\rceil}
\DeclarePairedDelimiter\floor{\lfloor}{\rfloor}

\usepackage{graphicx}
\usepackage{hyperref}
\usepackage{subfig}

\usepackage{todonotes}

\newtheorem{theorem}{Theorem}[section]
\newtheorem{lemma}[theorem]{Lemma}

\theoremstyle{definition}
\newtheorem{definition}[theorem]{Definition}

\theoremstyle{remark}
\newtheorem{remark}[theorem]{Remark}
\newtheorem{corollary}[theorem]{Corollary}

\numberwithin{equation}{section}

\begin{document}

\title[Computing equilibrium measures with power law kernels]{Computing equilibrium measures\\ with power law kernels}


\author{Timon S. Gutleb}
\address{Department of Mathematics, Imperial College London, London SW7 2AZ, UK}
\curraddr{}
\email{t.gutleb18@imperial.ac.uk}
\thanks{}

\author{Jos\'e A. Carrillo}
\address{Mathematical Institute, University of Oxford, Oxford OX2 6GG, UK}
\curraddr{}
\email{carrillo@maths.ox.ac.uk}
\thanks{}

\author{Sheehan Olver}
\address{Department of Mathematics, Imperial College London, London SW7 2AZ, UK}
\curraddr{}
\email{s.olver@imperial.ac.uk}
\thanks{}


\date{}

\dedicatory{}

\begin{abstract}
We introduce a method to numerically compute equilibrium measures for problems with attractive-repulsive power law kernels of the form $K(x-y) = \frac{|x-y|^\alpha}{\alpha}-\frac{|x-y|^\beta}{\beta}$ using recursively generated banded and approximately banded operators acting on expansions in ultraspherical polynomial bases. The proposed method reduces what is na\"\i vely a difficult to approach optimization problem over a measure space to a straightforward optimization problem over one or two variables fixing the support of the equilibrium measure. The structure and rapid convergence properties of the obtained operators results in high computational efficiency in the individual optimization steps. We discuss stability and convergence of the method under a Tikhonov regularization and use an implementation to showcase comparisons with analytically known solutions as well as discrete particle simulations. Finally, we numerically explore open questions with respect to existence and uniqueness of equilibrium measures as well as gap forming behaviour in parameter ranges of interest for power law kernels, where the support of the equilibrium measure splits into two intervals.
\end{abstract}

\maketitle

\section{Introduction}\label{sec:introductionsec}
The study of systems in equilibrium states is an important and commonly occurring aspect of the mathematical natural sciences and engineering disciplines. The specific problem of finding an \emph{equilibrium measure} for a dynamical system is encountered in swarming problems, where one seeks to understand the patterns and behaviours associated with natural processes stemming from biology, chemistry or physics by studying the minimization of some non-local interaction energy. Specific fields of study in which equilibrium measure problems find applications may for example be found in the description of the biological flocking and swarming behaviours exhibited by many different species of animals \cite{topaz_nonlocal_2006,parrish_animal_1997} as well as in the ensemble movements and self-organization of organisms of cellular scale and physical particle interactions \cite{BKSUB, BUKB,carrillo_adhesion_2018,carrillo_review_2017,hagan_dynamic_2006,carrillo_particle_2010,HP,kolokolnikov_emergent_2013}. For generality we will henceforth refer to the flocking or swarming objects simply as \emph{particles} and the interpretation of the model is left to the specific application. Along with other related systems all of these problems share a common mathematical continuum description in terms of probability measures $\rho$. This continuum description may be obtained from the originally discrete problem of particle dynamics with particle positions $x_i$, total particle number $N$ (assumed large) and given interaction potential $K$ described by
\begin{align}\label{eq:discreteversion}
\frac{\mathrm{d}v_i}{\mathrm{d}t}= f\left(|v_i|\right)v_i - \frac{1}{N} \sum_{j\neq i} K'({|x_i-x_j|}), 
\end{align}
where $v_i = \frac{\mathrm{d}x_i}{\mathrm{d}t}$ and $i = 1, ..., N$ and $f$ describes self-propulsion and friction, c.f. \cite{PhysRevLett.96.104302,BKSUB,CHM,carrillo_explicit_2016}. In the continuous limit this problem equates to finding probability measures or densities $\rho$ such that on the support of $\rho$ we have
\begin{align*}
\nabla K * \rho = \int K'({x-y}) \rho(y) \mathrm{d}y = 0.
\end{align*}
In this paper we specifically concern ourselves with potentials of attractive-repulsive power law form, i.e.
\begin{align*}
K(x-y) = \frac{|x-y|^\alpha}{\alpha}-\frac{|x-y|^\beta}{\beta},
\end{align*}
where the $\alpha$ term corresponds to the attractive and the $\beta$ term to the repulsive interactions and we require $\alpha>\beta$ for the existence of compactly supported solutions, cf. \cite{carrillo_explicit_2016}. Existence of global minimizers for problems of this type was discussed in \cite{canizo_existence_2015}, with results in  \cite{choksi_minimizers_2015,lopes_uniqueness_2019} proving that the minimizers are unique in the parameter ranges $\alpha \in (2,4)$ and $\beta	\in (-1,0)$. Explicit steady state solutions for various other parameter ranges where either $\alpha$ or $\beta$ is an even integer were obtained in \cite{carrillo_explicit_2016} and were conjectured to be unique global minimizers -- this conjecture was recently proved to be true in \cite{carrillo_radial_2021}. 
The power law equilibrium measures discussed in this paper also have close links to potential theory equilibrium measures with $\mathrm{log}$-kernels $K_0(x-y) = \mathrm{log}(|x-y|)$, which formally correspond to $\alpha = 0$ with a single interaction potential term, see \cite{deift1999orthogonal,saff2013logarithmic}, the numerical computation of which has previously been addressed by Olver \cite{olver_computation_2011} based on recurrence relationships satisfied by weighted Chebyshev polynomials.

In this paper we present a method for the one-dimensional power law equilibrium measure problem based on new recurrence relationships for weighted ultraspherical polynomials used for a sparse spectral method. The proposed method reduces the problem of computing the equilibrium measure for attractive-repulsive power law kernels from an optimization problem in measure space to one over the boundary of the support of the measure, i.e. single or two variable optimization problems.

The sections in this paper are organized as follows: Section \ref{sec:introductionsec} introduces the general mathematical framework of equilibrium measures as well as required elements of function approximation and spectral method approaches using ultraspherical polynomials. In Section \ref{sec:numericalequilibrium} we prove recurrence relationships and other properties required for the proposed method and related to the sparsity of the approximation, while Section \ref{sec:method} details the novel numerical approach for power law kernel equilibrium measures. In Section \ref{sec:twointerval} we detail how a modified version of these recurrence relationships can be used to compute equilibrium measures in parameter ranges in which the measure's support consists of two instead of a single interval. Section \ref{sec:analysis} addresses convergence of the proposed method. In Section \ref{sec:numericalexamples} we show numerical verifications of the method in problems with known solutions as well as comparison with alternative numerical approaches in certain problems without known solutions. In Section \ref{sec:numericalopenQ} we numerically investigate open questions regarding the uniqueness of solutions and gap forming behaviour, including comparisons with results obtained from discrete particle simulations. We conclude with a discussion of present and potential future research.

\subsection{General aspects of the theory of equilibrium measures} \label{sec:equilibriummeasuresintro}
\begin{definition}[Equilibrium measure]\label{def:mainequilmeas}
Given a kernel $K:\mathbb{R} \rightarrow \mathbb{R}$, the equilibrium measure is defined to be the unique measure $\mathrm{d}\rho(x)=\rho(x)\mathrm{d}x$ of mass $M$ such that the following expression is minimized:
\begin{equation}\label{energy}
\frac12\iint K(x-y) \mathrm{d}\rho(x)\mathrm{d}\rho(y) .
\end{equation}
\end{definition}
The total mass $M$ of such a measure $\mathrm{d}\rho(x)=\rho(x)\mathrm{d}x$ is defined to be
\begin{equation} \label{eq:masscondition}
M = \int_{\supp(\rho)} \rho(y) \mathrm{d}y,
\end{equation}
such that $M = 1$ corresponds to searching for a probability measure. In the absence of friction terms, the continuum evolution equation \cite{carrillo_derivation_2014,carrillo_explicit_2016} associated with  \eqref{eq:discreteversion} via the mean-field limit is an aggregation equation
\begin{equation}\label{evol}
\rho_t = \nabla \cdot (\rho \nabla K * \rho ).
\end{equation}
The energy \eqref{energy} is a Lyapunov functional for the previous evolution. Furthermore, the evolution equation \eqref{evol} can be interpreted as a gradient flow of the energy \eqref{energy}. This gradient flow structure \cite{villani_topics_2003,CMcV03,carrillo_global--time_nodate,balague_nonlocal_2013} implies that the equilibrium measures in the sense of Definition \ref{def:mainequilmeas} have to be steady state solutions of this evolution model. This leads to the Euler--Lagrange formulation of the problem:
\begin{align*}
K * \rho &= E, \quad \text{in } \supp(\rho),\\
K * \rho &\geq E, \quad \text{in } \mathbb{R}^d,
\end{align*}
which all local minimizers of $\rho$ must satisfy \cite{balague_dimensionality_2013}.
\subsection{Function approximation with ultraspherical polynomials} \label{sec:spectralultraspherical}
The ultraspherical or Gegenbauer polynomials are a well-studied example of a complete set of classical univariate orthogonal polynomials \cite{beals_special_2016,nist_2018} with natural domain $(-1,1)$ which have recently seen frequent successful use in numerical applications \cite{gautschi_orthogonal_2004,olver_fast_2013, townsend_automatic_2015, hale_ultraspherical_2019}. Every sufficiently smooth univariate function $f(x)$ defined on a real interval $\left(a,b \right)$ may be expanded in a set of orthogonal polynomials on this domain via
\begin{equation*}
    f(t) = \sum_{n=0}^\infty p_n(x) f_n = \mathbf{p}(x)^\mathsf{T} \mathbf{f},
\end{equation*}
where $f_n$ are the coefficients of $f(x)$ in 
$$
\mathbf{p}(x) := \begin{pmatrix}
p_0(x) \\
p_1(x) \\
           \vdots
         \end{pmatrix}.
$$
Numerical approximations may be obtained by truncating this sum and efficient computational tools exist to approximate the coefficients for a given function\footnote{FastTransforms, a recent open source C library  by Slevinsky \cite{slevinsky_conquering_2017,slevinsky_fast_2017,slevinsky_fasttransforms_2019} based on first expanding a function in a Chebyshev expansion using the DCT and utilising low rank techniques for converting between different Jacobi bases provides a convenient tool.}. We write $C_n^{(\lambda)}(x)$ to denote the ultraspherical polynomial of order $n$ with basis parameter $\lambda$ and adopt a vector notation to denote the full ultraspherical polynomial basis, i.e.:
\begin{align*}
     \mathbf{C}^{(\lambda)}(x) := \begin{pmatrix}
           C^{(\lambda)}_0(x) \\
           C^{(\lambda)}_1(x) \\
           \vdots
         \end{pmatrix}.
\end{align*}
With $0 \neq \lambda > -\frac{1}{2}$ the ultraspherical polynomials satisfy the orthogonality condition
\begin{equation}\label{eq:orthcondition}
\int_{-1}^1 (1-x^2)^{\lambda-\frac{1}{2}} C_n^{(\lambda)} (x)C_m^{(\lambda)} (x)\,\mathrm{d}x = \tfrac{2^{1-2\lambda}\pi \Gamma (n+2\lambda)}{n!(n+\lambda)(\Gamma (\lambda))^2} \delta_{nm},
\end{equation}
with respect to the weight $w(x)=(1-x^2)^{\lambda-\frac{1}{2}}$. This corresponds to the special case of $\alpha = \beta = \lambda - 1/2$ for the more general Jacobi polynomials $P_n^{(\alpha,\beta)}(x)$ but with a slightly different scaling \cite[18.7.1]{nist_2018}
\begin{equation*}
C_n^{(\lambda)}(x) = \tfrac{(2\lambda)_n}{(\lambda+\frac{1}{2})_n} P_n^{(\lambda-\frac{1}{2},\lambda-\frac{1}{2})}(x),
\end{equation*}
where $(\cdot)_n$ denotes the Pochhammer function or rising factorial. Classical orthogonal polynomials satisfy a three-term recurrence relationship which reduces to a two-term recurrence relationship for ultraspherical polynomials \cite[18.9.1]{nist_2018}:
\begin{equation}\label{eq:ultrasphrec}
x C_n^{(\lambda)}(x) = \tfrac{(n+2\lambda-1)}{2(n+\lambda)}C_{n-1}^{(\lambda)}(x) + \tfrac{n+1}{2(n+\lambda)}C_{n+1}^{(\lambda)}(x).
\end{equation}
This allows the definition of the Jacobi operator, the transpose of which can be used to define a tridiagonal multiplication-by-$x$ operator on coefficient space, i.e.
\begin{align*}
    &x f(x) = \mathbf{C}^{(\lambda)}(x)^\mathsf{T} \mathrm{X} \mathbf{f} 
\end{align*}
for $f(x) = \mathbf{C}^{(\lambda)}(x)^\mathsf{T} \mathbf f$ where $\mathrm{X}$ is tridiagonal.
Combined with element-wise addition and multiplication, as well as sparse derivative, basis change and integral operators this yields a numerical toolbox for function approximation and sparse spectral methods for integral and differential equations \cite{olver_fast_2013}. Many applications of ultraspherical spectral methods may require conversion methods between ultraspherical polynomials with different parameters $\lambda_1$ and $\lambda_2$. Townsend, Webb and Olver \cite{townsend_fast_2018} have recently produced an extensive analysis and review of conversion operators between different Jacobi polynomial expansions, see also the review in \cite{olver_acta_2020}.
\begin{remark}
One very general implementation of the above ideas may be found in the ApproxFun.jl package \cite{noauthor_juliaapproximation/approxfun.jl_2019} written in the Julia computing language \cite{beks2017}. The ApproxFun package ecosystem is used for the numerical experiments in this paper.\\
\end{remark}
The $m=0$ special case of the orthogonality condition in \eqref{eq:orthcondition} is simply:
\begin{equation}\label{eq:n0integrationcondition}
\int_{-1}^1 (1-x^2)^{\lambda-\frac{1}{2}} (x)C_n^{(\lambda)} (x)\,\mathrm{d}x = \tfrac{2^{1-2\lambda}\pi \Gamma (n+2\lambda)}{n!(n+\lambda)(\Gamma (\lambda))^2} \delta_{n0}.
\end{equation}
Definite integration thus only involves the $n=0$ element of a coefficient vector. This simple result is of great use when normalization of an approximated function is a given constraint, as will be the case in Section \ref{sec:method}.\\
Beyond the fundamental two-term recurrence relationship, the proofs in this paper will make use of two more specific recurrence properties of the ultraspherical polynomials. The first of these is the following known indefinite integral recurrence of the Jacobi polynomials \cite[18.17.1]{nist_2018}: for $n \neq 0$,
\begin{equation*}
\int (1-x)^{\alpha} (1+x)^{\beta}P^{(%
\alpha,\beta)}_{n}\left(x\right) \mathrm{d}x = -\tfrac{(1-x)^{\alpha+1}(1+x)^{\beta+1}}{2n}P^{(\alpha+1,%
\beta+1)}_{n-1}\left(x\right) + {\rm const},
\end{equation*}
which takes the following form for ultraspherical polynomials \cite[18.9.20]{nist_2018}: for $n \neq 0$
\begin{equation}\label{eq:ultrasphericalintegral}
\int(1-x^2)^{\lambda-\frac{1}{2}} C^{(\lambda)}_{n}\left(x\right) \mathrm{d}x = -\tfrac{2 \lambda }{n^2+2 \lambda  n}\left(1-x^2\right)^{\lambda +\frac{1}{2}} C_{n-1}^{(\lambda +1)}(x)  + {\rm const}.
\end{equation}
A further recurrence relationship of use for the proofs in this paper is one which relates multiplications with additional weight terms with degree and parameters of different ultraspherical polynomials \cite[18.9.8]{nist_2018}: 
\begin{equation}\label{eq:ultrasphaltrecurrence}
(1-x^{2})C^{(\lambda+1)}_{n}\left(x\right)=\tfrac{(n+2\lambda)(n+2\lambda+1)}{4\lambda(n+\lambda+1)}C^{(\lambda)}_{n}%
\left(x\right) -\tfrac{(n+1)(n+2)}{4\lambda(n+\lambda+1)}C^{(\lambda)}_{n+2}\left(x\right).
\end{equation}
\section{Towards a sparse spectral method for equilibrium measures} \label{sec:numericalequilibrium}
\subsection{Recurrence relationships for ultraspherical polynomials}\label{sec:recurrence}
We begin with a definition from fractional calculus which will be useful in our proofs:
\begin{definition}[R-L integral]\label{def:rl-integrals}
The left and right-handed Riemann--Liouville fractional integrals are defined respectively by
\begin{align*}
I_{L,a}^s \left[u\right](x) &:= \frac{1}{\Gamma(s)} \int_{a}^{x} (x-y)^{s-1}  u(y) dy,\\
I_{R,b}^s\left[u\right](x) &:= \frac{1}{\Gamma(s)} \int_{x}^{b} (y-x)^{s-1} u(y)  dy.
\end{align*}
\end{definition}
By construction they satisfy the following relationships with respect to differentiation and integration:
\begin{align}\label{eq:rlintegralproperties1}
\frac{d}{\mathrm{d}x}I^{s+1}_{L,a} \left[ f\right](x) = I^s_{L,a} \left[ f\right](x)&, \quad \frac{d}{\mathrm{d}x}I^{s+1}_{R,b} \left[ f\right](x) = I^s_{R,b} \left[ f\right](x), \\ \label{eq:rlintegralproperties2}
I^s_{L,a}\left[I^t_{L,a} \left[ f\right]\right] = I^{s+t}_{L,a}\left[f\right]&, \quad I^s_{R,b}\left[I^t_{R,b}\left[ f\right]\right] = I^{s+t}_{R,b}\left[f\right],
\end{align}
see e.g. \cite[Section 2]{milici2019introduction} and \cite[Section 2]{miller_introduction_1993} for a review of motivation, definitions and properties of fractional integrals. We will prove a recurrence relationship for Riemann--Liouville integrals when acting on ultraspherical polynomials. These fractional integral results are related to the spectrally convergent algorithms for half-integer and fractional integral equations using ultraspherical and Jacobi polynomials discussed in \cite{hale_olver_2018} but among other points differ in choice of bases resulting in operators with different bandedness properties. These recurrence results may then be combined to yield a useful corollary for computations with power-law kernels.

\begin{lemma}\label{lemma:twotermrecRLleft}
The left and right-handed Riemann--Liouville fractional integral operators
\begin{align*}
I^{1+\alpha}_{L,-1} \left[u\right](x) &= \frac{1}{\Gamma
(1+\alpha)} \int_{-1}^{x} (x-y)^{\alpha} u(y) \mathrm{d}y, \\
I^{1+\alpha}_{R,1} \left[u\right](x) &= \frac{1}{\Gamma
(1+\alpha)} \int_{x}^{1} (y-x)^{\alpha} u(y) \mathrm{d}y
\end{align*}
satisfy a two-term recurrence relationship when acting on the ultraspherical polynomials $C_n^{(\lambda)}(y)$ with weight $w(y) = (1-y^2)^{\lambda-\frac{1}{2}}$ such that
\begin{align*}
xI^{1+\alpha}_{L,-1}\left[wC_n^{(\lambda)}\right](x) &= \kappa_1I^{1+\alpha}_{L,-1}\left[ wC_{n-1}^{(\lambda)}\right](x) +\kappa_2 I^{1+\alpha}_{L,-1}\left[ w C_{n+1}^{(\lambda)}\right](x),\\
xI^{1+\alpha}_{R,1}\left[wC_n^{(\lambda)}\right](x) &= \kappa_1 I^{1+\alpha}_{R,1}\left[ wC_{n-1}^{(\lambda)}\right](x) +\kappa_2 I^{1+\alpha}_{R,1}\left[ w C_{n+1}^{(\lambda)}\right](x),
\end{align*}
where $n\geq2$ and the constants are
\begin{align*}
&\kappa_1 = \frac{(n-\alpha-1) (2 \lambda +n-1)}{2 n (\lambda +n)},\\
&\kappa_2 = \frac{(n+1) (2 \lambda +n+\alpha+1)}{2 (\lambda +n) (2 \lambda +n)}.
\end{align*}
\end{lemma}

\begin{proof}
We only present the proof for the left-handed case as the proof for the right-handed operator proceeds analogously. For the left-handed fractional integral, we multiply by $x$ and split it up to obtain
\begin{align*}
x &I_{L,-1}^{1+\alpha}\left[ w C_n^{(\lambda)}\right](x) = \tfrac{1}{\Gamma(1+\alpha)} \int_{-1}^{x} x (x-y)^{\alpha} w(y) C_n^{(\lambda)}(y) dy\\ &= \tfrac{1}{\Gamma(1+\alpha)}\int_{-1}^{x} (x-y) (x-y)^{\alpha} w(y) C_n^{(\lambda)}(y) dy + \tfrac{1}{\Gamma(1+\alpha)}\int_{-1}^{x} y (x-y)^{\alpha} w(y)  C_n^{(\lambda)}(y) dy.
\end{align*}
Expanding $y C_n^{(\lambda)}(y)$ using the two-term recurrence for the ultraspherical polynomials in  \eqref{eq:ultrasphrec} takes care of the second term. The first term can be rewritten and expanded as follows using the properties in (\ref{eq:rlintegralproperties1}-\ref{eq:rlintegralproperties2}):
\begin{align*}
\tfrac{1}{\Gamma(1+\alpha)} \int_{-1}^{x} (x-y)^{\alpha+1} w(y)  C_n^{(\lambda)}(y) dy  &= (1+\alpha)I_{L,-1}^{2+\alpha}\left[w C_n^{(\lambda)}\right](x)\\
&=  (1+\alpha)I_{L,-1}^{1+\alpha}\left[I_{L,-1}^{1}\left[w C_n^{(\lambda)}\right]\right](x).
\end{align*}
As the general antiderivative is known from  \eqref{eq:ultrasphericalintegral} it is straightforward to verify that 
\begin{align*}(1+\alpha)I_{L,-1}^{1}\left[w C_n^{(\lambda)}\right](y) = -\tfrac{(1+\alpha)2 \lambda}{n^2+2 \lambda  n}\left(1-y^2\right)^{\lambda +\frac{1}{2}}C^{(\lambda+1)}_{n-1}\left(y\right),
\end{align*}
which can then be expanded into a two-term recurrence relationship using the recurrence in  \eqref{eq:ultrasphaltrecurrence}:
\begin{align*}
&-\tfrac{(1+\alpha)2 \lambda}{n^2+2 \lambda  n}\left(1-y^2\right)^{\lambda +\frac{1}{2}}C^{(\lambda+1)}_{n-1}\left(y\right) \\ &= \tfrac{(1+\alpha)(n+1)}{2(n+\lambda)(n+2\lambda)} \left(1-y^2\right)^{\lambda -\frac{1}{2}}C^{(\lambda)}_{n+1}\left(y\right) - \tfrac{(1+\alpha)(n+2\lambda-1)}{2n(n+\lambda)} \left(1-y^2\right)^{\lambda -\frac{1}{2}}C^{(\lambda)}_{n-1}\left(y\right),
\end{align*}
Thus we have reduced all parts of the left-handed Riemann--Liouville integral to the desired recurrence form. Summing all terms one obtains:
\begin{align*}
&xI_{L,-1}^{1+\alpha}\left[w C_n^{(\lambda)}\right](x) \\&= \tfrac{(n+1) (2 \lambda +n+\alpha+1)}{2 (\lambda +n) (2 \lambda +n)} I_{L,-1}^{1+\alpha}\left[w C_{n+1}^{(\lambda)}\right](x) + \tfrac{(n-\alpha-1) (2 \lambda +n-1)}{2 n (\lambda +n)} I_{L,-1}^{1+\alpha}\left[w C_{n-1}^{(\lambda)}\right](x),
\end{align*} 
which concludes the proof.
\end{proof}
As the above only holds for $n\geq 2$ we treat the initial step $n=1$ separately in the following Lemma and delay discussion of the $n=0$ to the next section.
\begin{lemma}\label{lemma:n1left}
When acting on $C_1^{(\lambda)}(y) = 2 \lambda y$ with weight $w(y)=(1-y^2)^{\lambda-\frac{1}{2}}$ the left-handed Riemann--Liouville fractional integral reduces as follows:
\begin{equation*}
I^{1+\alpha}_{L,-1} \left[wC_1^{(\lambda)}\right](x) = \tfrac{-2\lambda}{\Gamma
(1+\alpha)} \int_{-1}^{x} (x-y)^{\alpha+1} w(y) \mathrm{d}y + \tfrac{2\lambda x}{\Gamma
(1+\alpha)} \int_{-1}^{x} (x-y)^{\alpha} w(y) \mathrm{d}y.
\end{equation*}
and analogously for the right-handed Riemann--Liouville fractional integral:
\begin{equation*}
I^{1+\alpha}_{R,1} \left[wC_1^{(\lambda)}\right](x) = \tfrac{2\lambda}{\Gamma
(1+\alpha)} \int_{x}^{1} (y-x)^{\alpha+1} w(y) \mathrm{d}y + \tfrac{2\lambda x}{\Gamma
(1+\alpha)} \int_{x}^{1} (y-x)^{\alpha} w(y) \mathrm{d}y.
\end{equation*}
\end{lemma}
\begin{proof}
Both cases proceed analogously, so we only prove the left-handed one:
\begin{align*}
I^{1+\alpha}_{L,-1}& \left[wC_1^{(\lambda)}\right](x) = \tfrac{2\lambda}{\Gamma
(1+\alpha)} \int_{-1}^{x} y  (x-y)^{\alpha} w(y) \mathrm{d}y\\ 
&=\tfrac{2\lambda}{\Gamma
(1+\alpha)} \int_{-1}^{x} (y-x)  (x-y)^{\alpha} w(y) \mathrm{d}y + \tfrac{2\lambda}{\Gamma
(1+\alpha)} \int_{-1}^{x} x  (x-y)^{\alpha} w(y) \mathrm{d}y\\
&=\tfrac{-2\lambda}{\Gamma
(1+\alpha)} \int_{-1}^{x} (x-y)^{\alpha+1} w(y) \mathrm{d}y + \tfrac{2\lambda x}{\Gamma
(1+\alpha)} \int_{-1}^{x} (x-y)^{\alpha} w(y) \mathrm{d}y.
\end{align*}
\end{proof}
The preceding results now yield two useful corollaries.
\begin{theorem}\label{theorem:twotermrecFullop}
The integral operator
\begin{equation*}
Q^{\alpha} \left[u\right](x) = \int_{-1}^{1} | x-y |^{\alpha} u(y) \mathrm{d}y
\end{equation*}
satisfies a two-term recurrence relationship when acting on the ultraspherical polynomials $C_n^{(\lambda)}(y)$ with weight $w(y) = (1-y^2)^{\lambda-\frac{1}{2}}$ such that
\begin{align*}
xQ^{\alpha}\left[wC_n^{(\lambda)}\right](x) = \kappa_1 Q^{\alpha}\left[ wC_{n-1}^{(\lambda)}\right](x) +\kappa_2 Q^{\alpha}\left[ w C_{n+1}^{(\lambda)}\right](x),
\end{align*}
where $n\geq2$ and with the constants
\begin{align*}
&\kappa_1 = \frac{(n-\alpha-1) (2 \lambda +n-1)}{2 n (\lambda +n)},\\
&\kappa_2 = \frac{(n+1) (2 \lambda +n+\alpha+1)}{2 (\lambda +n) (2 \lambda +n)}.
\end{align*}
\end{theorem}
\begin{proof}
This result follows immediately from Lemma \ref{lemma:twotermrecRLleft} when noting that one can rewrite this $Q^{\alpha}$ integral operator as a sum of left and right-handed Riemann--Liouville integrals (compare \cite[p.555]{ablowitz_complex_2003}):
\begin{align*}
\int_{-1}^{1} &| x-y |^{\alpha} (1-y^2)^{\lambda-\frac{1}{2}} C_n^{(\lambda)}(y) \mathrm{d}y = \\  &\int_{-1}^{x}(x-y)^{\alpha} (1-y^2)^{\lambda-\frac{1}{2}}  C_n^{(\lambda)}(y) \mathrm{d}y +\int_{x}^{1} (y-x)^{\alpha} (1-y^2)^{\lambda-\frac{1}{2}}    C_n^{(\lambda)}(y)  \mathrm{d}y,
\end{align*}
or in operator notation
\begin{equation*}
\tfrac{1}{\Gamma(1+\alpha)}Q^{\alpha}\left[wC_n^{(\lambda)}\right](x) = I_{L,-1}^{1+\alpha}\left[w C_n^{(\lambda)}\right](x) + I_{R,1}^{1+\alpha}\left[w C_n^{(\lambda)}\right](x).
\end{equation*}
\end{proof}
\begin{corollary}\label{corr:Qn1case}
When acting on $C_1^{(\lambda)}(y) = 2 \lambda y$ with weight $w(y)=(1-y^2)^{\lambda-\frac{1}{2}}$ the integral operator $Q^\alpha$ seen in Theorem~\ref{theorem:twotermrecFullop} reduces as follows:
\begin{align*}
Q^{\alpha} &\left[w C_1^{(\lambda)} \right](x)\\ &= 2\lambda \left( \int_{x}^{1} (y-x)^{\alpha+1} w(y) \mathrm{d}y -\int_{-1}^{x} (x-y)^{\alpha+1} w(y) \mathrm{d}y \right) + 2 \lambda x Q^{\alpha} \left[w\right](x).
\end{align*}
\begin{proof}
As with Theorem~\ref{theorem:twotermrecFullop}, this can be shown by writing $Q^\alpha$ as a sum of left- and right-handed Riemann--Liouville integrals, followed by using Lemma \ref{lemma:n1left}.
\end{proof}
\end{corollary}
\subsection{Structure and sparsity of the integral operator}\label{sec:structure}
In this section we discuss the structure and sparsity of the power law integral operator $Q^\alpha$ when acting on a function expanded in the weighted ultraspherical polynomials \begin{align*}
u(x) = (1-x^2)^{\lambda-\frac{1}{2}} \mathbf{C}^{(\lambda)}(x)^\mathsf{T} \mathbf{u},
\end{align*}
ranging from specific special to very general cases. The two-term recurrence relationships derived in the previous section make no assumptions about a connection between the weight parameter $\lambda$ and the kernel power $\alpha$, which is why they will be useful for the general case with an attractive and a repulsive kernel term. The observed sparsity pattern varies based on both the given power $\alpha$ and the chosen basis parameter $\lambda$.
\subsubsection{Diagonal operators for powers $\mathbf{\alpha\in (-1,1)}$}
In the special case regime $\alpha\in (-1,1)$, with appropriately chosen weight parameter $\lambda = -\frac{\alpha}{2}$ we can leverage a result due to Popov \cite{popov_properties_1963} which tells us that the power law $Q^{\alpha}$ operator is \emph{diagonal}:
\begin{lemma}
Within the range $x \in (-1,1)$ with $\alpha \in (-1,1)$ the integral operator
\begin{equation*}
Q^{\alpha} \left[u\right](x) = \int_{-1}^{1} | x-y |^{\alpha} u(y) \mathrm{d}y
\end{equation*}
evaluates as follows when acting on the weighted ultraspherical polynomials $C_n^{(\lambda)}(y)$ with $\lambda = -\frac{\alpha}{2}$ and $w(y) = (1-y^2)^{\lambda-\frac{1}{2}}$:
\begin{equation*}
Q^{\alpha} \left[wC_n^{-\left(\frac{\alpha}{2}\right)}\right](x) = \tfrac{(-1)^n\pi }{n B(-n+\alpha +1,n)\cos \left(\frac{\pi  \alpha }{2}\right)}C_n^{-\left(\frac{\alpha}{2}\right)}(x).
\end{equation*}
where $B$ is the Beta function
$$
B(x,y) = \tfrac{\Gamma(x) \Gamma(y)}{\Gamma(x+y)}.
$$
\begin{proof}
A proof of this result can be found in \cite{popov_properties_1963}.
\end{proof}
\end{lemma}
\subsubsection{Banded operators for higher powers $\mathbf{\alpha\in (-1,m)}$}
The above discussion is sufficient for treating equilibrium problems with power law kernels where the powers remain within the interval $(-1,1)$, with all the involved operators on the ultraspherical polynomials being diagonal. While highly efficient when $\alpha \in (-1,1)$, it is too restrictive for observing many interesting phenomena. Fortunately, by trading diagonal for more general banded operators via the recurrence relationships above, we can extend the above to cases where  $\alpha \in (-1,m)$, $m\in\mathbb{N}$ and thus treat arbitrarily high powers as long as the integrals converge and solutions exist.
\begin{lemma}\label{lemma:n01cases}
With the $\lambda$ parameter chosen such that $\lambda + \tfrac{\alpha}{2}$ is a non-negative integer with $\alpha>-1$ and $\lambda>-\tfrac{1}{2}$, the $n=0$ and $n=1$ cases required for the general recurrence in Theorem~\ref{theorem:twotermrecFullop} evaluate to  polynomials.
\begin{proof}
For the general powers case this will require making use of the following general result previously derived in the context of fractional Laplacians \cite{huang_explicit_2014,carrillo_explicit_2016}:
\begin{equation*}
\int_{B_R} |x-y|^\alpha (R^2-|y|^2)^{\frac{2k-\alpha-d}{2}} \mathrm{d}y = \frac{\pi^{\tfrac{d}{2}}}{\Gamma(\frac{d}{2})}B\left(\tfrac{\alpha+d}{2},\tfrac{2k+2-\alpha-d}{2} \right) R^{2k} {}_2 F_1\left(-\tfrac{\alpha}{2},-k;\tfrac{d}{2};\tfrac{|x|^2}{R^2} \right),
\end{equation*}
where ${}_2 F_1(a,b;c;z)$ denotes the Gauss hypergeometric function \cite[15.2.1]{nist_2018}. For the case $d=1$ and $R=1$ we thus have
\begin{equation}\label{eq:n0casegeneral}
\int_{-1}^1 |x-y|^\alpha (1-y^2)^{\frac{2k-\alpha-1}{2}} \mathrm{d}y = B\left(\tfrac{\alpha+1}{2},\tfrac{2k+1-\alpha}{2} \right) {}_2 F_1\left(-\tfrac{\alpha}{2},-k;\tfrac{1}{2};x^2 \right).
\end{equation}
If the second (or by symmetry the first) parameter in ${}_2 F_1$ is a negative integer $-n$ and the third parameter is not a non-positive integer then the Gauss hypergeometric function is a polynomial of order $n$ \cite[15.2.4]{nist_2018}. Explicitly: 
\begin{align}\label{eq:2f1polynomial}
{}_2 F_1(a,-n;c;z) = \sum_{j=0}^n (-1)^j \begin{pmatrix}
           n \\
           j
         \end{pmatrix} \frac{(a)_j}{(c)_j}z^j,
\end{align}
with $n\in\mathbb{N}$. The above expressions have domains of validity $\alpha \in (-d,2+2k-d)$ and $\alpha \in (-1,2k+1)$ respectively, cf. \cite{carrillo_explicit_2016}. The required minimal value of $k$ to be able to integrate with a given power of $\alpha>-1$ is thus determined by the expression $k > \frac{\alpha-1}{2}$. The condition is thus that the exponent $\lambda-\tfrac{1}{2} = \tfrac{2k-\alpha-1}{2}$ of the weight satisfies $\lambda>-\frac{1}{2}$. As seen in the previous section, the form of the $n=1$ case primarily revolves around the expression
\begin{equation*}
\int_{x}^{1} (y-x)^{\alpha+1} (1-y^2)^{\lambda-\frac{1}{2}} \mathrm{d}y -\int_{-1}^{x} (x-y)^{\alpha+1} (1-y^2)^{\lambda-\frac{1}{2}} \mathrm{d}y.
\end{equation*}
By Corollary \ref{corr:Qn1case} we see that if this expression is a polynomial for a simultaneous $\lambda$ as the $n=0$ case then the full $n=1$ expression will also evaluate to a polynomial. The above sum of integrals appearing in the $n=1$ case can be shown, either by using known integral representations of the hypergeometric function and a series of appropriate variable substitutions, to evaluate the expression
\begin{equation}
-2^{1-\alpha }\sqrt{\pi } \tfrac{\Gamma (\alpha +2) \Gamma \left(\lambda +\frac{1}{2}\right)}{\Gamma \left(\frac{\alpha }{2}\right) \Gamma \left(\frac{\alpha }{2}+\lambda +1\right)}\tfrac{\left(2 x^2 (\alpha +\lambda +1)\right) \, _2F_1\left(-\frac{\alpha }{2},-\lambda-\frac{\alpha }{2} ;\frac{1}{2};x^2\right)+\left(x^2-1\right) \, _2F_1\left(-\frac{\alpha }{2},-\lambda-\frac{\alpha }{2} ;-\frac{1}{2};x^2\right)}{\alpha  (\alpha +1) (\alpha +2 \lambda +1) x},
\end{equation}
again with conditions $\alpha>-1$ and $\lambda>-\frac{1}{2}$. By changing variables in  \eqref{eq:n0casegeneral} we can rewrite the $n=0$ expression into a form which is more useful for comparison:
\begin{equation*}
\int_{-1}^1 |x-y|^\alpha (1-y^2)^{\lambda-\frac{1}{2}} \mathrm{d}y = \tfrac{\Gamma \left(\frac{\alpha +1}{2}\right) \Gamma (\lambda+\frac{1}{2}) \, }{\Gamma \left(\lambda+\frac{\alpha }{2}+1\right)}\, _2F_1\left(-\tfrac{\alpha }{2},-\lambda-\tfrac{\alpha}{2};\tfrac{1}{2};x^2\right).
\end{equation*}
In this notation the required minimal value of $\lambda$ to be able to integrate with a given power of $\alpha>-1$ is given by the condition $\lambda > -\frac{1}{2}$. In order for the right-hand side to be a polynomial we require that $\lambda+\frac{\alpha}{2}$ be non-negative integer valued. The hypergeometric functions appearing in the $n=1$ case are found to automatically and simultaneously be polynomials if $\lambda+\frac{\alpha}{2}$ is a non-negative integer along with the mentioned integrability conditions $\lambda > -\frac{1}{2}$ and $\alpha > -1$. What remains to be shown is that when the hypergeometric functions are polynomials all constant summation terms in the numerator cancel so that the division by $x$ still yields a polynomial. For the terms in the numerator multiplied by factors involving $x^2$ this evidently poses no problem and what remains to be investigated is the $j=0$ term of the series
\begin{align*}
{}_2F_1 &\left(-\tfrac{\alpha }{2},-\lambda-\tfrac{\alpha}{2};\tfrac{1}{2};x^2\right)-{}_2F_1\left(-\tfrac{\alpha }{2},-\lambda-\tfrac{\alpha}{2};-\tfrac{1}{2};x^2\right) \\
&= \sum_{j=0}^{\lambda+\frac{\alpha}{2}} (-1)^j \begin{pmatrix}
           \lambda+\tfrac{\alpha}{2} \\
           j
         \end{pmatrix} \frac{(-\frac{\alpha }{2})_j}{(\frac{1}{2})_j}x^{2j} - \sum_{j=0}^{\lambda+\frac{\alpha}{2}} (-1)^j \begin{pmatrix}
           \lambda+\tfrac{\alpha}{2} \\
           j
         \end{pmatrix} \frac{(-\frac{\alpha }{2})_j}{(-\frac{1}{2})_j}x^{2j}\\
                  &= \sum_{j=1}^{\lambda+\frac{\alpha}{2}} (-1)^j \begin{pmatrix}
           \lambda+\tfrac{\alpha}{2} \\
           j
         \end{pmatrix} \left(-\frac{\alpha }{2}\right)_j \left(\frac{1}{{(\frac{1}{2})_j}}-\frac{1}{{(-\frac{1}{2})_j}} \right) x^{2j}.
\end{align*}
As can be seen from \eqref{eq:2f1polynomial} the 0th order terms are equal to $1$, meaning that the above difference only involves terms where $j=1$ or higher, thus allowing division by $x$ while remaining in finite polynomial form. Explicitly:
\begin{align*}
\sum_{j=1}^{\lambda+\frac{\alpha}{2}} (-1)^j \begin{pmatrix}
           \lambda+\tfrac{\alpha}{2} \\
           j
         \end{pmatrix} \left(-\frac{\alpha }{2}\right)_j \left(\frac{1}{{(\frac{1}{2})_j}}-\frac{1}{{(-\frac{1}{2})_j}} \right) x^{2j-1}.
\end{align*}
As all the other terms appearing in the $n=1$ expression are polynomials when $\lambda+\tfrac{\alpha}{2}$ is a non-negative integer this concludes the proof.
\end{proof}
\end{lemma}
\begin{remark}
Lemma \ref{lemma:n01cases} can also be used to show that the operator can in fact only be diagonal in the special case $\alpha \in (-1,1)$ discussed above. A diagonal operator would require the $n=0$ and $n=1$ cases to be polynomials of order $0$ and $1$ respectively, which for the appearing hypergeometric functions ${}_2 F_1(a,b;c;z)$ means that either $a=0$ or $b=0$. If $a=0$ for all the appearing hypergeometric functions we have $\alpha=0$. If $b=0$ then $\lambda = -\frac{\alpha}{2}$ which can only satisfy the integrability conditions $\lambda > -\frac{1}{2}$ and $\alpha > -1$ when $\alpha \in (-1,1)$.
\end{remark}
\begin{remark}
A good choice strategy for non-integer $\alpha > -1$ which results in a simultaneously polynomial right hand-side for the $n=0$ and $n=1$ cases is
\begin{equation}\label{eq:lambdaval}
\lambda = \begin{cases}
    \floor*{\frac{\alpha}{2}}-\frac{\alpha}{2},& \text{if } \floor*{\frac{\alpha}{2}}-\frac{\alpha}{2} > -\tfrac{1}{2}\\
    \ceil*{\frac{\alpha}{2}}-\frac{\alpha}{2},              & \text{otherwise.}
\end{cases}
\end{equation}
This always results in $\lambda + \frac{\alpha}{2}$ being a non-negative integer when $\alpha>-1$ and while it is not a unique choice it is the smallest choice to do so while still always satisfying the second integrability condition $\lambda > -\frac{1}{2}$ as long as $\alpha>-1$. Numerical experiments suggest that choosing lower admissible weight parameters also results in more stable solutions, as the boundary singularity is better approximated.\end{remark}
With the $n=0$ and $n=1$ cases in polynomial form we are now able to investigate the structure of the full $Q^\alpha$ operator acting on a coefficient vector. The number of non-zero bands is determined by the highest polynomial orders appearing in the $n=0$ and $n=1$ cases. In Figure \ref{fig:spy_diagonalandbanded} we present spy plots of operators obtained for two different values of $\alpha$.
\begin{figure}
     \subfloat[$\alpha = 3.9$]
    {{ \centering \includegraphics[width=6cm]{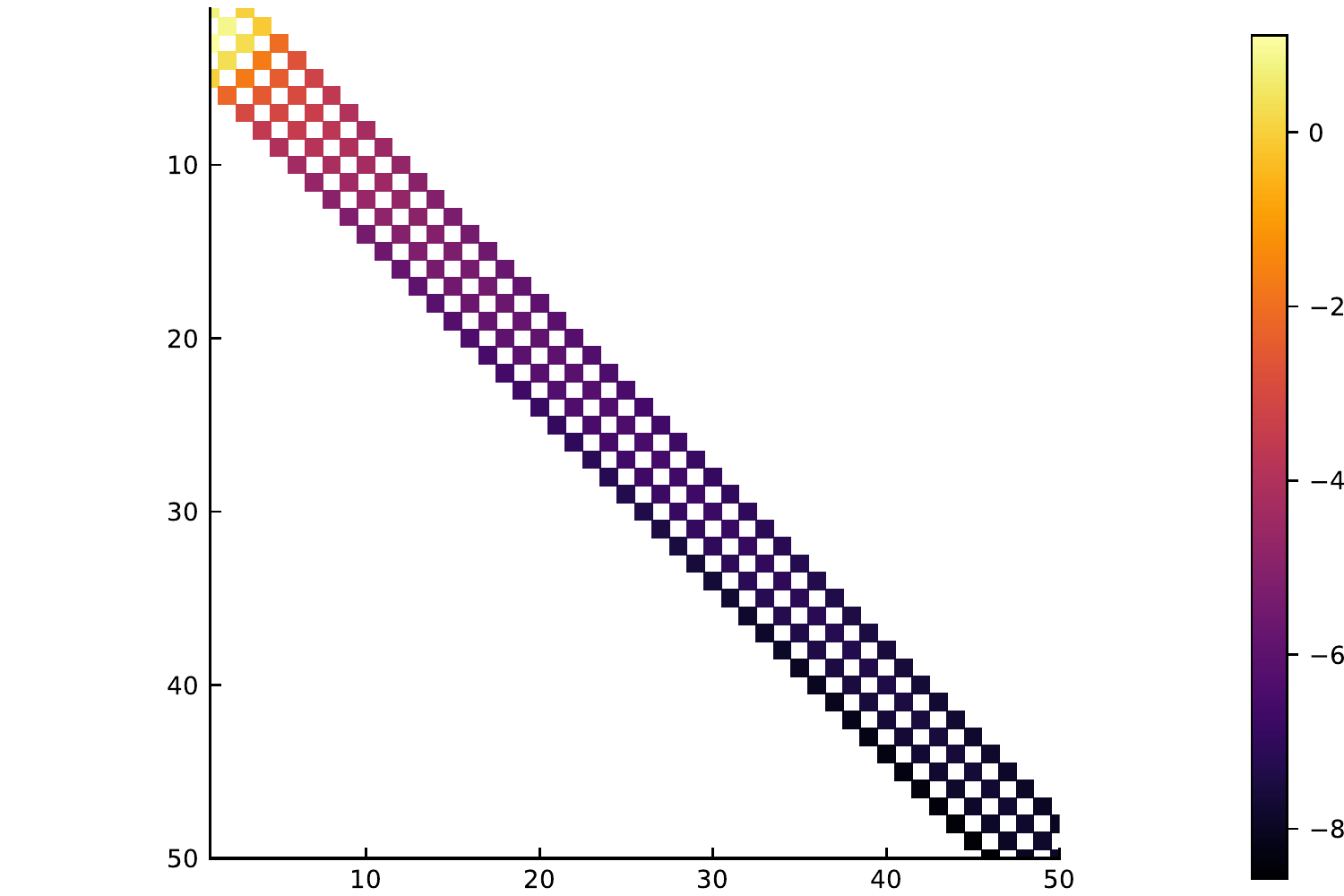} }}
         \subfloat[$\alpha = 2.5$]
    {{ \centering \includegraphics[width=6cm]{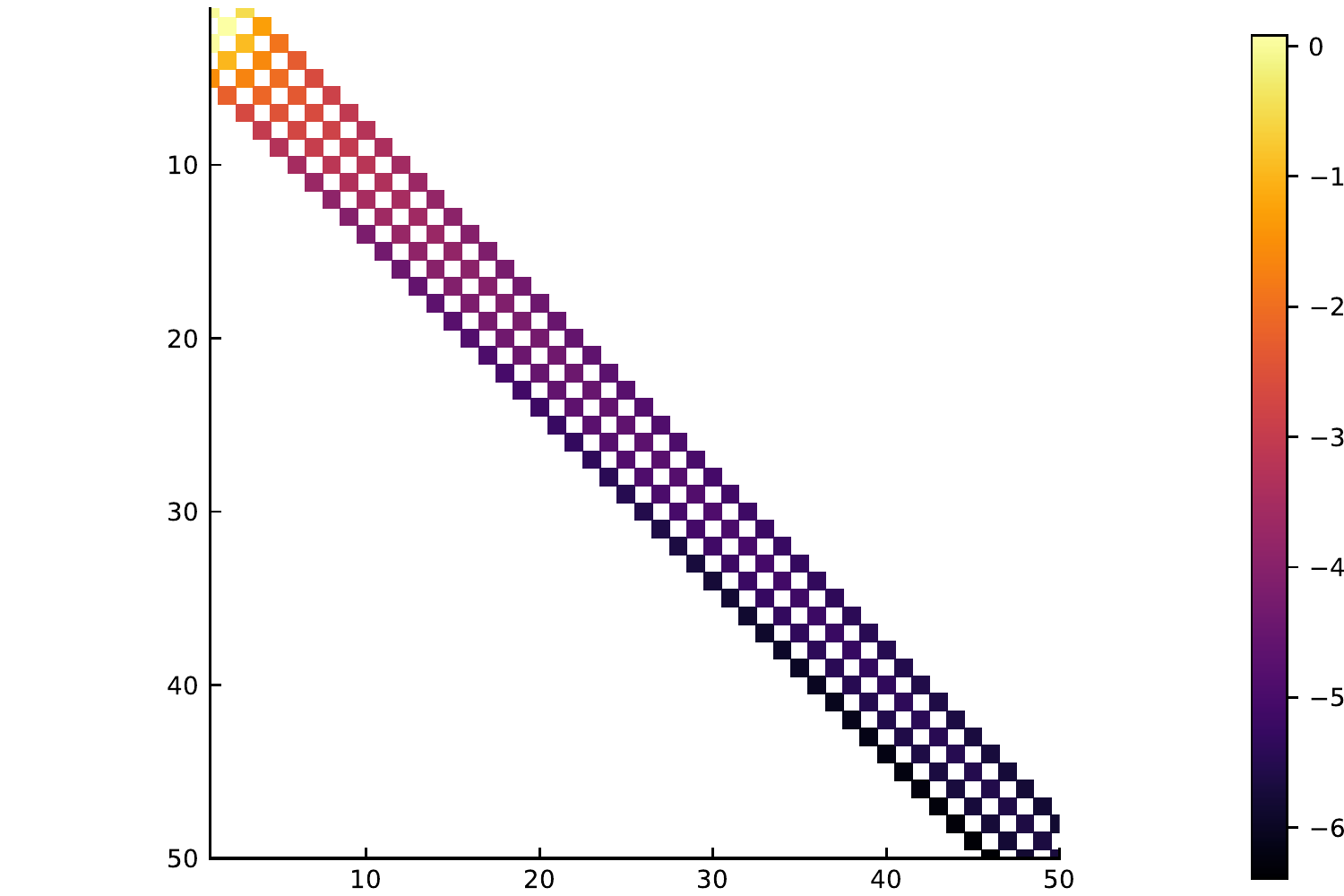} }}
    \caption{Banded operators for different values of $\alpha$ with $\lambda$ chosen as in \eqref{eq:lambdaval}. The legend is logarithmic, indicating order of magnitude of the entries.}%
    \label{fig:spy_diagonalandbanded}%
\end{figure}
\subsubsection{Sparsity structure for even integer $\alpha$ with arbitrary $\lambda$}\label{sec:bandedpowers}
A more specialized but nevertheless useful result when $\alpha$ is an even integer follows as a corollary of the above proof. Importantly, the choice of $\lambda$ becomes arbitrary for even integer $\alpha$.
\begin{corollary}
With $\alpha = 2n$, $n \in \mathbb{N}$ and for any $\lambda>-\tfrac{1}{2}$, the $n=0$ and $n=1$ cases required for the general recurrence in Theorem \ref{theorem:twotermrecFullop} evaluate to finite polynomials.
\begin{proof}
The proof follows the same steps as the proof of Lemma \ref{lemma:n01cases}. As a result of the Pochhammer symbols in the definition of the hypergeometric series the appearing hypergeometric functions ${}_2 F_1(a,b;c;z)$ are finite polynomials if either of the parameters $a$ or $b$ are negative integers. Observing that $a = -\frac{\alpha}{2}$ in all appearing hypergeometric functions in the $n=0$ and $n=1$ case concludes the proof.
\end{proof}
\end{corollary}
Due to cancellations which occur when $\alpha$ is an even integer, one finds that the operator then has sparsity in the form of a finite top-left-located triangle with side-length $\alpha+1$, see Figure \ref{fig:evenintegerop}. For attractive-repulsive problems of the form
\begin{align*}
\int_{-1}^{1} &| x-y |^{\alpha} (1-y^2)^{\lambda-\frac{1}{2}} C_n^{(\lambda)}(y) \mathrm{d}y - \int_{-1}^{1} | x-y |^{\beta} (1-y^2)^{\lambda-\frac{1}{2}} C_n^{(\lambda)}(y) \mathrm{d}y,
\end{align*}
where one of the powers $\alpha$ and $\beta$ is even integer valued, the full operator may be split into the two respective parts. This can be written as
\begin{equation*}
Q^{\alpha}\left[wC_n^{(\lambda)}\right]-Q^{\beta}\left[wC_n^{(\lambda)}\right],
\end{equation*} 
where we refer to the operators respectively as the $\alpha$-operator and $\beta$-operator as shorthands. The parameter $\lambda$ is then chosen to make the arbitrary power part of the operator banded as discussed above.
\begin{figure}
     \subfloat[$\alpha=2$.]
    {{ \centering \includegraphics[width=6cm]{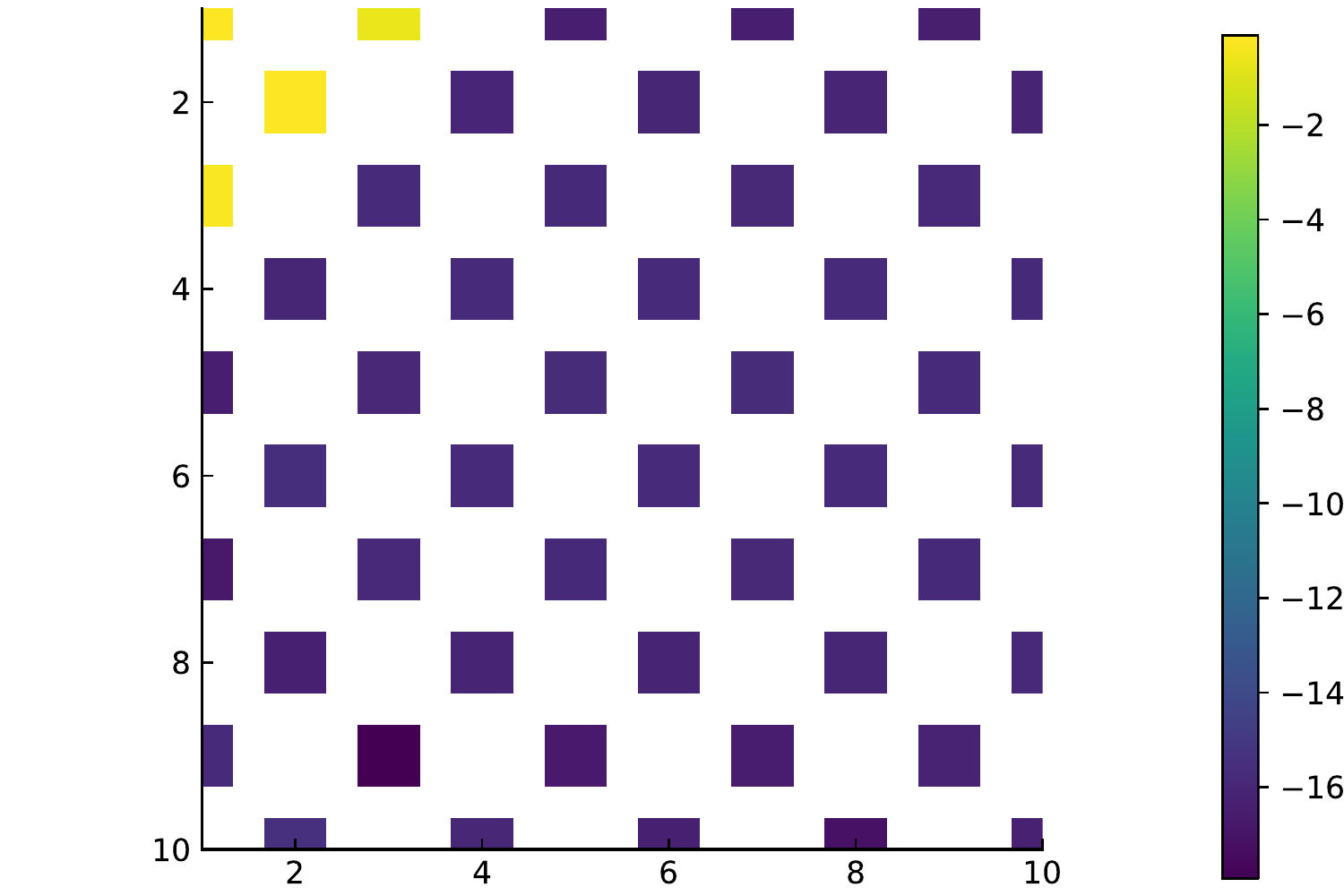} }}
         \subfloat[$\alpha=4$.]
    {{ \centering \includegraphics[width=6cm]{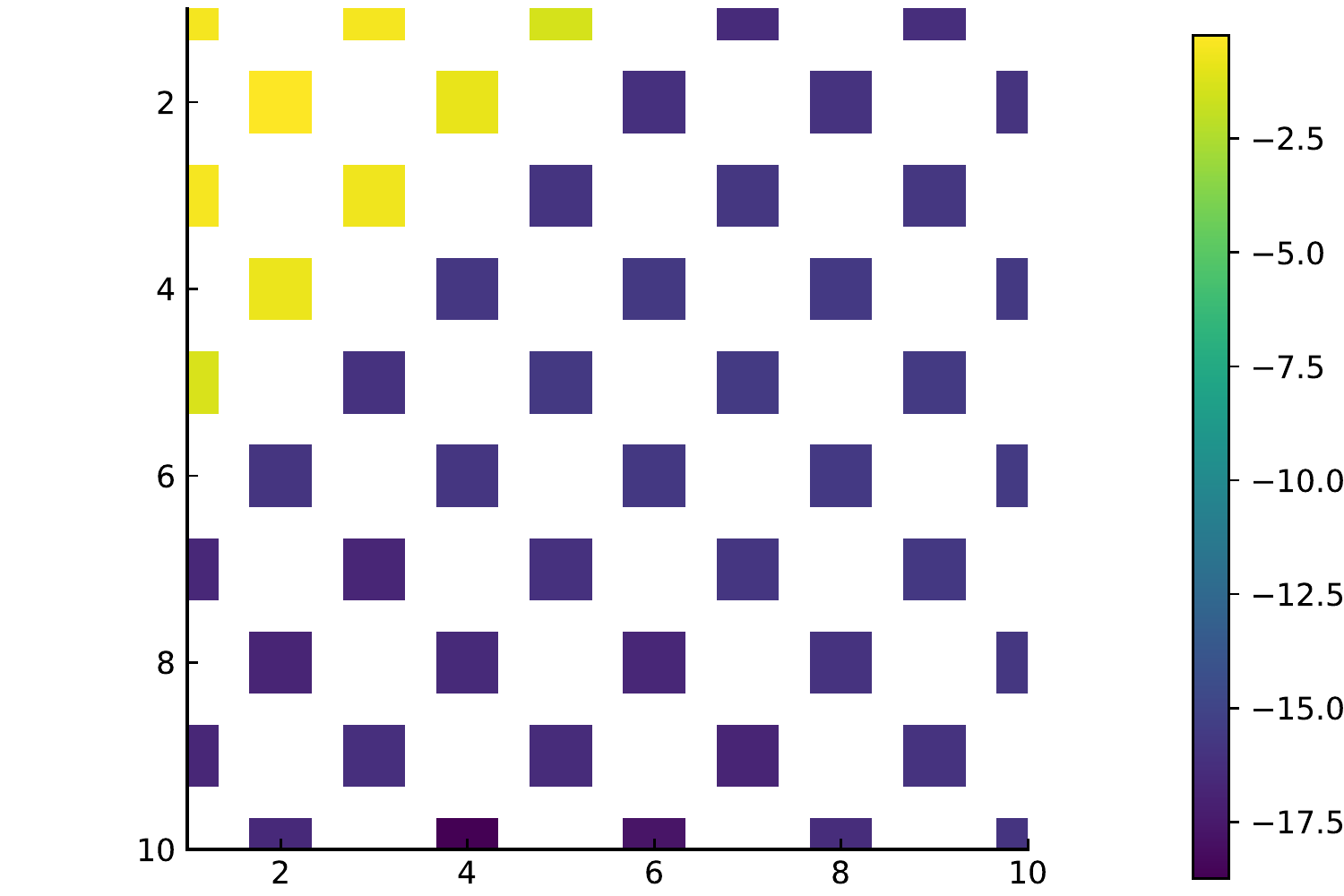} }}
    \caption{Spy plots for even integer valued $\alpha$ with $\lambda=\frac{\alpha}{2}$. The legend is logarithmic, indicating order of magnitude of the entries.}%
    \label{fig:evenintegerop}%
\end{figure}

\subsubsection{Approximate bandedness for simultaneous powers $\mathbf{\alpha, \beta}\in (-1,m)$}
In this section we discuss how to obtain operators which may be well approximated as \emph{simultaneously} banded for general arbitrary powers $\alpha$ and $\beta$ for applications in attractive-repulsive kernel problems, i.e. problems of the form
\begin{equation*}
Q^{\alpha}\left[wC_n^{(\lambda)}\right]-Q^{\beta}\left[wC_n^{(\lambda)}\right].
\end{equation*} 
The previous section discussed a method to accomplish this for the special cases where one of the two parameters is even integer-valued. Without loss of generality we will henceforth assume that $\lambda$ is chosen based on $\alpha$, i.e. the $\alpha$-operator will be assumed to be of banded form as discussed above. We will also posit that neither of the two parameters are integer valued. The question then becomes what the form of the $\beta$ operator is for a $\lambda$ which has not been adjusted to yield a polynomial---and the answer is largely already contained in the previous sections: the recurrence relationship we proved for such integral operators made no claims about a relationship between the parameter and $\lambda$, such a relationship was only required when we wanted it to result specifically in polynomial form. The above results for the general $n=0$ and $n=1$ cases hold for generic $\beta$ and $\lambda$. Thus one approach we may take is to approximate the non-polynomial functions in the $n=0$ and $n=1$ cases via the known expressions and then compute the remaining steps recursively. As neither of the initial two steps are polynomial the resulting operator is dense. In practice, spectral approximation of functions are truncated at finite polynomial degree, resulting in an upper triangular matrix along with a number of sub-diagonals determined by the approximation order of choice. The form of the recurrence relationship also results in most of the super diagonals tending to zero, resulting in an approximately banded operator. For the appropriate special cases such as when $\alpha$ or $\beta$ are an even integer the method for attractive-repulsive kernels reduces to the previous results as both operators can be chosen to be banded. If a desired bandwidth of approximation is already known for the $\beta$-operator, then it may be generated as a banded approximation directly by approximating the initial hypergeometric functions as polynomials of order corresponding to the desired bandwidth, thus omitting the computationally wasteful step of generating a dense operator first and then cropping it. In Figure \ref{fig:spy_approxbanded} we present spy plots of example $\beta$-operators obtained for different values of $\alpha$ and $\beta$ in a basis in which $\alpha$ is banded, including demonstrations that this method for generating the operators reduces to the diagonal and banded operators discussed in previous sections in the appropriate special cases. In general, the choice of $\lambda$ has an important impact on the accuracy of the approximation and the speed of convergence with increasing polynomial orders, meaning that despite the more favourable numerical structure, the smallest admissible $\lambda$ is generally to be preferred, see also Figure \ref{fig:spy_approxbanded_tradeoff}.
\begin{figure}
     \subfloat[$\alpha=0.8$.]
    {{ \centering \includegraphics[width=4.2cm]{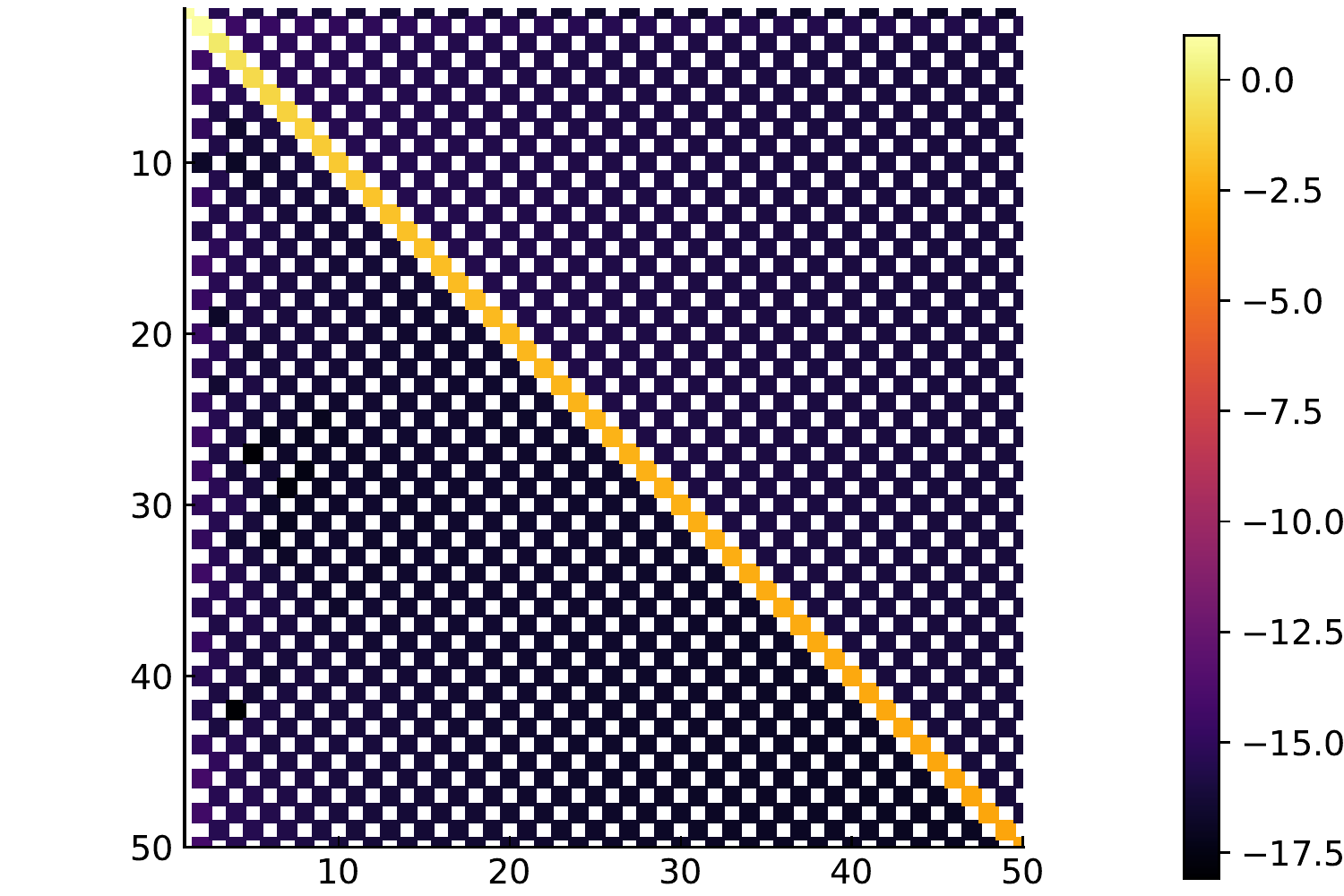} }}
         \subfloat[$\alpha=2.5$.]
    {{ \centering \includegraphics[width=4.2cm]{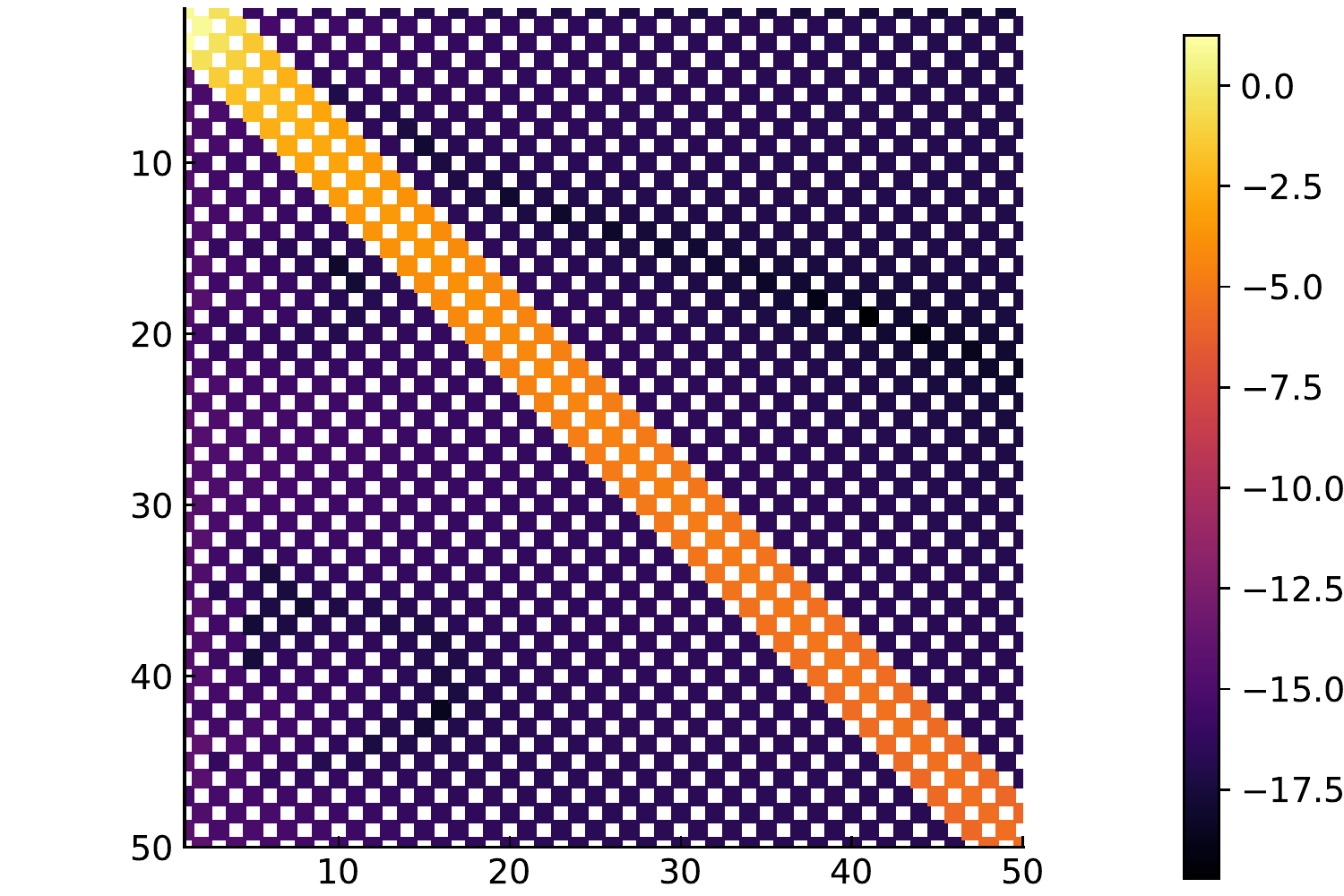} }}
             \subfloat[$\alpha=3.8, \beta=1.7$.]
    {{ \centering \includegraphics[width=4.2cm]{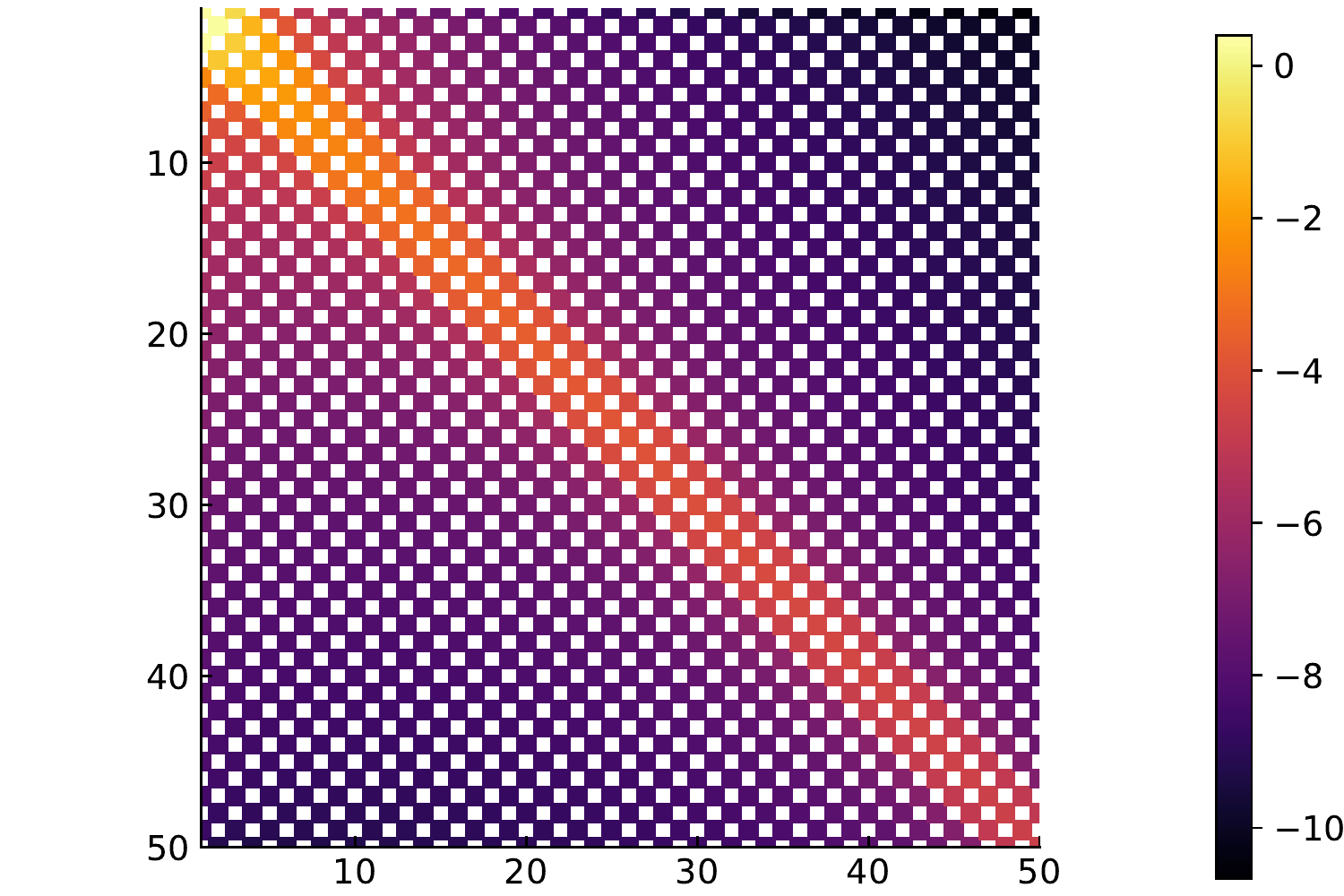} }}
    \caption{Spy plots of operators for different $\alpha$-banding bases. The approx. banded operator method includes the previously discussed diagonal (A) and banded (B) operators as special cases. (C) shows a typical generic attractive-repulsive example. The legend is logarithmic, indicating order of magnitude of the entries.}%
    \label{fig:spy_approxbanded}%
\end{figure}
\begin{figure}
     \subfloat[$\lambda = \floor*{\frac{\alpha}{2}}-\frac{\alpha}{2}$]
    {{ \centering \includegraphics[width=4.2cm]{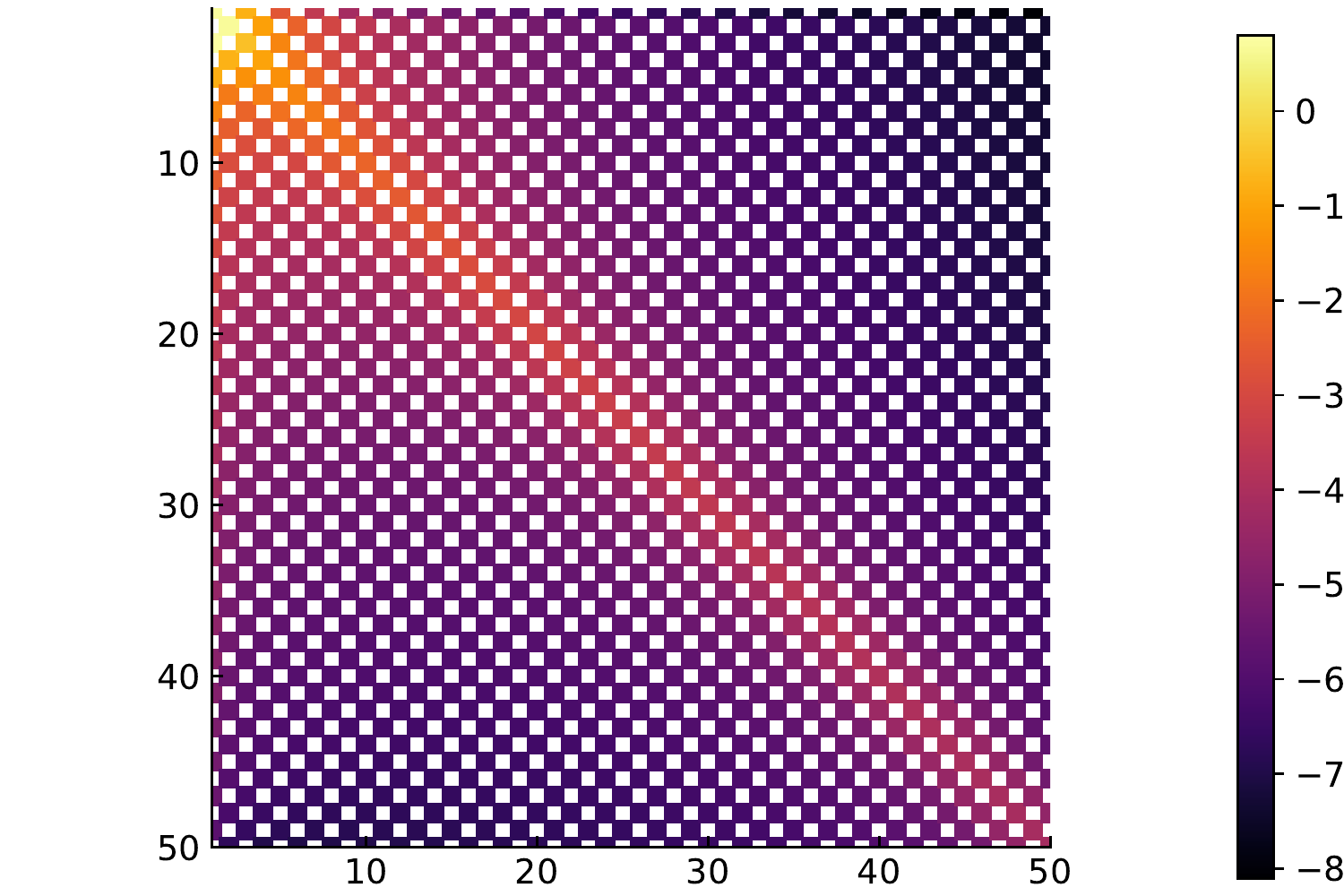} }}
         \subfloat[$\lambda = \ceil*{\frac{\alpha}{2}}-\frac{\alpha}{2}$]
    {{ \centering \includegraphics[width=4.2cm]{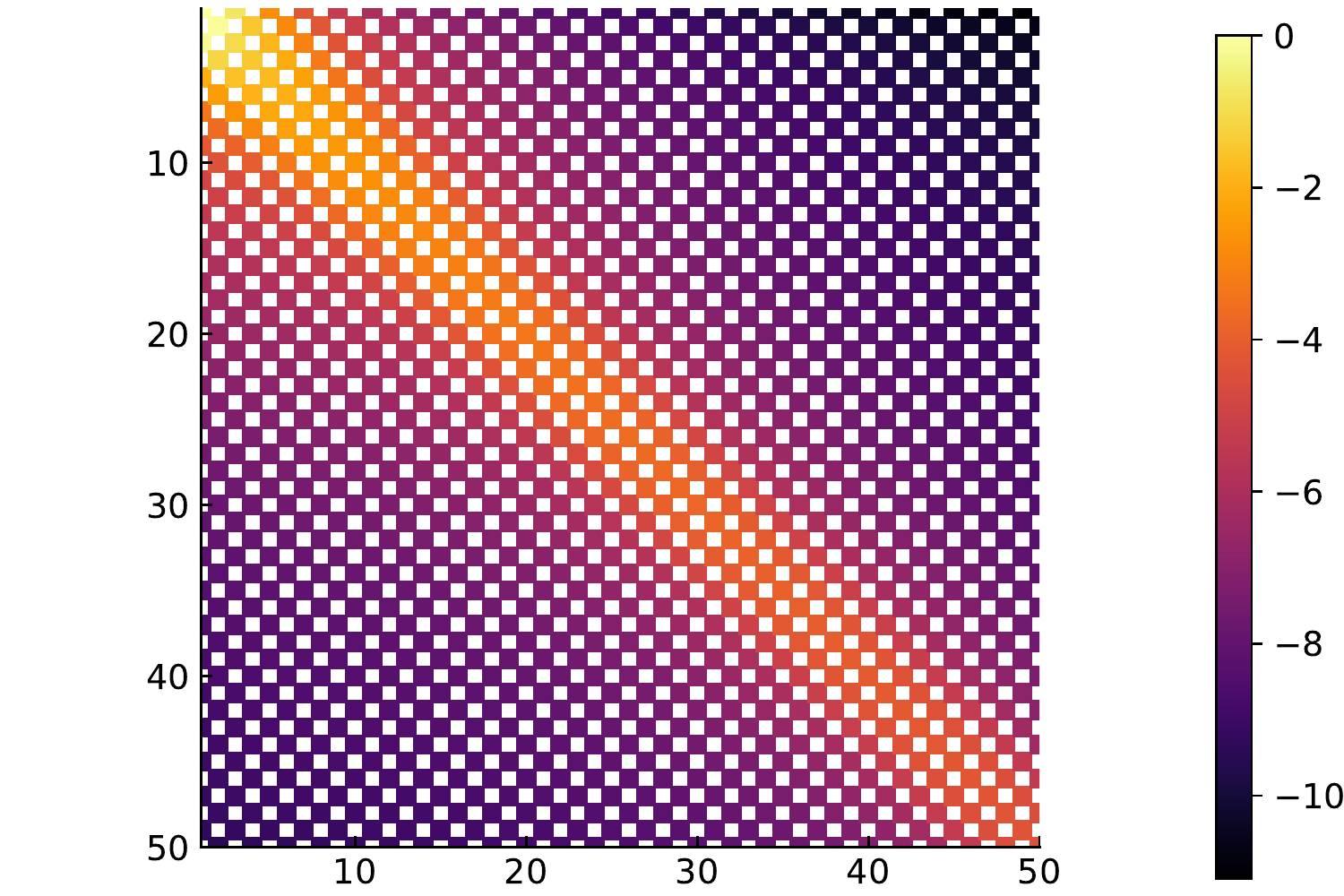} }}
             \subfloat[$\lambda = \floor*{\frac{\alpha}{2}}+\tfrac{2-\alpha}{2}$]
    {{ \centering \includegraphics[width=4.2cm]{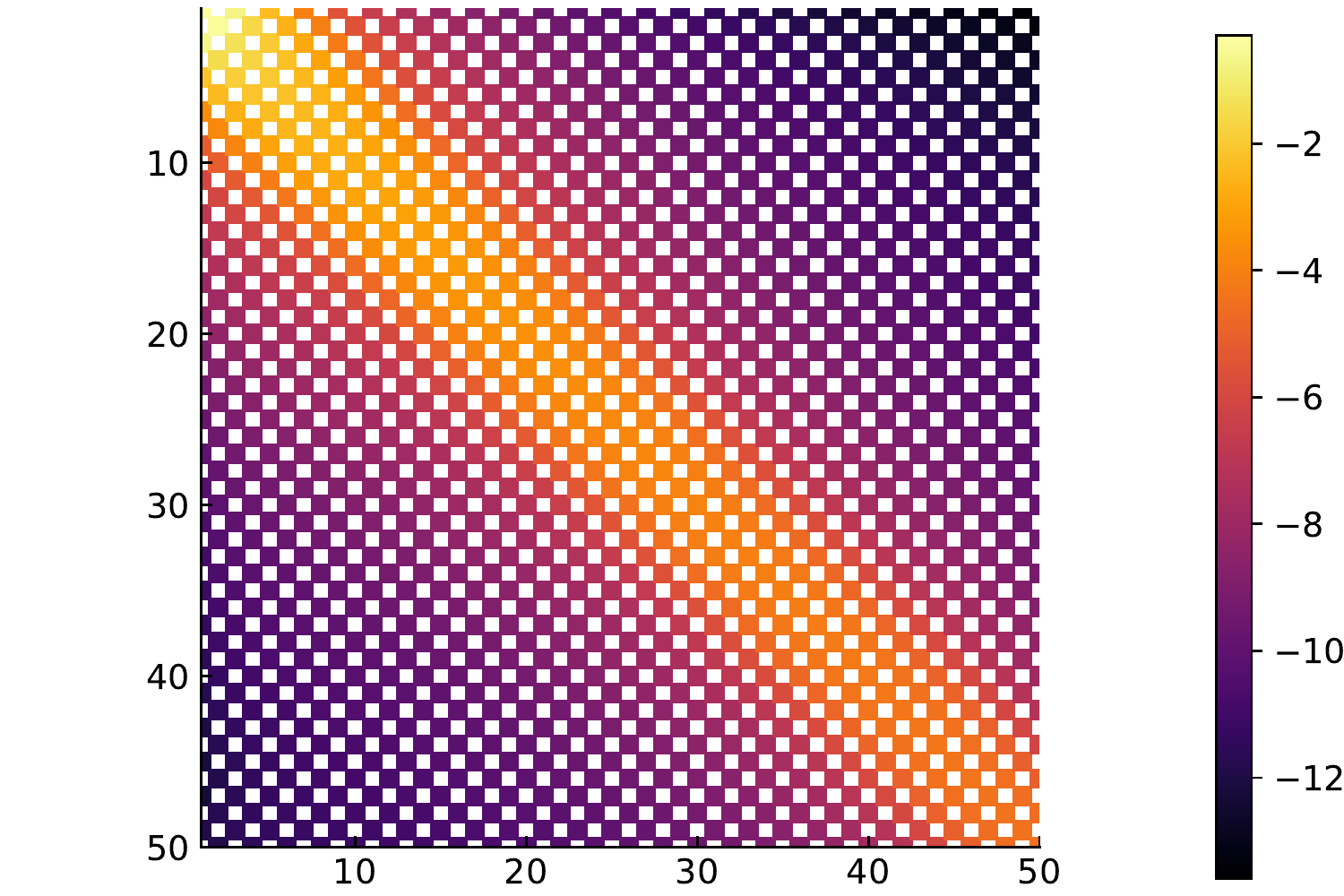} }}
    \caption{Spy plots of $\beta$-operators for the example pair $(\alpha, \beta) = (2.5, 1.5)$ in different $\alpha$-banding bases. Note that as the $\alpha$-banding $\lambda$ increases in value the bandwidth of the $\alpha$ operator grows, so that there is a trade-off in the quality of the banded approximation of the $\beta$-operator.  The legend is logarithmic, as in Figure \ref{fig:spy_approxbanded}.}%
    \label{fig:spy_approxbanded_tradeoff}%
\end{figure}
\section{Description of the method}\label{sec:method}
In this section we describe the combination of spectral and optimization methods which allows us to solve equilibrium measure problems with power law kernels. The operators are generated using the recurrence relationship we proved above and have sparsity or approximate sparsity structure as described in the previous section. We address two problems which require different solution approaches: The first kind of problem we discuss involves a single interaction type and an external potential. The second and primary problem of interest in this paper is the case of attractive-repulsive interaction terms with vanishing external potential.
Throughout all of this section describing numerical methods for different types of equilibrium problems with power law kernels we will be looking for a positive measure $\mathrm{d} \rho = \rho (x) \mathrm{d} x$ minimizing a certain energy expression $E$ with given mass $M$.
\subsection{Interactions with external potentials}\label{sec:methodpotential}
In this section we describe a method to find positive measure $\mathrm{d} \rho = \rho (x) \mathrm{d} x$ along with its assumed single interval support $\mathrm{supp}(\rho)=[a,b]$ which minimizes the energy expression
\begin{equation*}
E = \frac{1}{\alpha} \int_a^b |x-y|^\alpha \rho (y) \mathrm{d} y + V(x),
\end{equation*}
where $V(x)$ describes a sufficiently smooth external potential. To treat this problem with ultraspherical polynomials we first transform the variabes
\begin{align}\label{eq:radiustransformations1}
x &= \tfrac{b+a}{2}+\tfrac{b-a}{2}s\\
y &= \tfrac{b+a}{2}+\tfrac{b-a}{2}t \label{eq:radiustransformations2}
\end{align}
in order to move to the integration to the interval $(-1,1)$, yielding
\begin{equation*}
E = \frac{1}{\alpha} \left(\frac{b-a}{2}\right)^{\alpha+1} \int_{-1}^1 |s-t|^\alpha \tilde{\rho}(t) \mathrm{d} t + \tilde{V}(s),
\end{equation*}
where we have introduced the shorthands 
\begin{align*}
\tilde{\rho}(t) &= \rho\left(\tfrac{b+a}{2}+\tfrac{b-a}{2}t\right),\\ 
\tilde{V}(s) &= V\left(\tfrac{b+a}{2}+\tfrac{b-a}{2}s\right).
\end{align*}
We can rewrite this as the linear problem
\begin{align*}
E-\tilde{V} = \tfrac{1}{\alpha} \left(\tfrac{b-a}{2}\right)^{\alpha+1} Q^\alpha \tilde{\varrho},
\end{align*}
where the measure is assumed to be expanded in the weighted $\mathbf{C}^{(\lambda_\alpha)}$ basis, i.e.:
\begin{equation*}
\tilde{\rho}(t) = (1-t^2)^{\lambda_\alpha-\frac{1}{2}} \mathbf{C}^{(\lambda_\alpha)}(t){}^{\mathsf{T}}  \tilde{\varrho},
\end{equation*}
and the $\lambda_\alpha$ parameter is selected on the basis of the discussion and results in Section \ref{sec:bandedpowers} as a function of $\alpha$, yielding banded operators. As both $E$ and $\tilde{\varrho}$ are unknown for a given pair $(a,b)$, we differentiate once to obtain
\begin{align*}
-D\tilde{V} = \tfrac{1}{\alpha} \left(\tfrac{b-a}{2}\right)^{\alpha+1} D Q^\alpha \tilde{\varrho},
\end{align*}
which may be used to determine $\tilde{\varrho}$ up to the constant term (in the $n=0$ coefficient) for which information may have been lost in the differentiation. Via the known property of ultraspherical polynomials in \eqref{eq:n0integrationcondition} and the mass condition
\begin{align*}
\frac{b-a}{2} \int_{-1}^1 \tilde{\rho} (t) \mathrm{d} t = M.
\end{align*}
we can fix this $n=0$ term which allows us to obtain a measure given the support interval $[a,b]$. Using this measure computing its energy is a simple matter of using the already generated operators and evaluating at a point in the support, where the energy is constant by construction. As we can compute a unique corresponding measure for a given interval support $[a,b]$ and straightforwardly compute its corresponding energy we can thus reduce the na\"\i vely required intractable optimization over a measure space to an optimization over the interval boundary variables $a$ and $b$. Furthermore, all of the involved operators for the problem described in this section are exactly banded, resulting in high computational efficiency. For symmetric external potentials the problem can be reduced to a single variable optimization over the radius $R$ of the support $[a,b]=[-R,R]$.\\
The linear system we end up with using the above direct approach to the solution of the equilibrium measure problem is a Fredholm integral equation of first kind. These problems are generally ill-posed with numerical approximations of the inverse operator resulting in blow-up as the order $n$ increases, leading to stability concerns. 
To counteract this, one approach is to use Tikhonov regularization \cite{neggal_projected_2016,colton_inverse_2013,nair_linear_2009,tikhonov1963regularization,tikhonov1963solution,phillips1962technique} to instead solve an adjacent Fredholm integral equation of second kind with far better stability properties whose solution converges to the desired one in some well-defined sense -- we address this in Section \ref{sec:analysis}. Numerical comparisons of stability and convergence are presented in Section \ref{sec:numericalexamples}.
\subsection{Attractive-repulsive interactions with single interval support}
In this section we aim to find positive measure $\mathrm{d} \rho = \rho (x) \mathrm{d} x$ along with its assumed single interval support $\mathrm{supp}(\rho)=[a,b]$ which minimizes the energy expression
\begin{equation*}
E = \frac{1}{\alpha} \int_a^b |x-y|^\alpha \rho (y) \mathrm{d} y - \frac{1}{\beta} \int_a^b |x-y|^\beta \rho(y) \mathrm{d} y.
\end{equation*}
While in general an external potential may be added in similar fashion to Section \ref{sec:methodpotential} we omit discussion of this here as it leads to a similar approach as seen in the previous section and having no external potential presents unique challenges in itself. We will also assume the support to be a single interval $[a,b]$, a restriction we will dispose of in the following section in which we will address two interval methods. As before we begin by transforming the above energy expression to the interval $(-1,1)$ via (\ref{eq:radiustransformations1}--\ref{eq:radiustransformations2}) yielding:
\begin{equation*}
E = \tfrac{1}{\alpha} \left(\tfrac{b-a}{2}\right)^{\alpha+1} \int_{-1}^1 |s-t|^\alpha \tilde{\rho}(t) \mathrm{d} t - \tfrac{1}{\beta} \left(\tfrac{b-a}{2}\right)^{\beta+1} \int_{-1}^1  |s-t|^\beta \tilde{\rho}(t) \mathrm{d} t,
\end{equation*}
where we once again use the shorthand $\tilde{\rho}(t) = \rho(\tfrac{b+a}{2}+\tfrac{b-a}{2}t)$. Dividing through the constant energy expression and rewriting it from the point of view of operators acting on ultraspherical polynomial basis expansions with coefficients $\tilde{\varrho}$ we obtain
\begin{align*}
1 = \tfrac{1}{\alpha} \left(\tfrac{b-a}{2}\right)^{\alpha+1} Q^\alpha{} \tfrac{\tilde{\varrho}}{E} - \tfrac{1}{\beta} \left(\tfrac{b-a}{2}\right)^{\beta+1} Q^\beta{} \tfrac{\tilde{\varrho}}{E}.
\end{align*}
The implicit weight parameter $\lambda$ here is suitably chosen in accordance with the results in section \ref{sec:recurrence} to yield a banded $Q$ operator for one of the parameters $\alpha$ or $\beta$ and an approximate banded operator for the second parameter, placing them in the same $\lambda$-basis. The above expression allows us to compute an ultraspherical approximation of the expression $\frac{\tilde{\varrho}}{E}$ for a given pair $(a,b)$ which as for the simpler potential case above may then be used to find the desired equilibrium measure $\rho(x)$ via the mass condition in \eqref{eq:masscondition}. No differentiation is required in the absence of external potentials. Using the thus obtained $\rho(x)$ it is a straightforward application of the already computed operators to evaluate for the energy $E$ and therefore, as before, this approach reduces what na\"\i vely would appear as an optimization problem over a difficult to work with measure space to an easy to work with minimization with respect to only two variables---the support interval parameters $a$ and $b$. In fact, as this problem is translation symmetric in the absence of an external potential $V(x)$, this particular problem is reducible to an optimization over a single variable, as either the left or right boundary may be fixed arbitrarily, e.g. by setting $a=-1$.\\
This direct approach also ends in having to solve an ill-posed Fredholm integral equation of first kind with compact operator and thus potentially unstable numerical inversion. Regularization of this problem is addressed in Section \ref{sec:analysis}.

\section{Approach for measures with two interval support}\label{sec:twointerval}
\subsection{Two interval support recurrence relationship}
So far we assumed at multiple points that the equilibrium measure has single interval $[a,b]$. As we will see in the numerical experiments section and was discussed analytically in \cite{carrillo_explicit_2016}, this assumption is no longer true in certain high parameter ranges, as the obtained measures become negative around the origin and are thus no longer admissible. In this section we discuss how this assumption may be dropped to obtain a generalized approach for measure supports consisting of multiple disjoint intervals. For the sake of brevity, we will not discuss adding external potentials to the multiple interval case and furthermore specifically focus on the case of the support consisting of two disjoint intervals as opposed to an arbitrary number. We leave it as a remark that extensions including potentials and more than two intervals are possible for well-behaved external potentials by lifting the assumption of rotational symmetry while following a similar chain of arguments in a piecewise polynomial space. The two interval case is particularly of interest for the discussion of gap formation in high parameter ranges mentioned above.\\
We begin by revisiting what form the problem takes given the assumption of a two interval support. As no external potential will be present, the solution will exhibit translational invariance on $\mathbb{R}$, so we can without loss of generality assume that the two intervals will be centered around the origin, i.e. we let $$\mathrm{supp}(\rho) = [-b,-a] \cup [a,b].$$ As we are for now primarily interested in the structure and recurrence relations of the operator and adding a second operator is straightforward, we only explore the following single power terms instead of the full governing equation:
\begin{align*}
\int_{-b}^{-a} |x-y|^\alpha \rho_l(y) \mathrm{d}y + \int_{a}^{b} |x-y|^\alpha \rho_r(y) \mathrm{d}y = E(x),
\end{align*}
where $x \in [-b,-a] \cup [a,b]$ and
\begin{align*}
\rho(y) = \begin{cases} \rho_l(y), \qquad \text{if} \quad y \in [-b,-a]\\
\rho_r(y),\qquad \text{if} \quad y \in [a,b].
\end{cases}
\end{align*}
By rotational symmetry we have $\rho_l(-y) = \rho_r(y)$ which we can use to rewrite the integral equation with the substitution $y \rightarrow -y$:
\begin{align}\label{eq:symmetricmultiint}
\int_{a}^{b} |x+y|^\alpha \rho_r(y) \mathrm{d}y + \int_{a}^{b} |x-y|^\alpha \rho_r(y) \mathrm{d}y = E(x),
\end{align}
We are only interested in $E(x)$ where $x \in [-b,-a] \cup [a,b]$, in other words we are interested in $E([-b,-a])$ and $E([a,b])$. As can be easily verified from  \eqref{eq:symmetricmultiint}, the rotational symmetry of the problem leads to a symmetry in $x$ as well, namely the problem is completely invariant under the transformation $x\rightarrow -x$. As a consequence, no information is gained in the rotationally symmetric case from considering both $E([-b,-a])$ and $E([a,b])$, and we may instead freely choose to only consider $E([a,b])$ and thus only values $x\in[a,b]$ to fully describe the problem. A notable difference to the one interval case is the appearance of the kernel $|x+y|^\alpha = |y-(-x)|^\alpha$ in one of the terms, which is equivalent to evaluating the familiar power law kernel $|y-x|^\alpha$ for $x\in [-b,-a]$. An intuition for this term stemming from the interpretation of interacting particles is that this term computes the long range influence of the particles in $[-b,-a]$ acting on those in $[a,b]$.\\
We already know how to generate the operator in the second term in  \eqref{eq:symmetricmultiint} after a transformation of variables
\begin{align*}
y &= \left(\tfrac{b+a}{2}\right) + \left(\tfrac{b-a}{2}\right)t,\\
x &= \left(\tfrac{b+a}{2}\right) + \left(\tfrac{b-a}{2}\right)s,
\end{align*}
to the interval $[-1,1]$. This leads to
\begin{align}\label{eq:symmetricmultiint2}
\left(\tfrac{b-a}{2}\right)^{\alpha+1}\left(\int_{-1}^{1} \left\lvert-\left(\tfrac{2(b+a)}{b-a}+s\right)-t\right\rvert ^\alpha \rho_r(t) \mathrm{d}t + \int_{-1}^{1} |s-t|^\alpha \rho_r(t) \mathrm{d}y \right) = E(s),
\end{align}
with $s\in[-1,1]$. In this form we see that the second operator may be generated exactly as before using Theorem \ref{theorem:twotermrecFullop}. Furthermore, as it makes no assumptions on $x$, we see that Theorem \ref{theorem:twotermrecFullop} can also be used to generate the first operator if the $n=0$ and $n=1$ cases are known, by replacing all occurrences of $x$ with $\frac{2(b+a)}{a-b}-s$. 
\subsection{The $n=0$ and $n=1$ cases for multiple intervals}
The previous section covered how to generate the operators assuming the $n=0$ and $n=1$ cases are available. When $s \in [-1,1]$ we have $\frac{2(b+a)}{a-b}-s \in \left[\frac{3b+a}{a-b},\frac{3a+b}{a-b}\right]$, so the analytic solutions in Section \ref{sec:structure} cannot be used. Instead, we require the following solutions with $w(y)=(1-y^2)^{\lambda-\frac{1}{2}}$ which extend those results to the real line:
\begin{align*}
&\int_{-1}^{1} |x-y|^\alpha w(y) \mathrm{d}y =\tfrac{\sqrt{\pi}\Gamma(\lambda+\frac{1}{2})|x|^{\alpha}{}_2 F_1\left(\frac{1-\alpha}{2},-\frac{\alpha}{2};1+\lambda;\frac{1}{x^2}\right)}{\Gamma(1+\lambda)}, \text{ if } |x|>1.\\
&\int_{-1}^{1} |x-y|^\alpha w(y) C_1^{(\lambda)}(x) \mathrm{d}y =\begin{cases}  \tfrac{\sqrt{\pi}\alpha \lambda \Gamma(\lambda+\frac{1}{2})(-x)^{\alpha-1}{}_2 F_1\left(\frac{1-\alpha}{2},1-\frac{\alpha}{2};2+\lambda;\frac{1}{x^2}\right)}{\Gamma(2+\lambda)}, \text{if } x<-1,\\
 \tfrac{-\sqrt{\pi}\alpha \lambda \Gamma(\lambda+\frac{1}{2})x^{\alpha-1}{}_2 F_1\left(\frac{1-\alpha}{2},1-\frac{\alpha}{2};2+\lambda;\frac{1}{x^2}\right)}{\Gamma(2+\lambda)}, \text{ if } x>1.
\end{cases}
\end{align*}
We recall that the Gauss hypergeometric function $_2 F_1$ is a finite polynomial in its last argument when its first or second parameter is a negative integer. For the above $n=0$ and $n=1$ cases this corresponds to $\alpha \in \mathbb{N}_0$, meaning that generally the appearing hypergeometric functions will not be finite polynomials and must thus instead be approximated. As it is often difficult to divine the properties of a particular hypergeometric function from its standard form, what we seek is a systematic representation of the $n=0$ and $n=1$ solutions in terms of more familiar functions whose properties with varying parameters are more readily understood. The characteristic feature of the hypergeometric functions appearing in the $n=0$ and $n=1$ solution is that their first and second parameter differ by exactly $\frac{1}{2}$. Such hypergeometric functions have long-standing known connections to the theory of associated Legendre functions, going back to \cite{gormley_generalization_1934}, see also \cite[3.6.2]{bateman_higher_1981}. Explicitly we have \cite[07.23.03.0111.01;07.23.03.0110.01]{wolframfunctions_2020}:
\begin{align*}
\, _2F_1\left(\gamma,\gamma+\tfrac{1}{2};c;z\right)=\tfrac{2^{c-\frac{1}{2}} e^{i \pi  \left(-2 \gamma+c-\frac{1}{2}\right)} z^{\frac{1}{4} (1-2 c)} \Gamma (c) (1-z)^{\frac{1}{4} (2 c-1)-\gamma} Q_{c-\frac{3}{2}}^{2 \gamma-c+\frac{1}{2}}\left(\frac{1}{\sqrt{z}}\right)}{\sqrt{\pi } \Gamma (2 \gamma)},
\end{align*}
with $z=\frac{1}{x^2}\notin(1,\infty)$ and
\begin{align*}
\, _2F_1\left(\gamma,\gamma+\tfrac{1}{2};c;z\right)=2^{c-1} z^{\frac{1-c}{2}} \Gamma (c) (1-z)^{\frac{c-1}{2}-\gamma} P_{2 \gamma-c}^{1-c}\left(\tfrac{1}{\sqrt{1-z}}\right),
\end{align*}
where $z=\frac{1}{x^2}\notin(-\infty,0)$, where we remind ourselves that ${}_2F_1$ is symmetric with respect to exchange of the first two parameters. In the above, $Q^\mu_\nu(x)$ denotes the associated Legendre function of second kind and $P^\mu_\nu(x)$ the associated Legendre functions of first kind. The properties of associated Legendre functions are well understood, see e.g. \cite[14.1--14.20]{nist_2018} and \cite{bateman_higher_1981} and software packages to compute them to arbitrary numerical precision are widely available. In practice, our implementation of the described method computes polynomial approximations of the hypergeometric function implementation in the HypergeometricFunctions.jl package \cite{noauthor_juliamath/hypergeometricfunctions.jl_2020}, which uses the fast and reliable algorithms described in \cite{michel_fast_2008,pearson_numerical_2017}.
\section{Convergence and stability of the method} \label{sec:analysis}
Two aspects must be considered in the discussion of convergence for this method. First, the convergence of the ultraspherical spectral method itself, which is used to find the measure for a given interval of support $[a,b]$. Second, the optimization used to determine the energy minimum and thus the boundaries of said support, where convergence properties are inherited from the optimization method of choice---in our case primarily a Newton method with Hager-Zhang type linesearch \cite{hager_algorithm_2006} as implemented in the Optim.jl package \cite{mogensen_optim_2018}. As any optimization method with sufficiently good convergence properties may be used and the energy functions are well-behaved in the neighbourhood of the minimizers, the focus of this section is placed on the spectral method aspect. This part consists of finding solutions to
\begin{equation}\label{eq:solveanalytic}
E = \frac{1}{\alpha} \int_a^b |x-y|^\alpha \rho (y) \mathrm{d} y - \frac{1}{\beta} \int_a^b |x-y|^\beta \rho(y) \mathrm{d} y,
\end{equation}
for given data tuple $(a,b,\alpha,\beta)$. We will assume throughout this section that we are intending to solve this problem in a region where unique solutions exist.\\
For the case of vanishing external potential as in \eqref{eq:solveanalytic} and with $\alpha, \beta > 0$, which are the primary cases of interest to our applications, the problem to solve has the form of a Fredholm intergral equation of first kind with convolution kernel $K(x,y) \in L^2(a,b)$. Generally speaking, Fredholm integral equations of first kind are ill-posed problems, as on the typically considered Banach spaces such operators are Hilbert-Schmidt and compact and thus not invertible and attempts to compute numerical approximations to their inverse become unbounded. Even in scenarios where the thus obtained solutions may be somewhat sensible (as is the case for the method presented in this paper, established by numerical experimentation) there is a high risk of instability with respect to the order of approximation as well as potential noise in the initial data, c.f. discussions in \cite{neggal_projected_2016}. As the initial data for this problem include the boundary of the support which is obtained via an optimization method this deserves attention.\\
Given some well-behavedness assumptions, however, integral equations of first kind may be transformed into a well-posed problem. There are two primary ways to accomplish this: Modifications to the considered Banach spaces or modifications to the ill-posed operator. The first option, which we will not pursue here, typically considers the problem on weighted Banach spaces instead. For a space with appropriate weight, the operator may no longer be compact and thus an inverse may exist. Such approaches largely rely on being able to tightly characterize the compactness of the specific operators, an example of using such an approach for convergence proofs of sparse spectral methods for \emph{Volterra} integral equations of first kind may be found in \cite{gutleb_sparse_2019}.\\
A second common approach to transforming ill-posed first kind integral equations into well-posed ones, largely going back to Tikhonov and Philips \cite{tikhonov1963regularization,tikhonov1963solution,phillips1962technique}, is to use regularization methods to solve well-posed second kind equations which are adjacent to the ill-posed problem of interest and then consider an appropriate limit, see e.g. discussions in \cite{neggal_projected_2016,nair_linear_2009,colton_inverse_2013} and the references therein. The Tikhonov regularization approach, which we do not detail here, transforms our problem in \eqref{eq:solveanalytic} into the second kind Fredholm integral equation
\begin{equation}\label{eq:solveanalytic2}
(s\mathcal{I} + \mathcal{F}^*\mathcal{F}) \rho_s = \mathcal{F}^*E,
\end{equation}
with shorthand notation $\mathcal{F} = Q^\alpha_\lambda - Q^\beta_\lambda$. This Tikhonov regularization of the original problem is then solved via a finite-dimensional approximation via the proposed spectral method, resulting in a spectral \emph{projected} Tikhonov regularization method \cite{neggal_projected_2016}. In the numerical experiments section we present a comparison of the na\"\i ve inversion and the projected Tikhonov regularization approach, where the stability advantages of the latter become clear. The stability of the Tikhonov regularization method with respect to perturbations in the support boundary data $(a,b)$ are likewise explored in the numerical experiments. This approach raises the crucial question of how the error of these adjacent solutions depend on the parameter $s$ and what an appropriate choice for $s$ is in the context of such a numerical scheme. For numerical purposes the trade-off between more stability for large $s$ and higher accuracy for small $s$ has received much attention, see references listed above, and will thus not be repeated here. For the purpose of the method introduced in this paper, the error incurred by using the $s$-regularized $n$-th order approximation denoted $\rho_{s,n}$ can be split up into an error due to the Tikhonov regularization $\|\rho_{s} - \rho  \|$ and an error due to the spectral method $\| \rho_{s,n} - \rho_{s}\|$ using the triangle inequality
\begin{align*}
\|\rho_{s,n} - \rho\| =  \| \rho_{s,n} - \rho_{s} + \rho_{s} - \rho  \| \leq \| \rho_{s,n} - \rho_{s}\|+  \|\rho_{s} - \rho  \|,
\end{align*}
The Tikhonov regularization provides the stability for what would otherwise be an ill-posed problem while the use of spectral methods keeps the convergence rate comparatively high with other numerical approaches.
\section{Numerical experiments and verification} \label{sec:numericalexamples}
In this section we present numerical experiments using the methods introduced above for two primary purposes: The first set explores the advantages of using a Tikhonov regularization of the problem instead of a direct linear solve approach. The remaining sections compare numerical solutions obtained via the spectral method to analytically known solutions as well as solutions computed through alternative numerical means where possible. The method is found to be highly efficient and results in reliably small errors even near boundary singularities.\\
The figures in these sections show differences and errors as absolute values for both absolute and relative errors. With the exception of the first set of experiments, where we explicitly compare the Tikhonov regularization to a direct approach, the Tikhonov regularized method is used throughout this section.
\subsection{Stability and error comparison for Tikhonov regularization}
We have discussed two options to use the introduced spectral method to find a measure associated with a given compact support -- a direct approach and one using Tikhonov regularization. In this section we demonstrate the superiority of the Tikhonov regularization approach in terms of errors and stability as $n$ increases for examples where the analytic measures are known.\\
We compare stability of the obtained measures as the order of approximation $n$ increases for the following problem with analytically known solutions: Let $M=1$, $(\alpha_1,\beta_1 ) = (\frac{7}{3},2)$ and $(\alpha_2,\beta_2) = (2,-\frac{1}{5})$. We seek minimizers to the following energy expression
\begin{equation*}
E = \frac{1}{\alpha_i} \int_a^b |x-y|^{\alpha_i}\rho (y) \mathrm{d} y - \frac{1}{\beta_i} \int_a^b |x-y|^{\beta_i} \rho(y) \mathrm{d} y,
\end{equation*}
with $i \in {1,2}$. For $(\alpha_1,\beta_1)$ this has known solutions \cite{carrillo_explicit_2016,carrillo_radial_2021}:
\begin{align}\label{eq:analyticsolutionset11}
&b = -a = \left[ - \tfrac{\cos\left(\frac{\alpha \pi}{2}\right)}{\pi (\alpha-1)} B\left(\tfrac{1}{2},\tfrac{3-\alpha}{2}\right) \right]^{\frac{1}{\alpha-2}},\\\label{eq:analyticsolutionset12}
&\rho (x) = -\tfrac{M}{\alpha-1} \tfrac{\cos\left(\frac{\alpha \pi}{2} \right)}{\pi} \left(b^2-x^2 \right)^{\frac{1-\alpha}{2}},
\end{align}
where $B(\cdot,\cdot)$ denotes the Beta function. For $(\alpha_2,\beta_2)$ this has known analytic solutions \cite{carrillo_explicit_2016,lopes_uniqueness_2019,carrillo_radial_2021}:
\begin{align}\label{eq:analyticsolutionset21}
&b = -a = \left[ \tfrac{\cos \left( \frac{(2-\beta)\pi}{2}\right)}{(\beta-1)\pi} B\left( \tfrac{1}{2},\tfrac{3-\beta}{2} \right) \right]^{\frac{1}{\beta-2}},\\\label{eq:analyticsolutionset22}
&\rho (x) = -\tfrac{M}{\alpha-1} \tfrac{\cos\left(\frac{\alpha \pi}{2} \right)}{\pi} \left(b^2-x^2 \right)^{\frac{1-\alpha}{2}}.
\end{align}
Figure \ref{fig:Tikhonovconverge} shows the behaviour of the absolute error as the order of approximation $n$ increases when the measure is computed with the analytically known support, demonstrating the improved stability of the Tikhonov regularization. Figure \ref{fig:TikhonovconvergePerturb} shows the behaviour of the absolute error at $n=50$ with respect to perturbations in the boundary data $(a,b)$ supplied for the support of the measure---we find that in general, for both the direct and the Tikhonov approach, perturbations in the support boundary linearly correlate with errors on the measure. This information is vital for the choice of optimization convergence requirements in the following sections but no optimizations were used for the results presented in this section.\\
Even for the stable Tikhonov regularized problem, increasing the order of approximation does not yield a significant improvement in the obtained error here, as the method we propose converges to a good approximation even for very low $n$ if the appropriate weighted basis is chosen. The exceptional convergence in this section, however, is largely owed to the special properties of the solution in the cases with analytically known solutions, as their boundary singularities are exactly resolved in our basis choice. In Section \ref{sec:convergencebandwidth} we will discuss convergence of the solution coefficients for more general examples where the singularities are by necessity approximated. As for a stable spectral method solution the error will remain approximately constant once convergence has been reached and certainly should not \emph{increase}, criteria such as these could be used to produce an automatically converging algorithm which also automatically determines a sensible choice of Tikhonov parameter $s$ for which such convergent behaviour is observed.
\subsection{Comparisons with analytically derived steady state solutions}
As mentioned above, solutions for certain parameter ranges of $\alpha$ and $\beta$ for the attractive-repulsive power law kernel problem, where one of the parameters is an even integer, were derived in \cite{carrillo_explicit_2016}. In this section we investigate the error for numerical solutions obtained via the introduced spectral method in combination with optimization without using any a priori knowledge about the radius of the support.\\
In Figure \ref{fig:set1-1_errors} we plot the numerically obtained equilibrium measure and present absolute and relative errors in the measure normalized to the interval $(-1,1)$ determined with the above-discussed numerical method for the example case 
\begin{align*}
(\alpha_1,\quad \beta_1,\quad M_1) = \left(2,\quad \tfrac{3}{2},\quad 1\right).
\end{align*}
This allows comparing errors in the measure and errors in the radius somewhat separately. The analytic solution is the one seen in (\ref{eq:analyticsolutionset21}-\ref{eq:analyticsolutionset22}) above. A simple Newton method with Hager-Zhang linesearch found support radius with absolute error of order $10^{-12}$ compared to the analytically expected radius in (\ref{eq:analyticsolutionset22}).\\
In Figure \ref{fig:set1-2_errors} we plot the numerically obtained equilibrium measure and present absolute and relative errors in the measure normalized to the interval $(-1,1)$ for the example case
\begin{align*}
(\alpha_2,\quad \beta_2,\quad M_2) = \left(\tfrac{7}{3},\quad 2,\quad 3\right).
\end{align*}
The analytic solution is the one seen in (\ref{eq:analyticsolutionset11}-\ref{eq:analyticsolutionset12}) above. A univariate Newton method with Hager-Zhang linesearch found support radius with absolute error of order $10^{-13}$ compared to the analytically expected radius in (\ref{eq:analyticsolutionset12}). Note that the accuracy loss in this section compared to the previous section is explained by the fact that we now use an optimization method to compute the radius rather than using an analytically known one, compare Figure \ref{fig:TikhonovconvergePerturb}. Requiring a more precise result from the the numerical optimization increases the obtained precision. 
\subsection{Comparisons with other approaches in low parameter regimes}
We compare the above method with a more straightforward application of the recurrence relationships derived in Section \ref{sec:recurrence}, which uses a root-search algorithm over a space of approximated measures instead of optimizing over the boundary of the support. This method only works for $\alpha$ and $\beta$ for which the equilibrium measure is the unique measure which vanishes on the boundary of its support. The parameter ranges in which this less efficient method works can be inferred from result proved in \cite{carrillo_explicit_2016}. In such cases, we can use a root search algorithm method to find coefficients of a polynomial approximation with zeros on the boundary.\\
As the previous sections focused on results for attractive-repulsive problems without external potential, this section presents results for an attraction-only and a repulsion-only problem---respectively with a radially symmetric and asymmetric external potential. Naturally, a two variable support boundary based optimization method has to be used for asymmetric potentials instead of a univariate one. We use gradient free Nelder-Mead optimization for our support boundary optimization method in this section to showcase that the method of optimization used is up to the user. The alternative coefficient space method is performed using a simple Newton-type root search method which acts on as many coefficients as are chosen for the approximation but has higher computational complexity as this number is always larger than $2$ for sensible, non-trivial results. We seek minimizers to the following energy expression
\begin{equation*}
E = \frac{1}{\alpha_i} \int_a^b |x-y|^{\alpha_i} \rho (y) \mathrm{d} y + V_i(x),
\end{equation*}
for the following data:
\begin{align}\label{eq:example2-1}
\left(\alpha_1,\quad M_1,\quad V_1(x)\hspace{1mm}\right) &= \left(\tfrac{1}{2},\quad 1,\quad x^2\right),\\ \label{eq:example2-2}
\left(\alpha_2,\quad M_2,\quad  V_2(x)\hspace{1mm}\right) &= \left(-\tfrac{2}{3},\quad \tfrac{5}{3},\quad -x^4+\sin(x)\right). 
\end{align}
Figure  \ref{fig:set2numexp} shows the absolute difference between the equilibrium measures obtained using the two approaches outlined above. We also plot the asymmetric case equilibrium measure in Figure \ref{fig:set2numexp}, which demonstrates that radial symmetry in the measure, potential and support of the measure are not required for either of the two methods to work. Both approaches arrive at the same support boundary with deviations of order $10^{-13}$. As stated above, we stress that the root-search approach only works for the limited range of parameters for which $\rho(x)$ vanishes on the boundary but in said ranges provides a useful means for comparison.
\begin{figure}
     \subfloat[$(\alpha_1,\beta_1,s) = (\frac{7}{3},2, 10^{-13})$]
    {{ \centering \includegraphics[width=6.2cm]{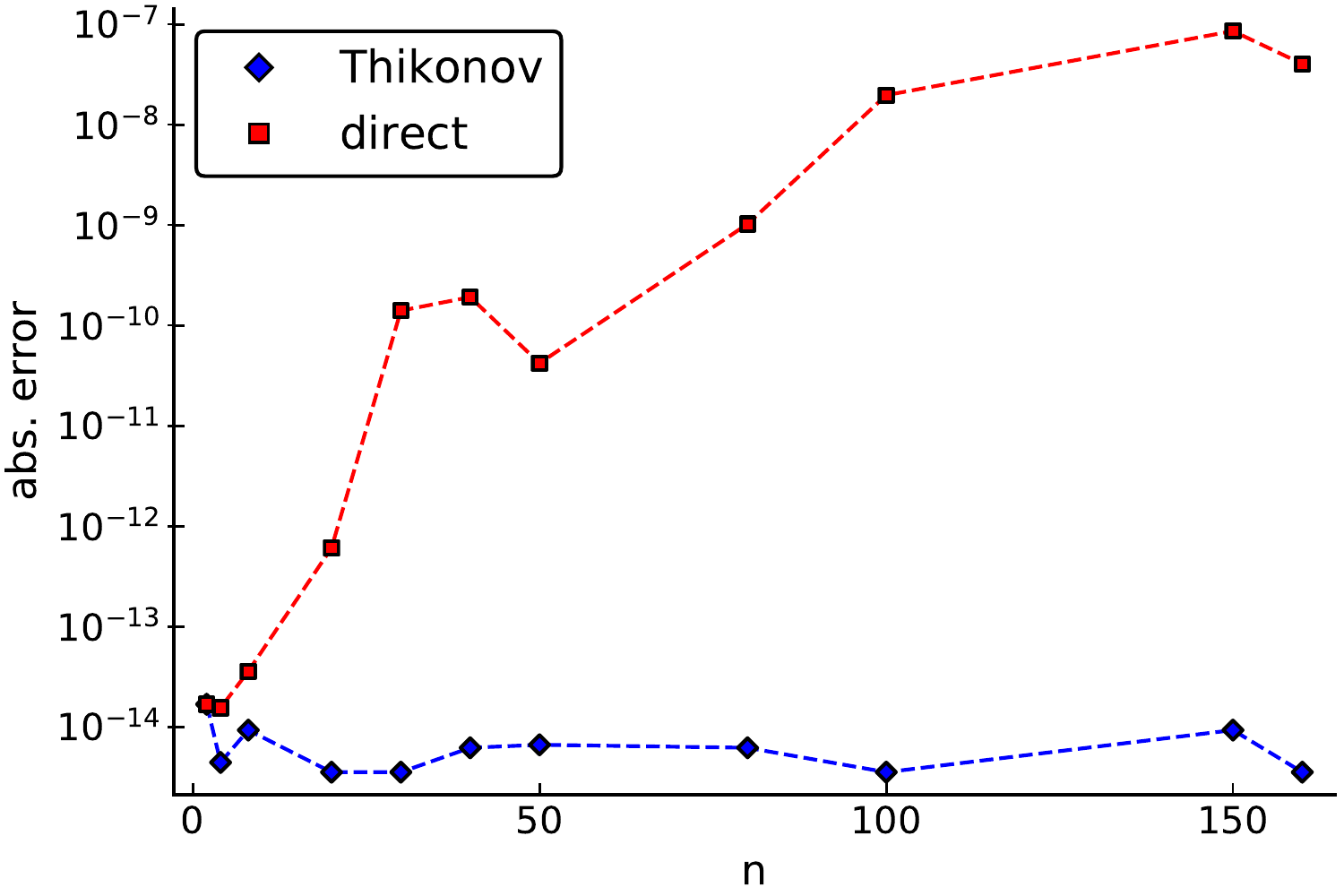} }}
         \subfloat[$(\alpha_2,\beta_2,s) = (2,-\frac{1}{5},10^{-13})$]
    {{ \centering \includegraphics[width=6.35cm]{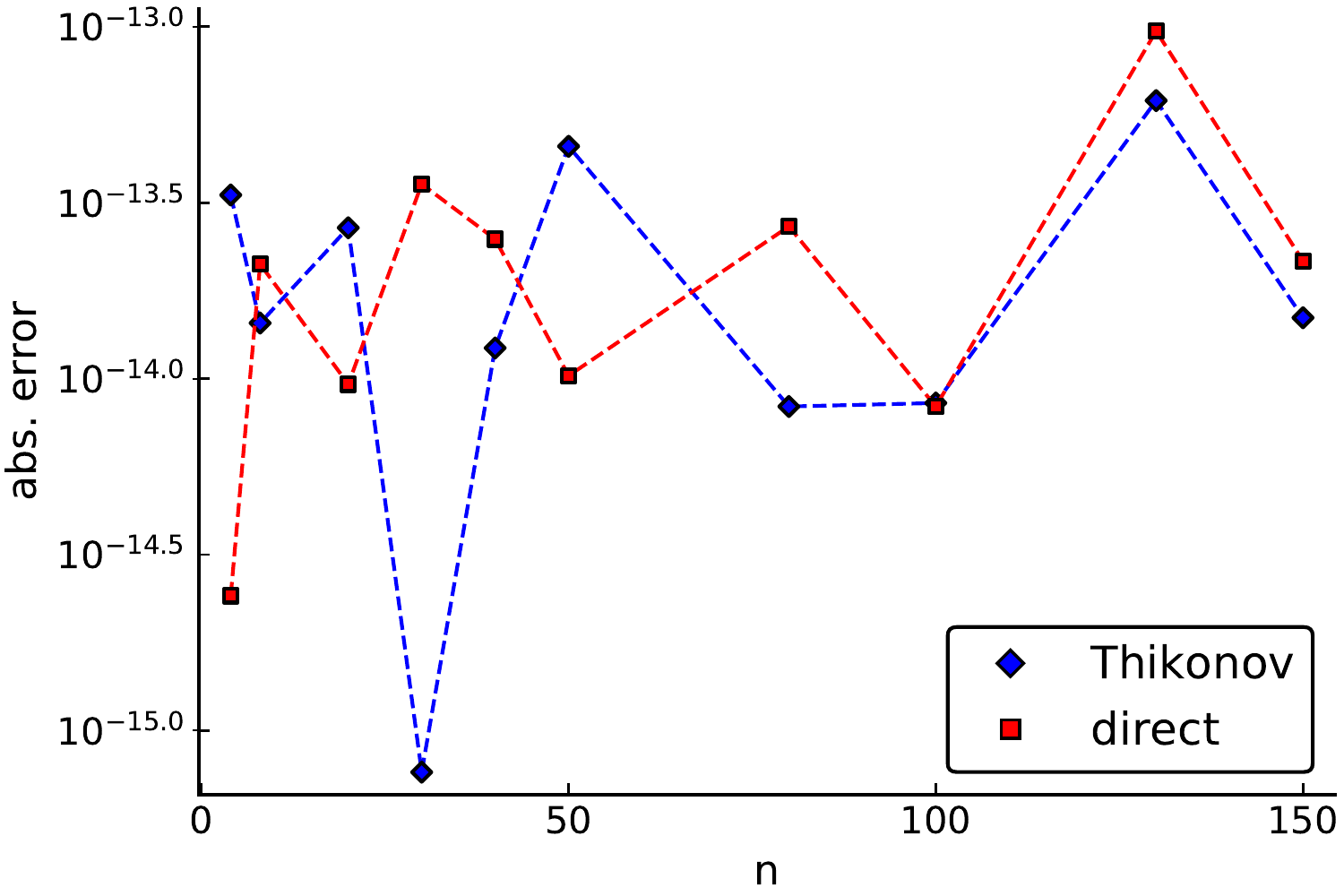} }}
    \caption{Semi-logarithmic convergence plot of absolute errors close to the boundaries compared to analytic solutions in (\ref{eq:analyticsolutionset11}-\ref{eq:analyticsolutionset22}). In (A) the direct method is unstable and grows with increasing order, while the error of the Tikhonov regularized solution shows the behaviour expected of a converged numerical solution. In (B) the original problem appears stable without regularization.}%
    \label{fig:Tikhonovconverge}%
\end{figure}
\begin{figure}
     \subfloat[$(\alpha_1,\beta_1,s) = (\frac{7}{3},2,10^{-13})$]
    {{ \centering \includegraphics[width=6.3cm]{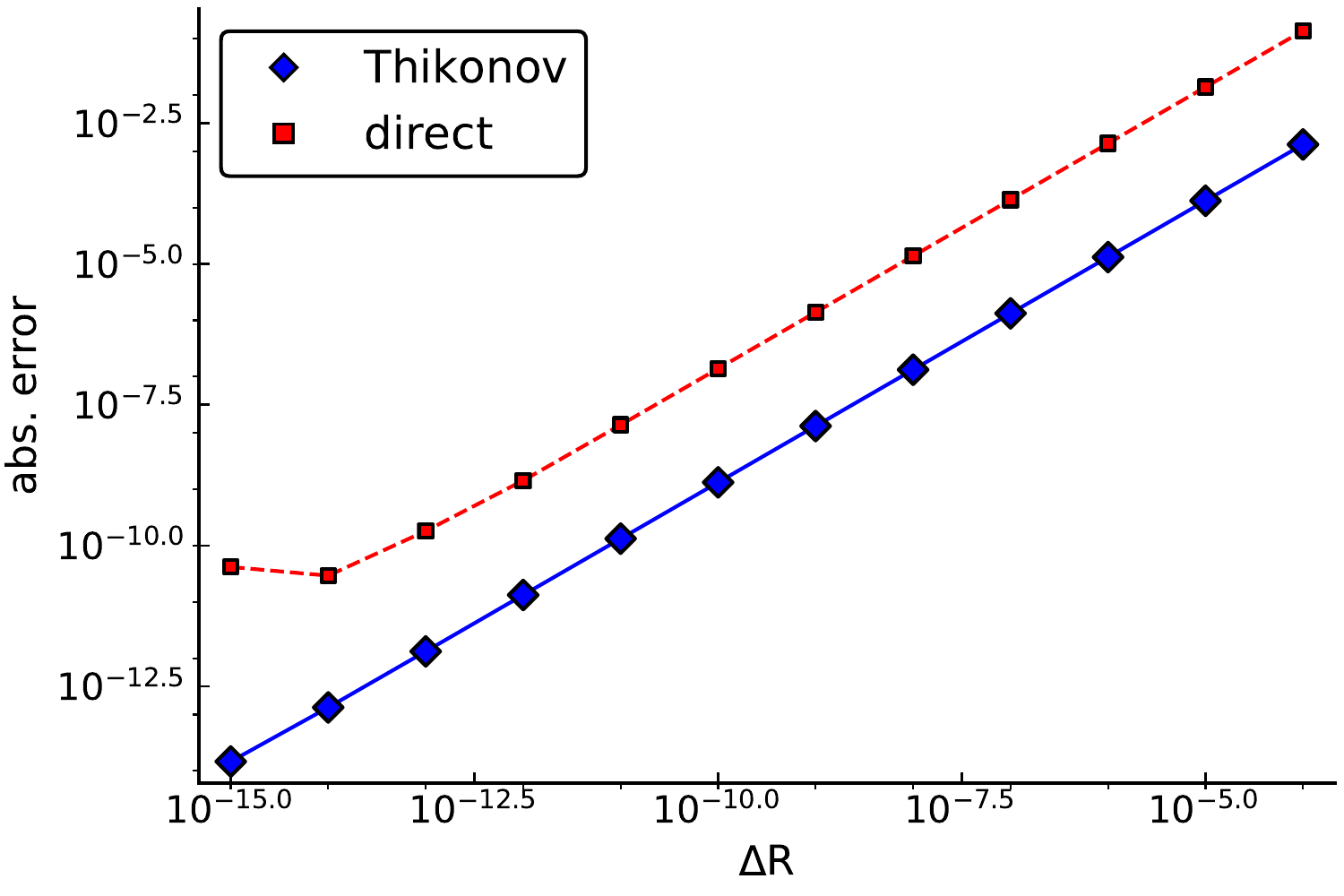} }}
         \subfloat[$(\alpha_2,\beta_2,s) = (2,-\frac{1}{5},10^{-13})$]
    {{ \centering \includegraphics[width=6.35cm]{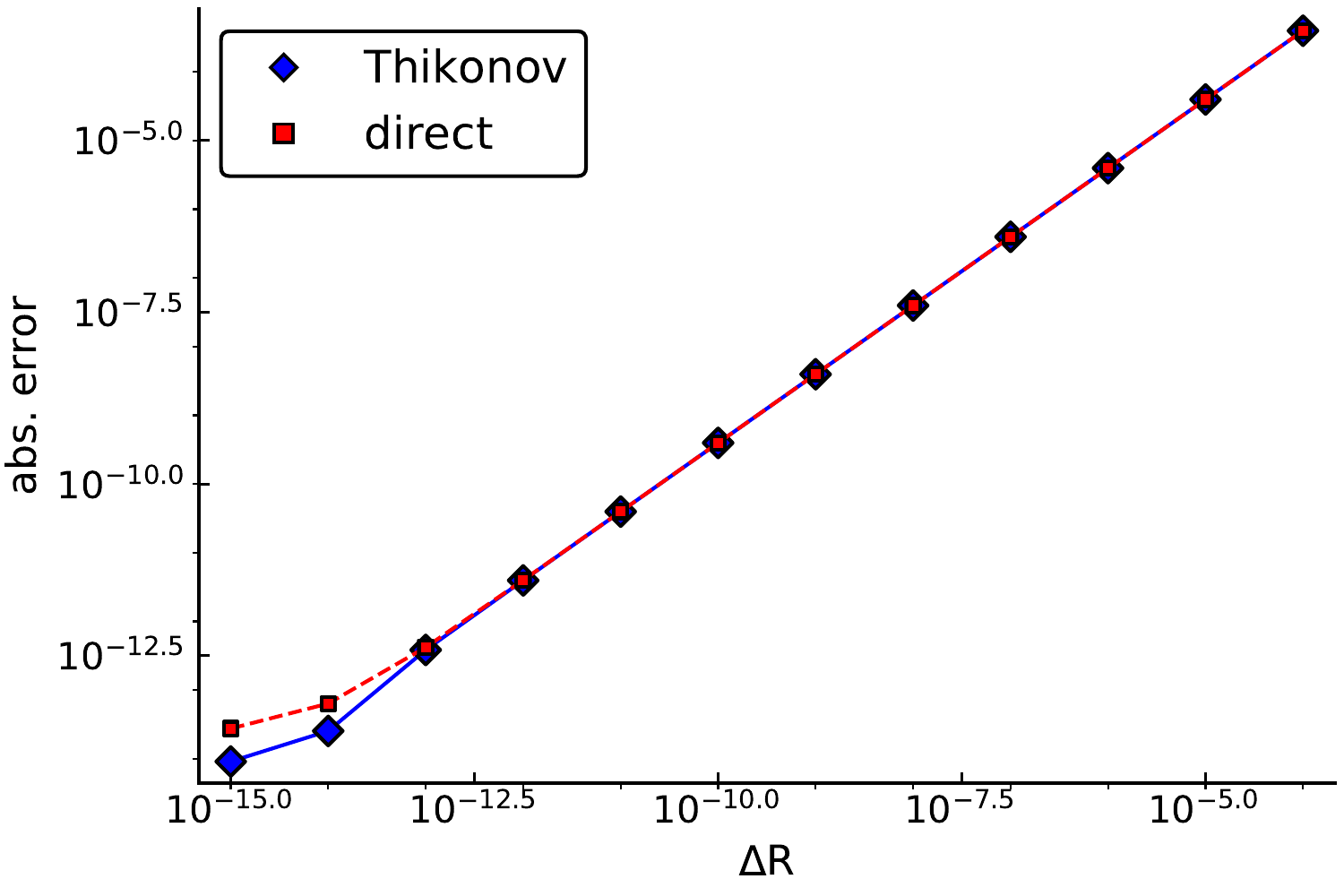} }}
    \caption{Absolute errors close to the boundaries for $n=50$ as the value of the analytic support radius $R$ is perturbed by $\Delta R$. To obtain a sensible comparison, both the analytic and numerically obtained measure were shifted to $(-1,1)$ before comparing.}%
    \label{fig:TikhonovconvergePerturb}%
\end{figure}
\begin{figure}[ht]
     \subfloat[]
    {{ \centering \includegraphics[width=4.2cm]{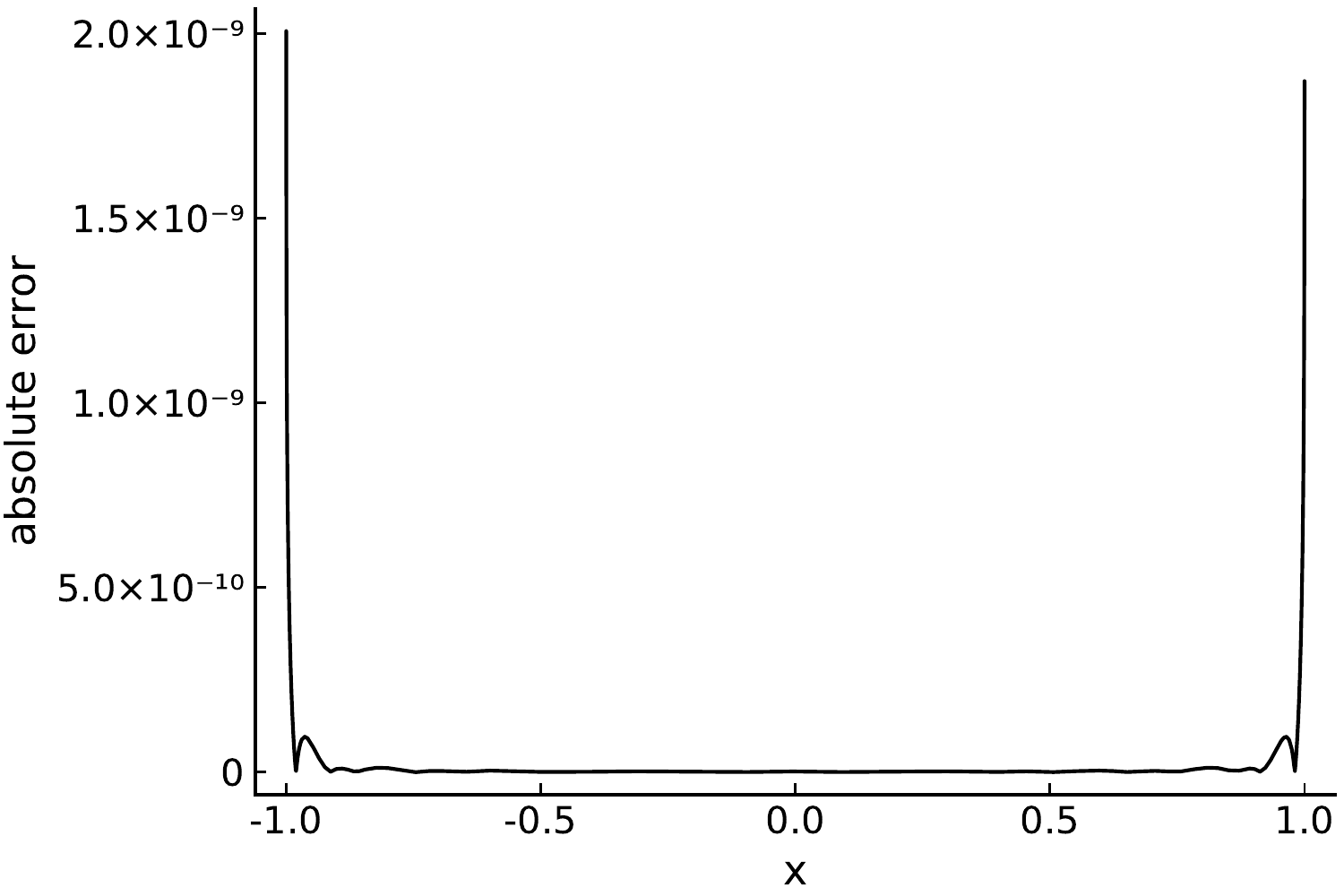} }}
     \subfloat[]
    {{ \centering \includegraphics[width=4.2cm]{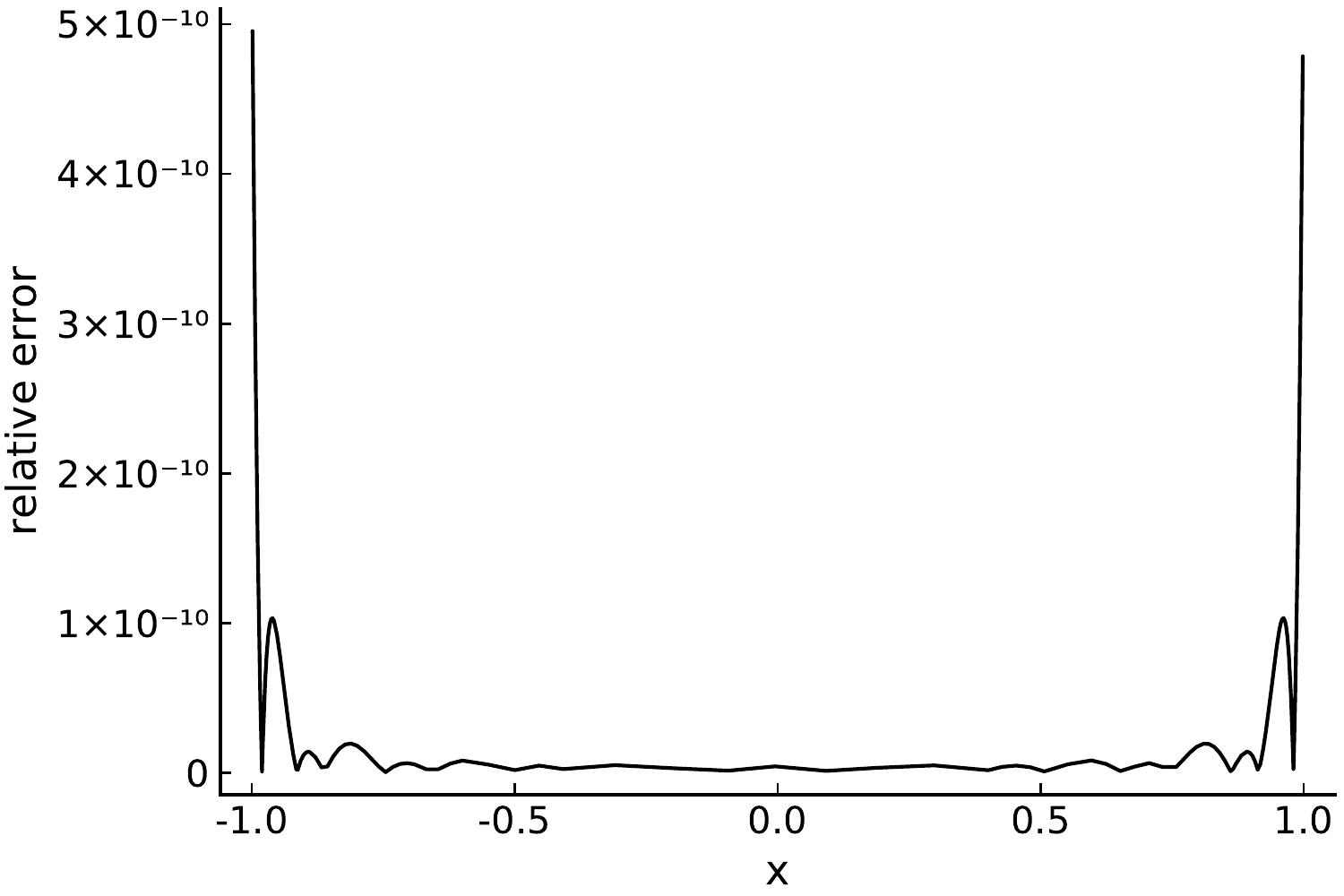} }}
         \subfloat[]
    {{ \centering \includegraphics[width=4.2cm]{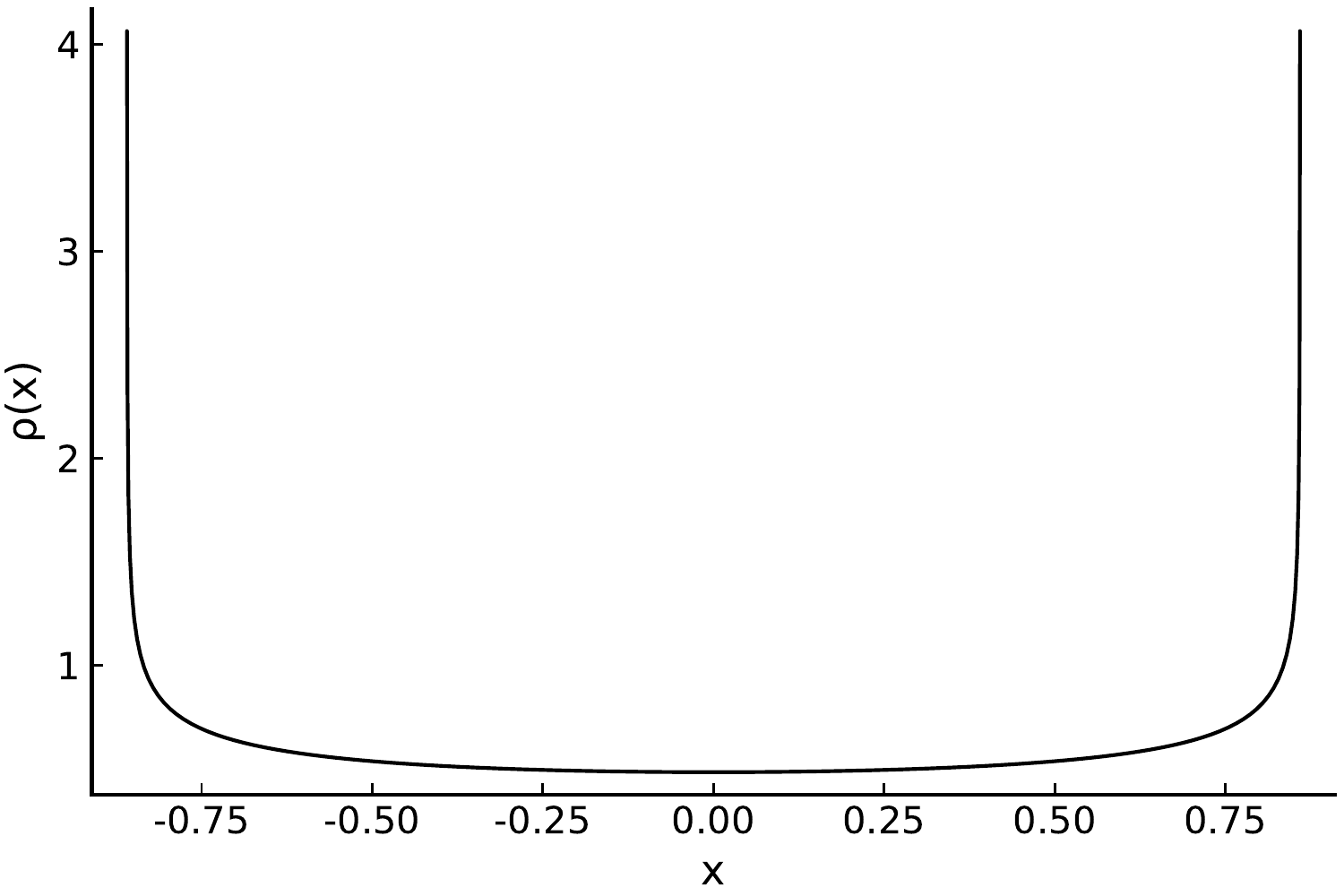} }}
    \caption{(A) and (B) respectively show absolute and relative errors comparing analytic and computed solutions for an attractive-repulsive equilibrium problem with parameters $\alpha = 2$, $\beta = 1.5$ and $M=1$. Normalized to $(-1,1)$ for ease of comparison. (C) shows the numerically obtained solution $\rho(x)$ on its radius of support.}%
    \label{fig:set1-1_errors}%
\end{figure}
\begin{figure}[ht]
     \subfloat[]
    {{ \centering \includegraphics[width=4.2cm]{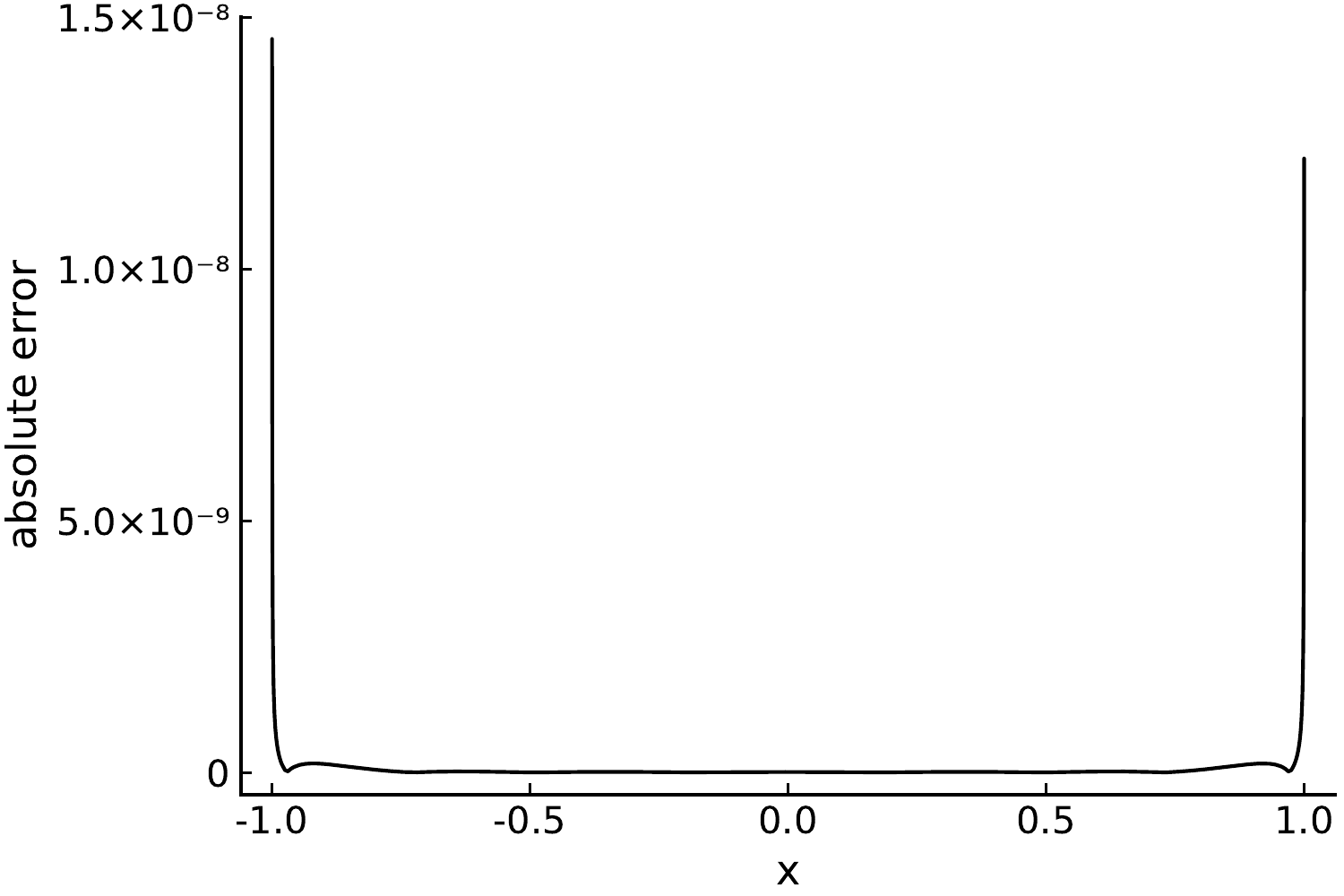} }}
     \subfloat[]
    {{ \centering \includegraphics[width=4.2cm]{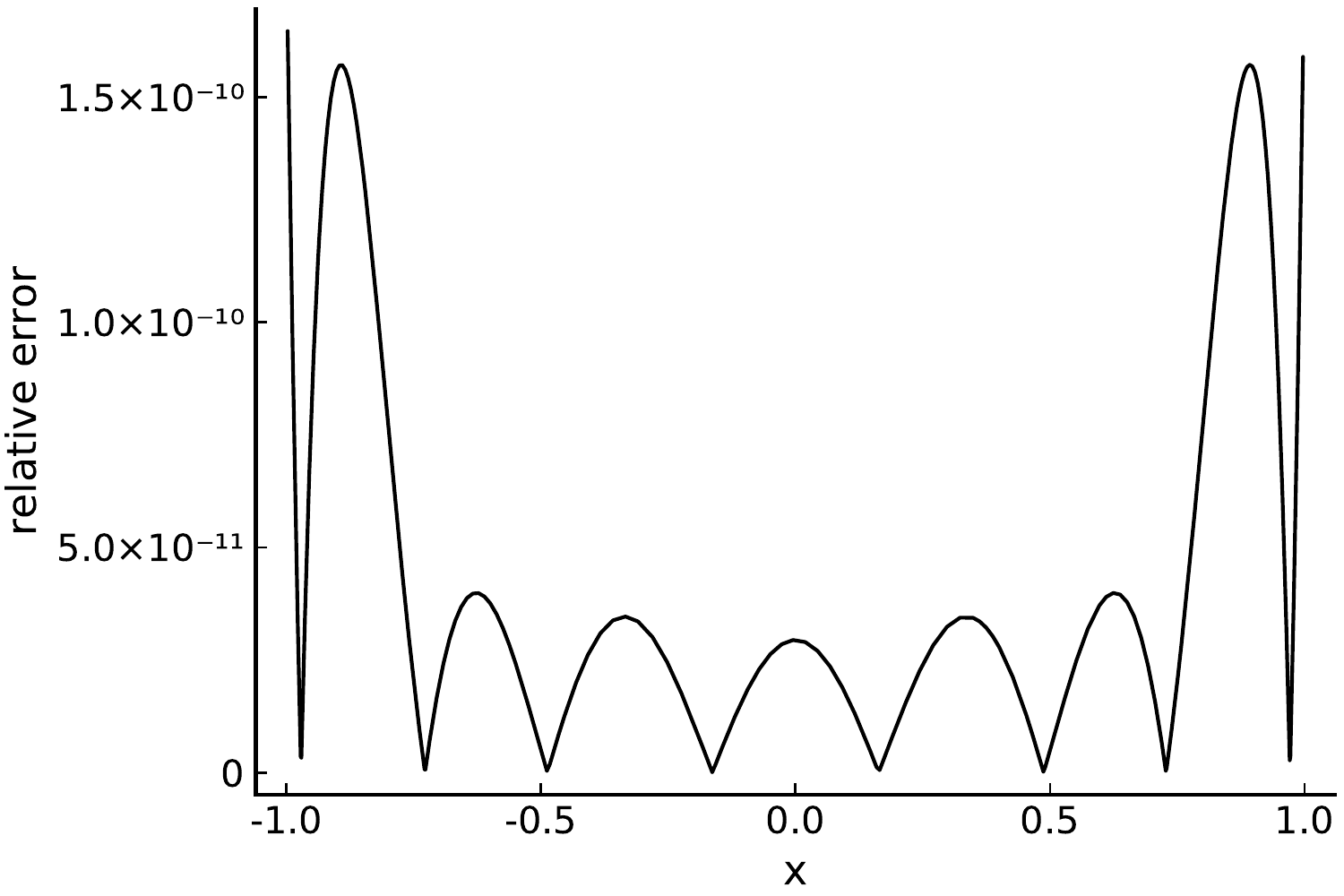} }}
         \subfloat[]
    {{ \centering \includegraphics[width=4.2cm]{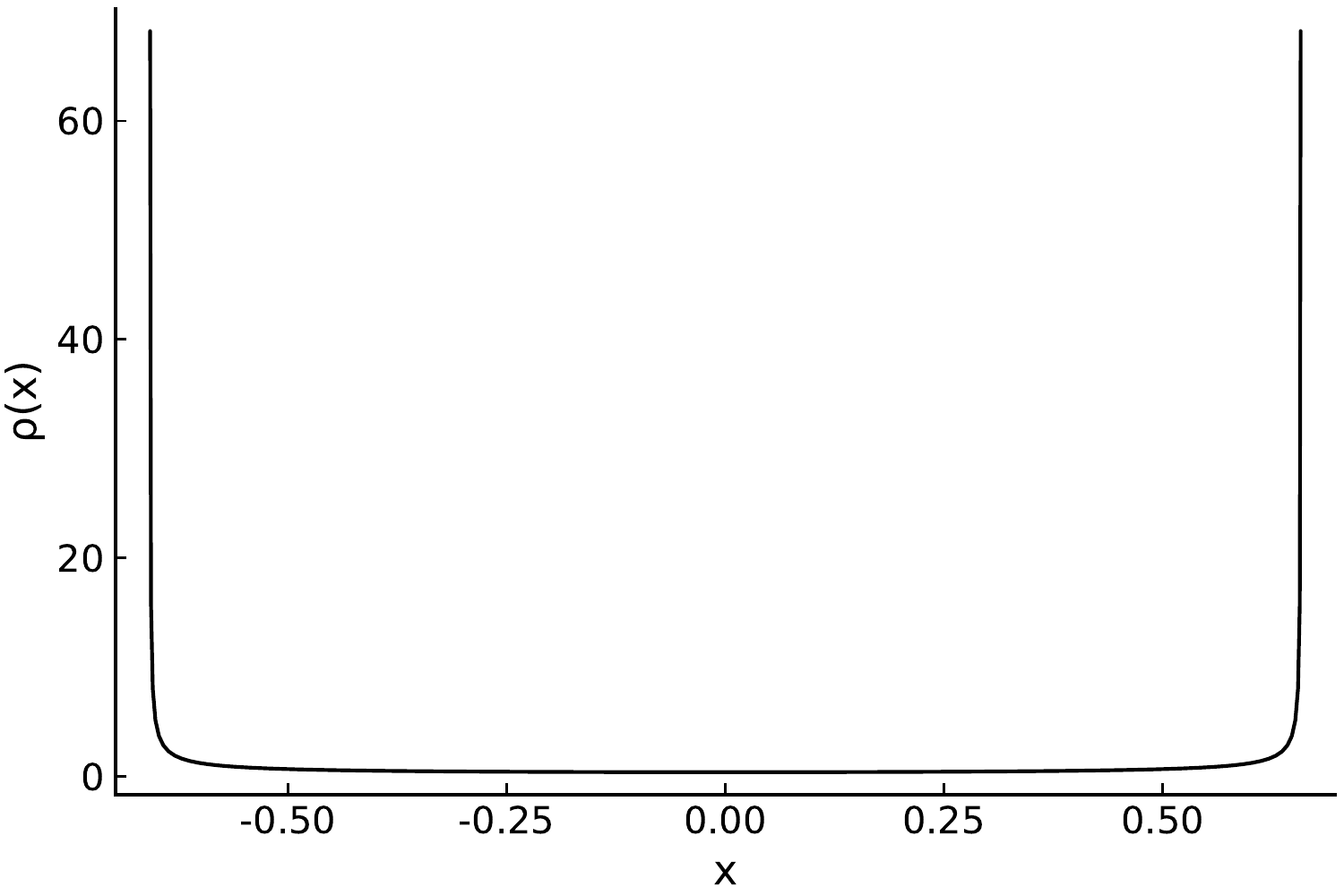} }}
    \caption{(A) and (B) respectively show absolute and relative errors comparing analytic and computed solutions for an attractive-repulsive equilibrium problem with parameters $\alpha = \frac{7}{3}$, $\beta = 2$ and $M=3$. Normalized to $(-1,1)$ for ease of comparison. (C) shows the numerically obtained equilibrium measure $\rho(x)$ on its radius of support.}%
    \label{fig:set1-2_errors}%
\end{figure}
\begin{figure}[ht]
     \subfloat[]
    {{ \centering \includegraphics[width=4.2cm]{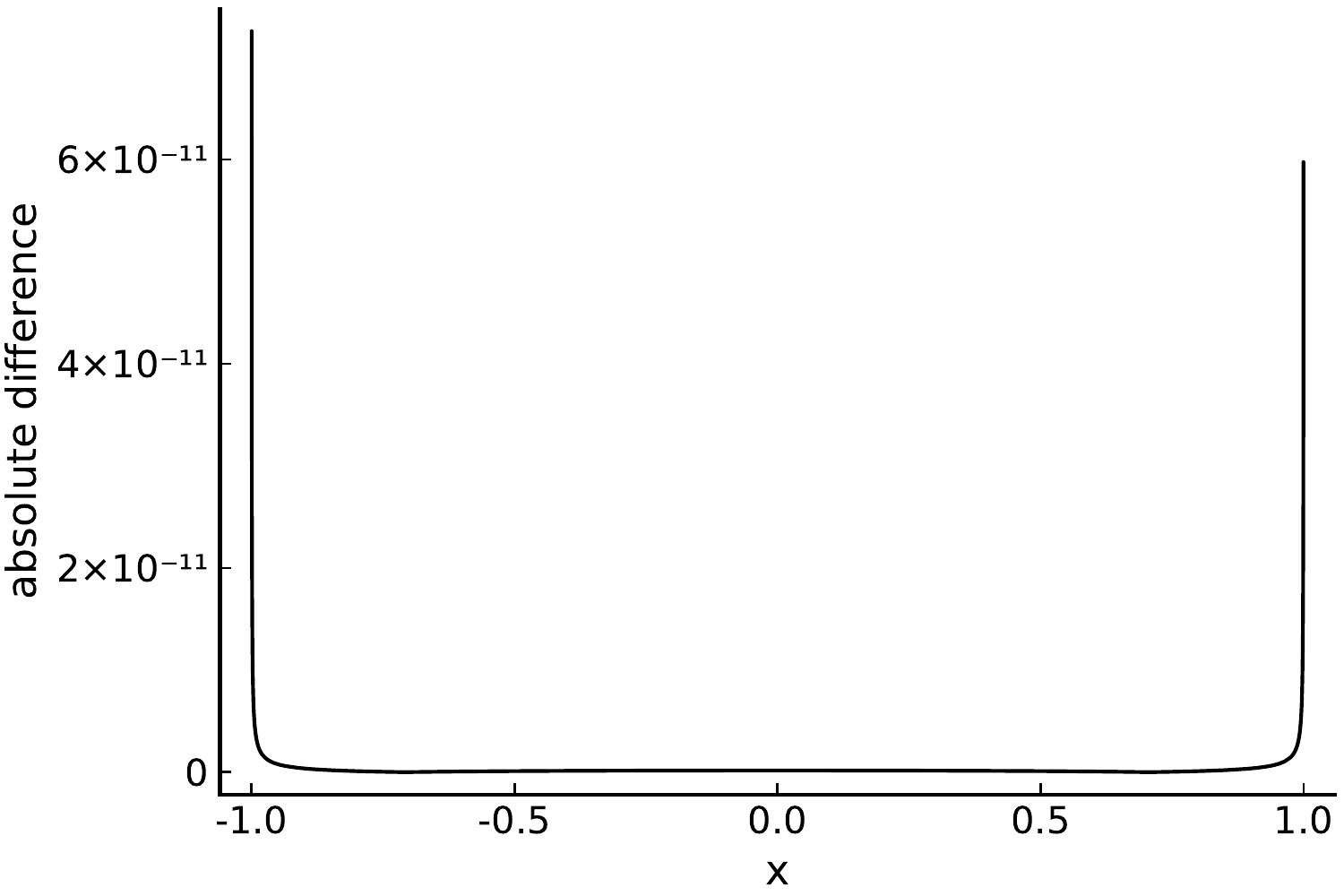} }}
     \subfloat[]
    {{ \centering \includegraphics[width=4.2cm]{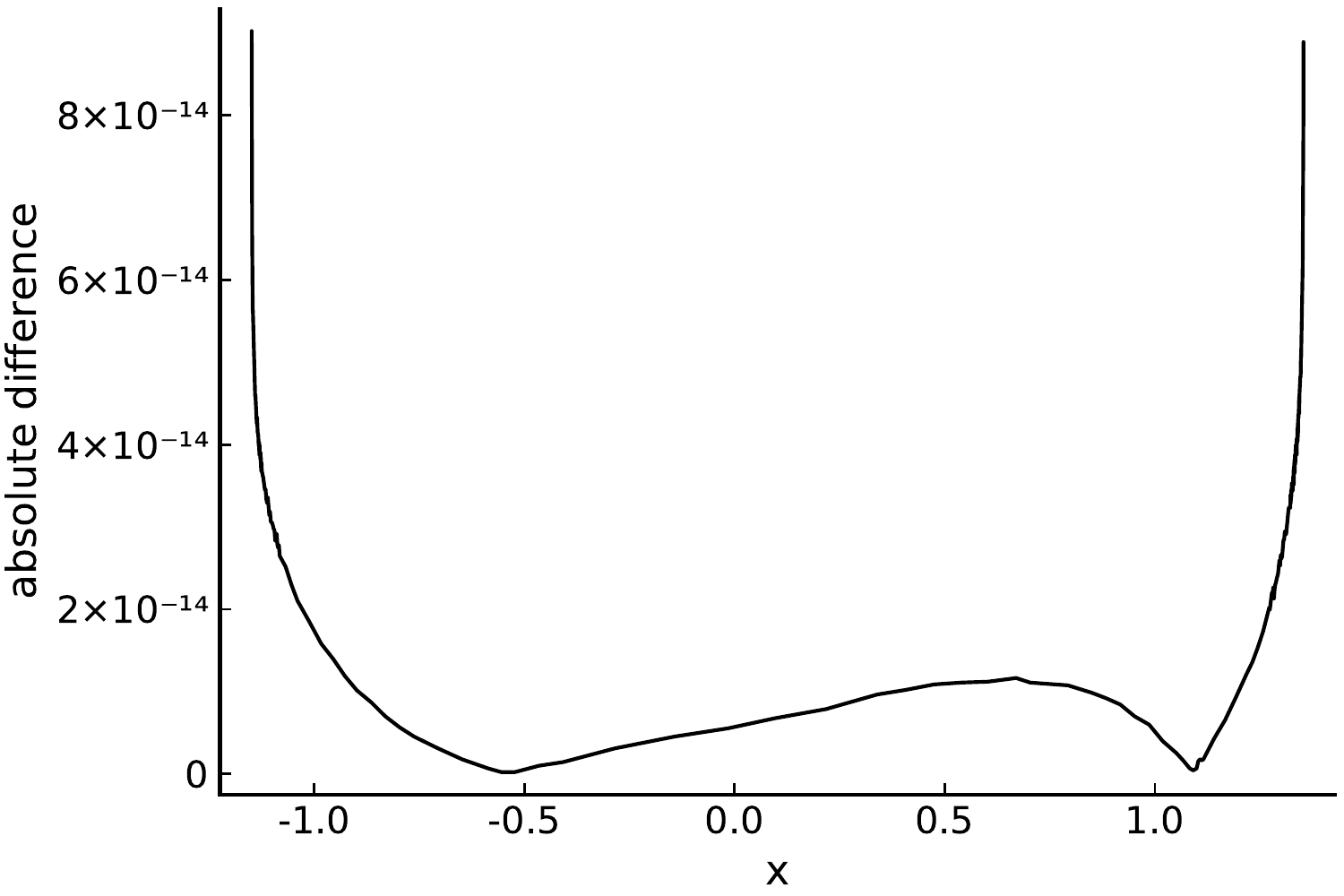} }}
         \subfloat[]
    {{ \centering \includegraphics[width=4.2cm]{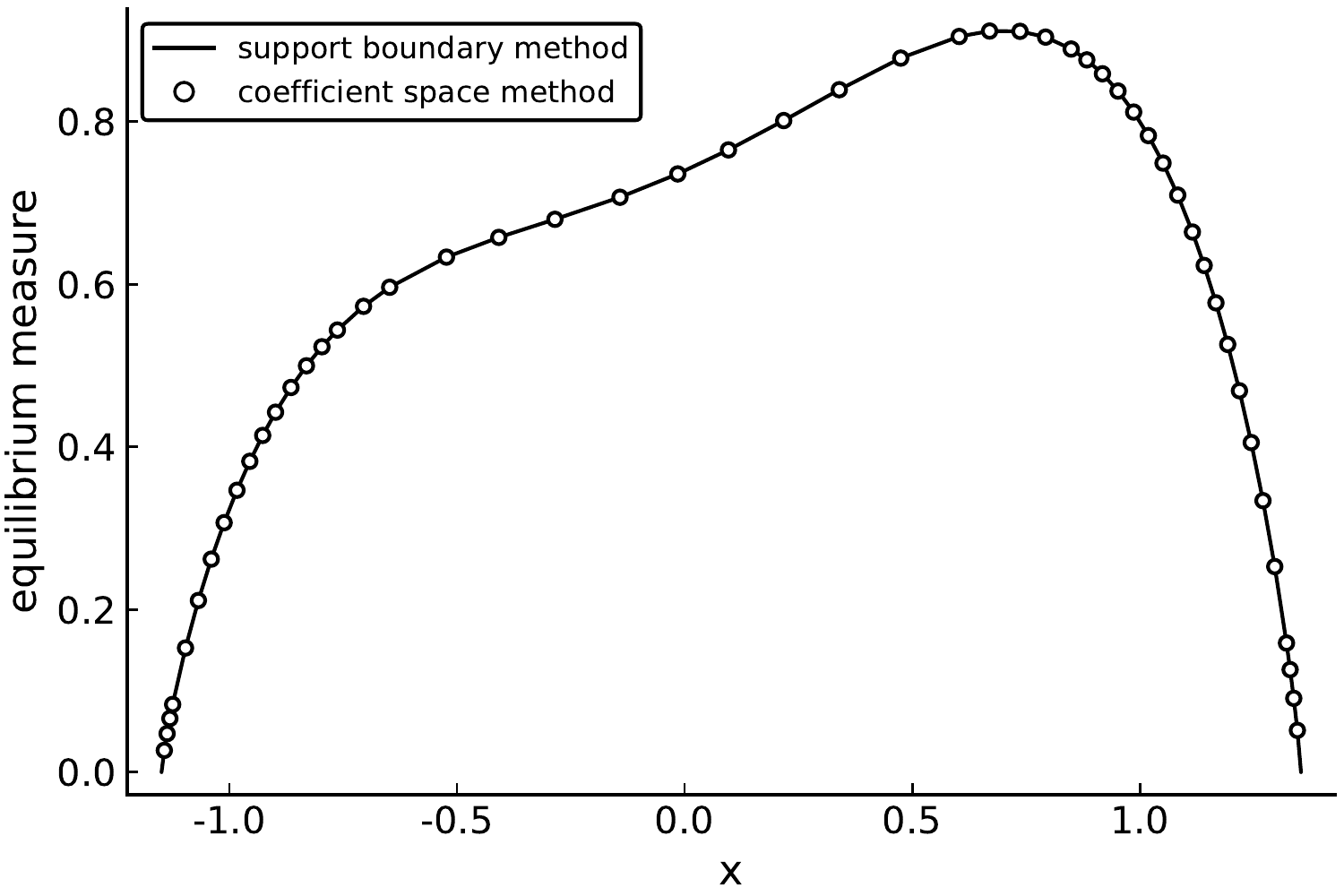} }}
    \caption{(A) shows absolute difference between the two different methods for the problem defined by \eqref{eq:example2-1}, (B) shows the difference for the problem  defined by \eqref{eq:example2-2}. (C) shows the solution obtained by the two methods for the asymmetric problem defined by \eqref{eq:example2-2}.}%
    \label{fig:set2numexp}%
\end{figure}
\subsection{Convergence of solution coefficients}\label{sec:convergencebandwidth}
In this section we investigate the convergence behaviour of the coefficients obtained by our method. In particular, we will compare the diagonal attractive operator and thus in a sense optimal non-analytic (i.e. not an even integer) case where $\alpha_1 \in (-1,1)$ with a higher bandwidth requiring case $\alpha_2 > 1$. We pick two pseudo-random $\alpha_i$ values for this comparison which are representative of the generic observed behaviour and fix a joint $\beta < \alpha_i$:
\begin{align}
(\alpha_1,\quad \beta,\quad M) &= \left(0.912,\quad 0.881,\quad 1\right),\label{eq:convergencelowhigh1}\\
(\alpha_2,\quad \beta,\quad M) &= \left(1.772,\quad 0.881,\quad 1\right),\label{eq:convergencelowhigh2}
\end{align}
As no analytic solutions are available in these generic cases, we show the convergence of the solution coefficients in the ultraspherical polynomial basis in the numerical sense in  Figure \ref{fig:revisionplots1}. Finally, we plot the solution equilibrium measures as well as discrete particle simulation histograms corresponding to the parameter choices in (\ref{eq:convergencelowhigh1}-\ref{eq:convergencelowhigh2}) in Figure \ref{fig:revisionplots2} to further demonstrate that the appropriate solutions were obtained even outside of the range of current analytic approaches.
\begin{figure}[ht]
     \subfloat[$(\alpha_1, \beta) = \left(0.912, 0.881 \right)$]
    {{ \centering \includegraphics[width=6.4cm]{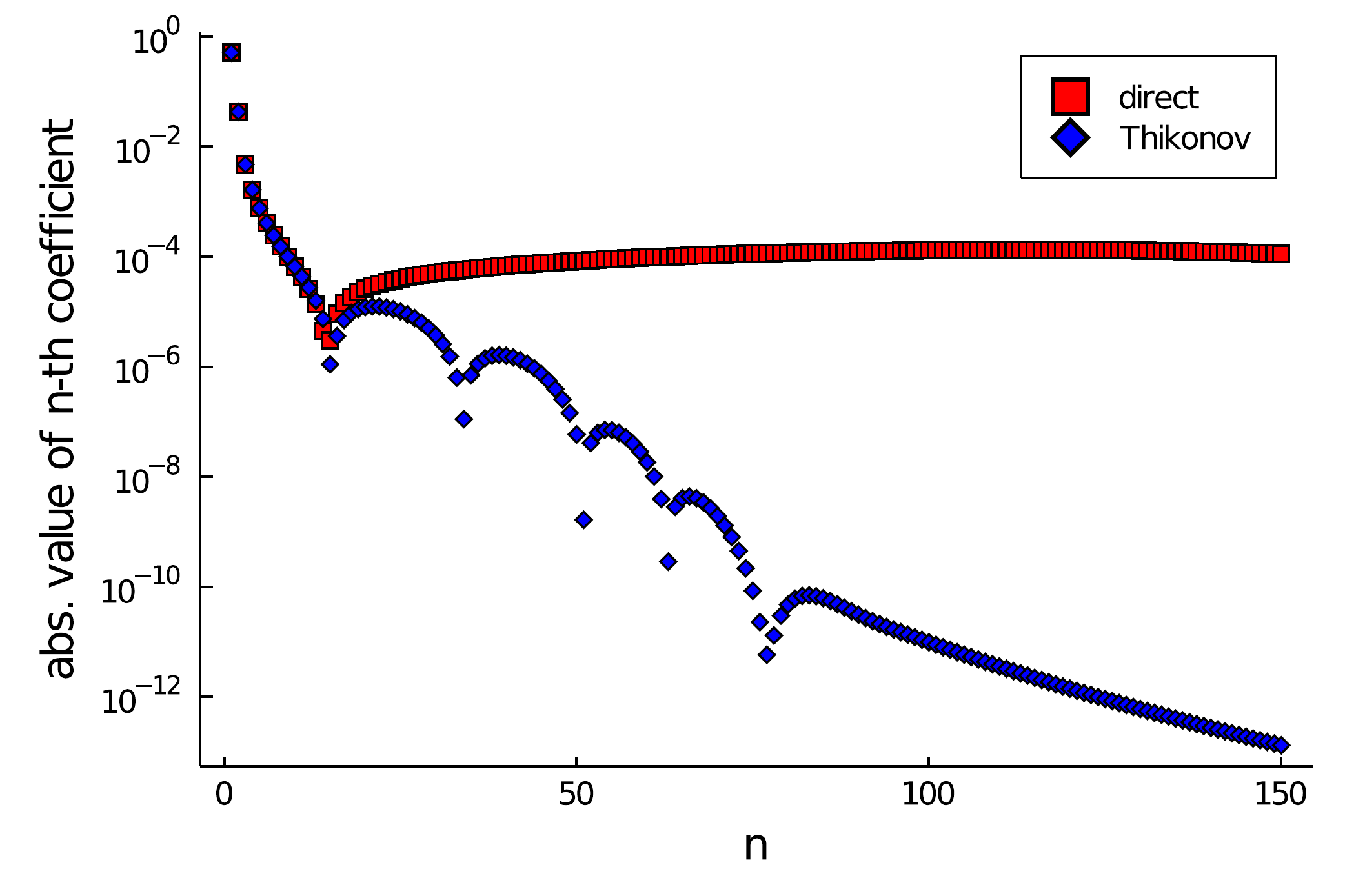} }}
     \subfloat[$(\alpha_2, \beta) = \left(1.772, 0.881 \right)$]
    {{ \centering \includegraphics[width=6.4cm]{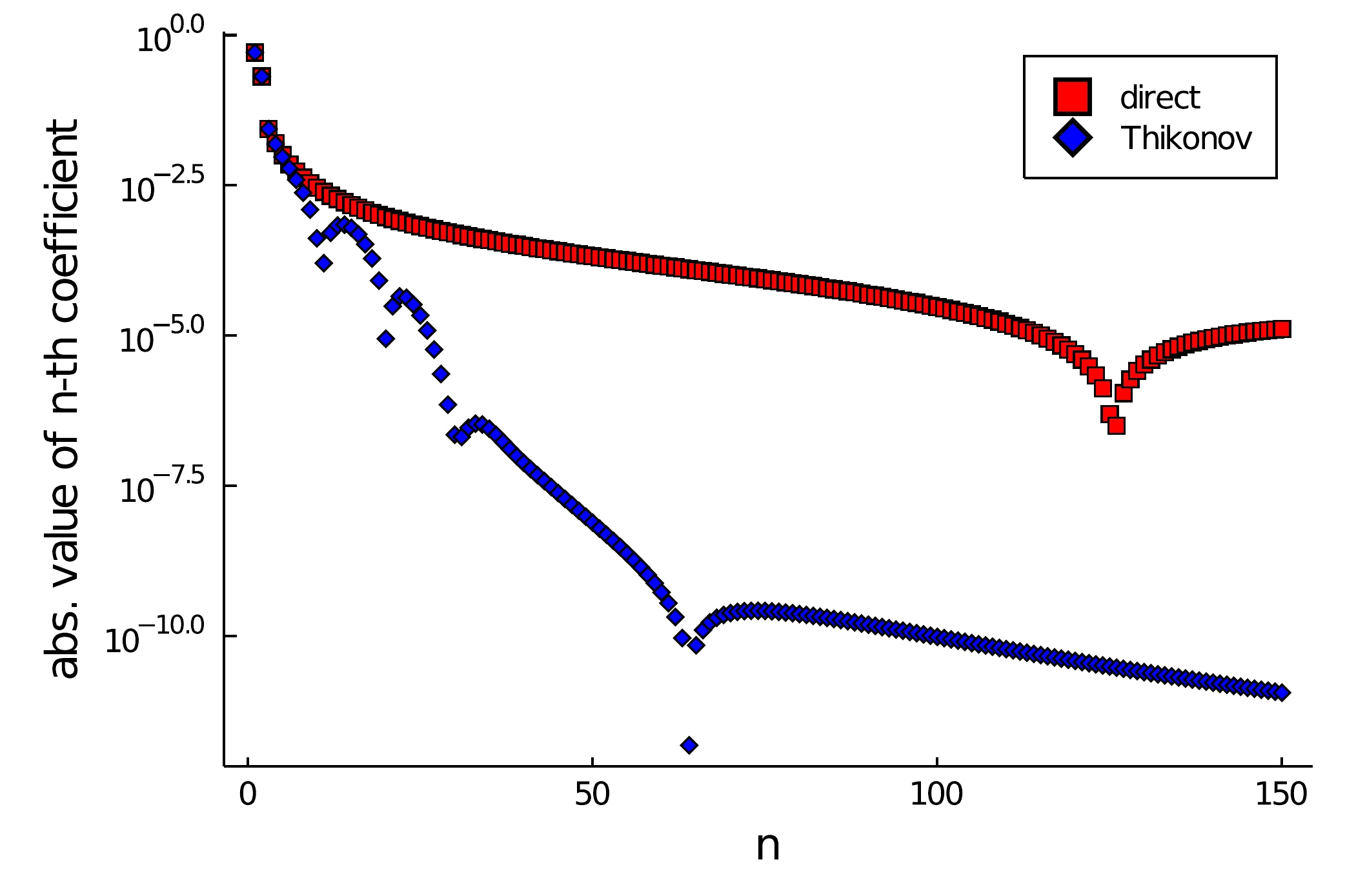} }}
    \caption{Semi-log plot of absolute values of the solution coefficients approximated with $n=150$ for problems (\ref{eq:convergencelowhigh1}-\ref{eq:convergencelowhigh2}).}%
    \label{fig:revisionplots1}%
\end{figure}
\begin{figure}[ht]
     \subfloat[$(\alpha_1, \beta) = \left(0.912, 0.881 \right)$]
    {{ \centering \includegraphics[width=6.4cm]{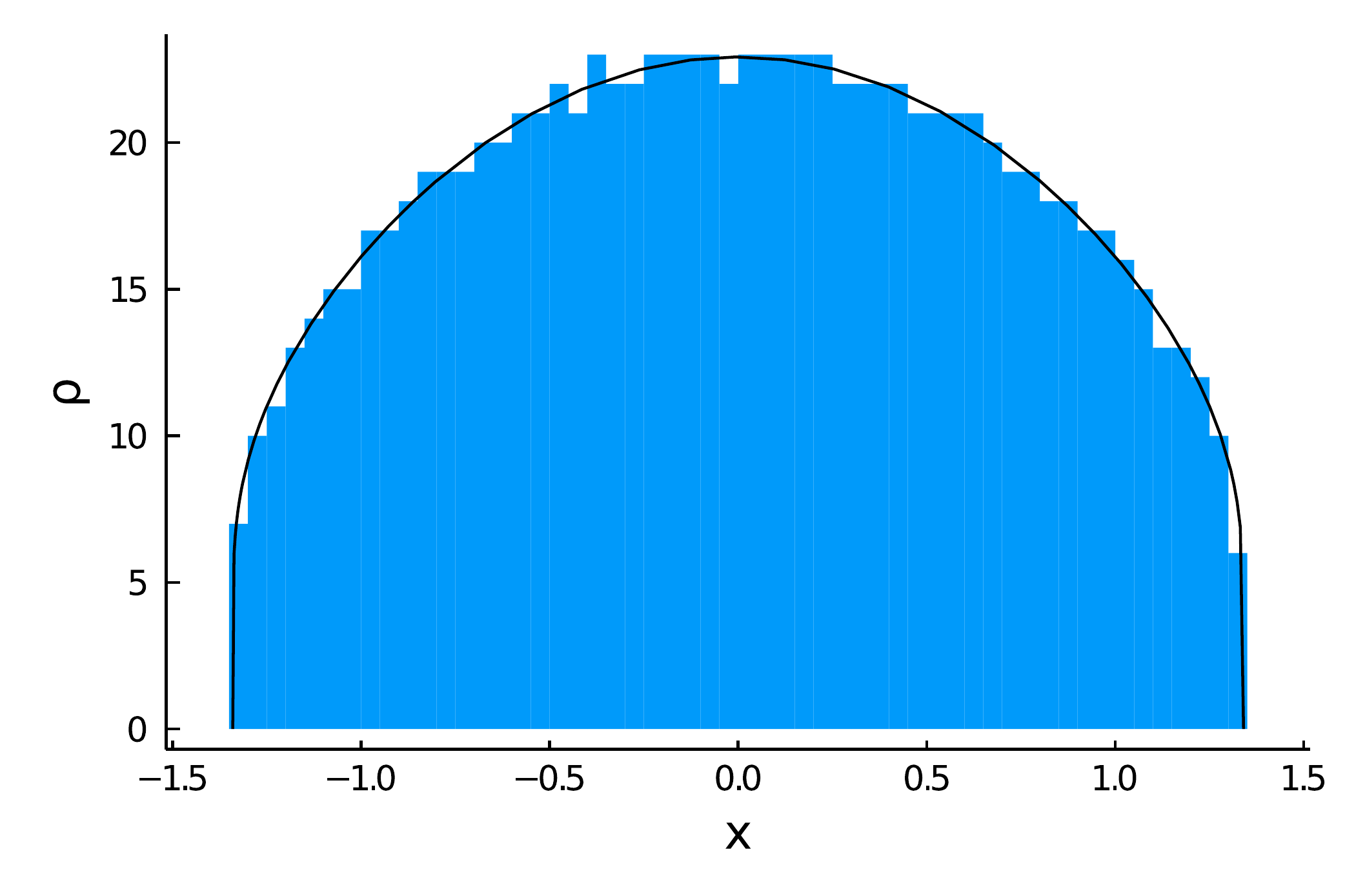} }}
     \subfloat[$(\alpha_2, \beta) = \left(1.772, 0.881 \right)$]
    {{ \centering \includegraphics[width=6.4cm]{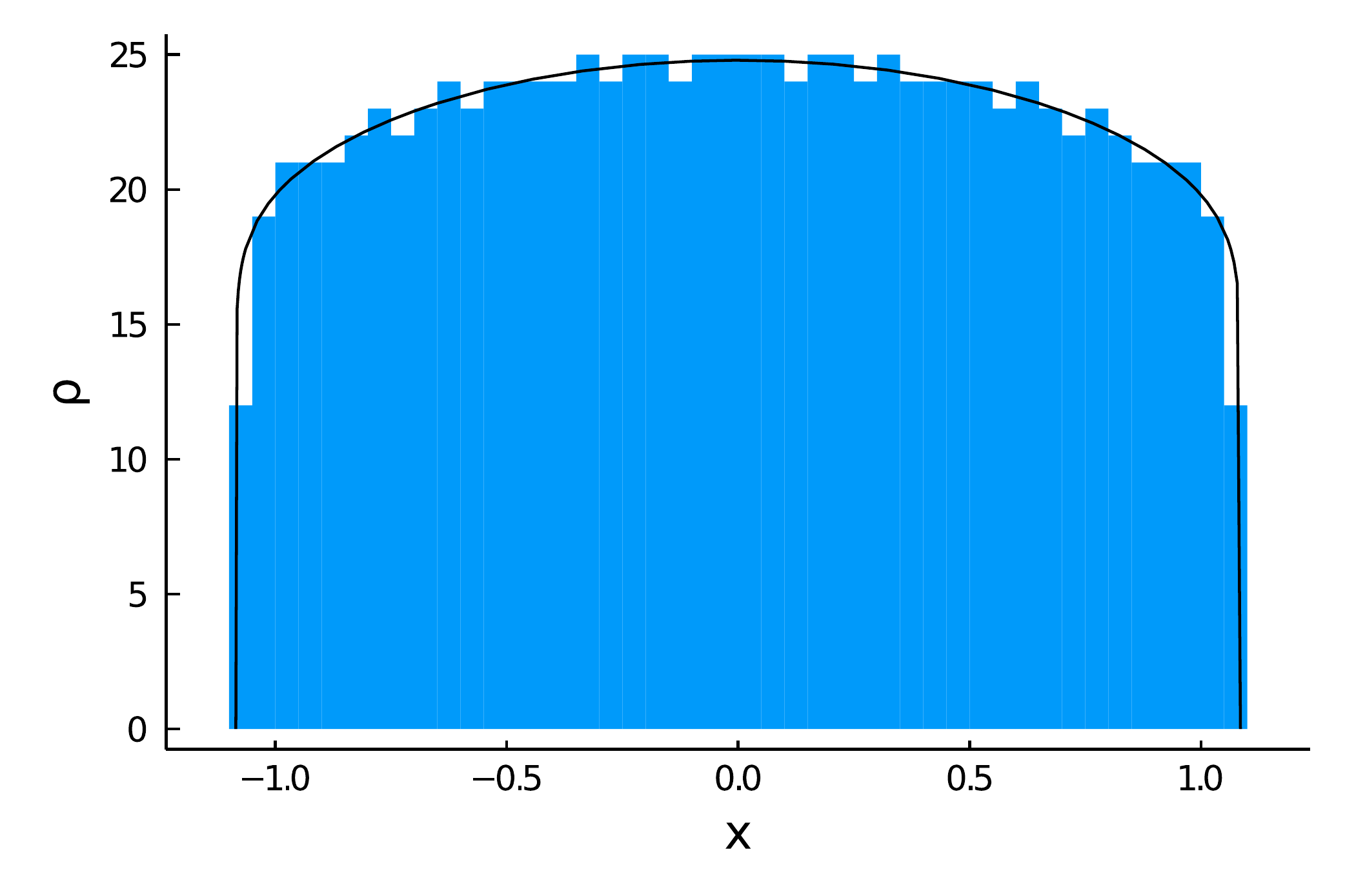} }}
    \caption{Computed solution measures in black and $1000$ particles based density histograms in blue for the problems in (\ref{eq:convergencelowhigh1}-\ref{eq:convergencelowhigh2}).}%
    \label{fig:revisionplots2}%
\end{figure}
\section{Numerical investigation of open problems}\label{sec:numericalopenQ}
We now move away from testing the introduced method against known solutions and alternative approaches and instead use it to investigate conjectures in parameter ranges outside of reach of current analytical methods. In Section \ref{sec:alpha4} we study the uniqueness of global minimizers outside of known parameter ranges. In Section \ref{sec:beta1} we explore the gap formation boundary and behaviour for general $\alpha$ and $\beta$. In Section \ref{sec:twointervalexperiments} we use our two interval approach to further explore the post-gap formation equlibrium measures, including comparisons to discrete particle models. Finally, Section \ref{sec:twointervalexperiments2} explores the transition to two interval support in more detail.
\subsection{Uniqueness and existence of single interval solutions}\label{sec:alpha4}
As mentioned in the introduction, various analytic results are known with regards to the existence and uniqueness of equilibrium measures for power law kernels. Some general existence results for global minimizers, including power law kernels and some Morse-type potentials, can be found in \cite{canizo_existence_2015}. Choksi et al. derived the unique existence of a global minimizer for power law kernels with parameters $\alpha=2$ and $\beta \in (-1,0)$ in \cite{choksi_minimizers_2015}. Lopes \cite{lopes_uniqueness_2019} extended this known parameter range with uniqueness to $\alpha \in (2,4)$, $\beta	\in (-1,0)$. Carrillo and Huang \cite{carrillo_explicit_2016} obtained conjectured solutions for various parameter ranges where either $\alpha$ or $\beta$ is an even integer, which were recently proved to be the unique global minimizers in \cite{carrillo_radial_2021}.\\
In this section, we explore the question of uniqueness of equilibrium measure solutions for the entire admissible parameter range $\beta \in (-1,2)$ from a numerical perspective. Our results agree with the proved results in \cite{lopes_uniqueness_2019} and add additional support to the conjecture that the steady states found in \cite{carrillo_explicit_2016} are indeed the unique global minimizers (obtained prior to the proofs given in \cite{carrillo_radial_2021}).\\
In Figure \ref{fig:uniquenesstests} we plot the energy as a function of the support radius for parameter choices which are representative of the generally observed behaviour---in the case where known analytic solutions exist we indicate its location and find that it coincides with the energy minima as one would expect. While the energy may be made arbitrarily small for each parameter choice by continuously increasing the radius of the support, positive measures only exist in the neighbourhood of the observed local minimum. This does not pose a problem for optimization methods, as the optimization methods can easily be tuned to find local minima and a soft penalty method which penalizes negative measures is straightforwardly implemented for well-behaved functions like these. These local energy minimizers are thus found to be the global minimizers among admissible measures. The general pattern observed for energy plots of this sort is the existence of positive measures with compact support in a neighbourhood of the minimizer, with only not strictly positive measures existing for smaller or larger radii. Some specific problems appear to only turn negative as the radius grows, see Figure \ref{fig:uniquenesstests}, but this does not impact the uniqueness of the obtained minimizers as the energy grows with decreasing radius.\\
Our method also allows us to study where single interval solutions can exist. As we will explore in more detail in the following section, for certain high parameter ranges we observe gap forming behaviour for the equilibrium measures and one would thus have to utilize multiple interval approaches to find an appropriate solution (with gap). To study the existence of single interval solutions, and thus also gap forming behaviour, we can for example take fixed $\alpha = 3.5$ and increase $\beta$ within the interval $(1,2)$ starting at $1$ until we observe the described phenomenon for the single interval approach. A rough visualization of this can be found in Figure \ref{fig:existencesingleinterval}, where we see that for $\alpha = 3.5$ single interval solutions stop existing somewhere approximately at $1.71$ --- we study this boundary for general $\alpha$ and in more detail in the next section.
\subsection{Gap forming behaviour when $\beta \geq1 $}\label{sec:beta1}
Analytically obtained single interval steady state solutions for certain integer parameters show negative values at the origin for high enough values \cite{carrillo_explicit_2016}, indicating that the interval of support splits into multiple intervals. Numerical investigation on the basis of discrete particle swarms has further confirmed that a gap appears to form around the origin \cite{balague_dimensionality_2013} but the precise nature of the transition boundary from single interval support to multiple interval support is unknown. Figure \ref{fig:gapformingbehaviour} shows scans of the parameter range $\alpha \in (2,4)$ and $\beta \in (1,2)$ using the method introduced in this paper, clearly showing the gap-formation boundary, which may be computed with arbitrary smoothness by probing a denser grid of points. As the trajectory of this boundary is not analytically known, it is of interest to explore it using the introduced numerical method.
\begin{figure}[ht]
         \subfloat[$(\alpha, \beta, M) = \left(2, -0.5, 1 \right)$]
    {{ \centering \includegraphics[width=5.8cm]{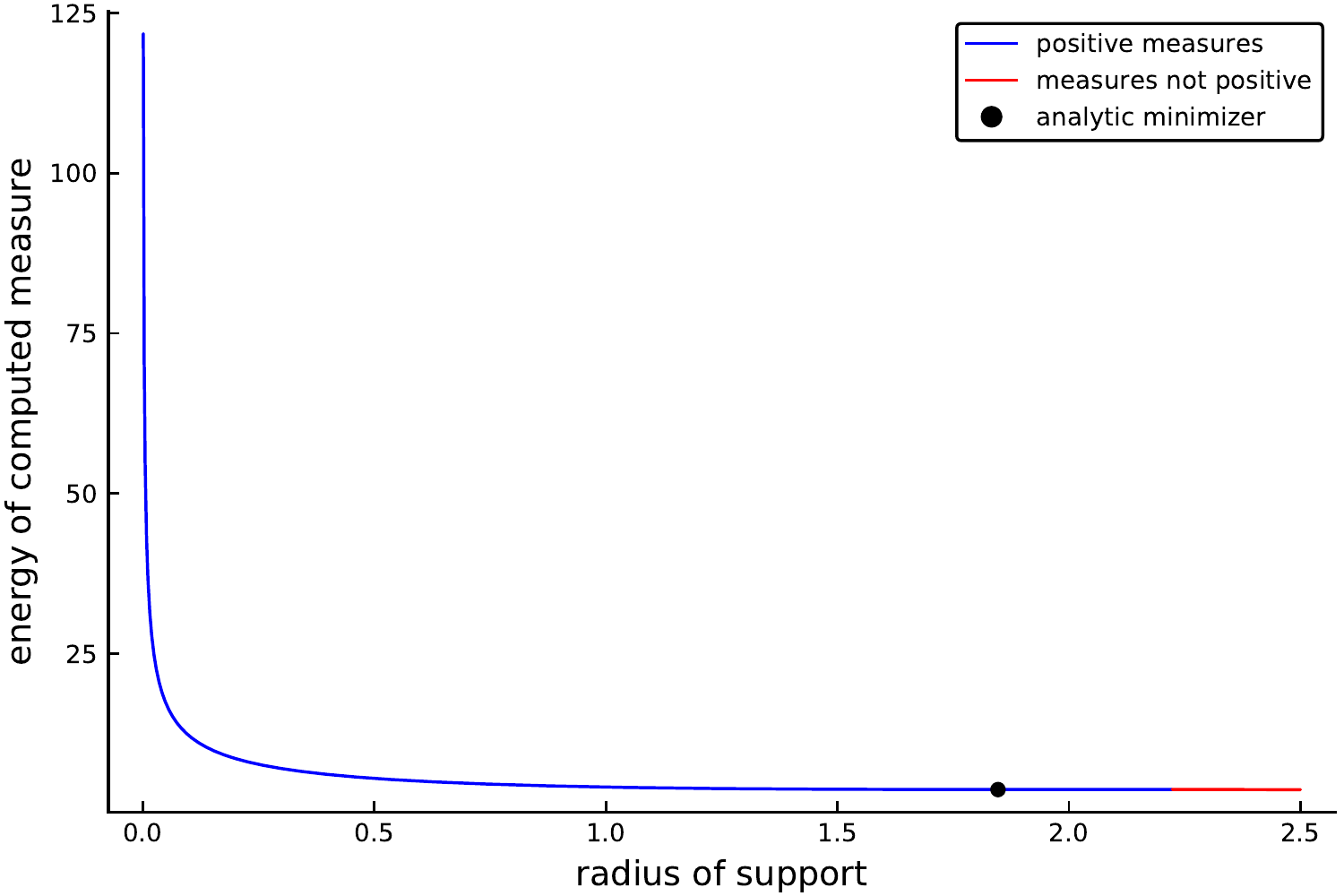} }}
             \subfloat[$(\alpha,  \beta, M) = \left(2, -0.5, 1 \right)$]
    {{ \centering \includegraphics[width=5.8cm]{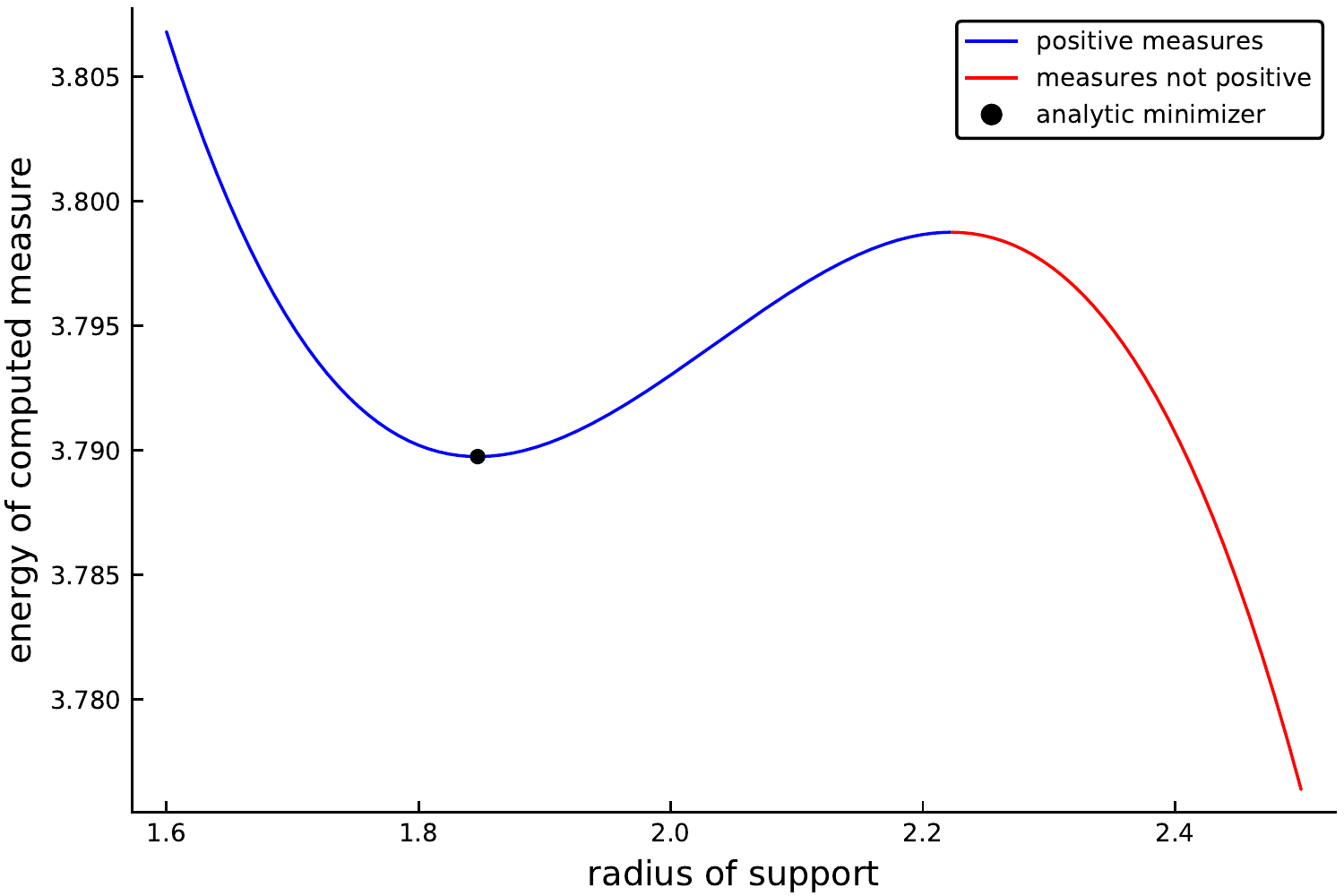} }}\\
         \subfloat[$(\alpha,  \beta, M) = \left(3.5, 1.4, 1 \right)$]
    {{ \centering \includegraphics[width=5.8cm]{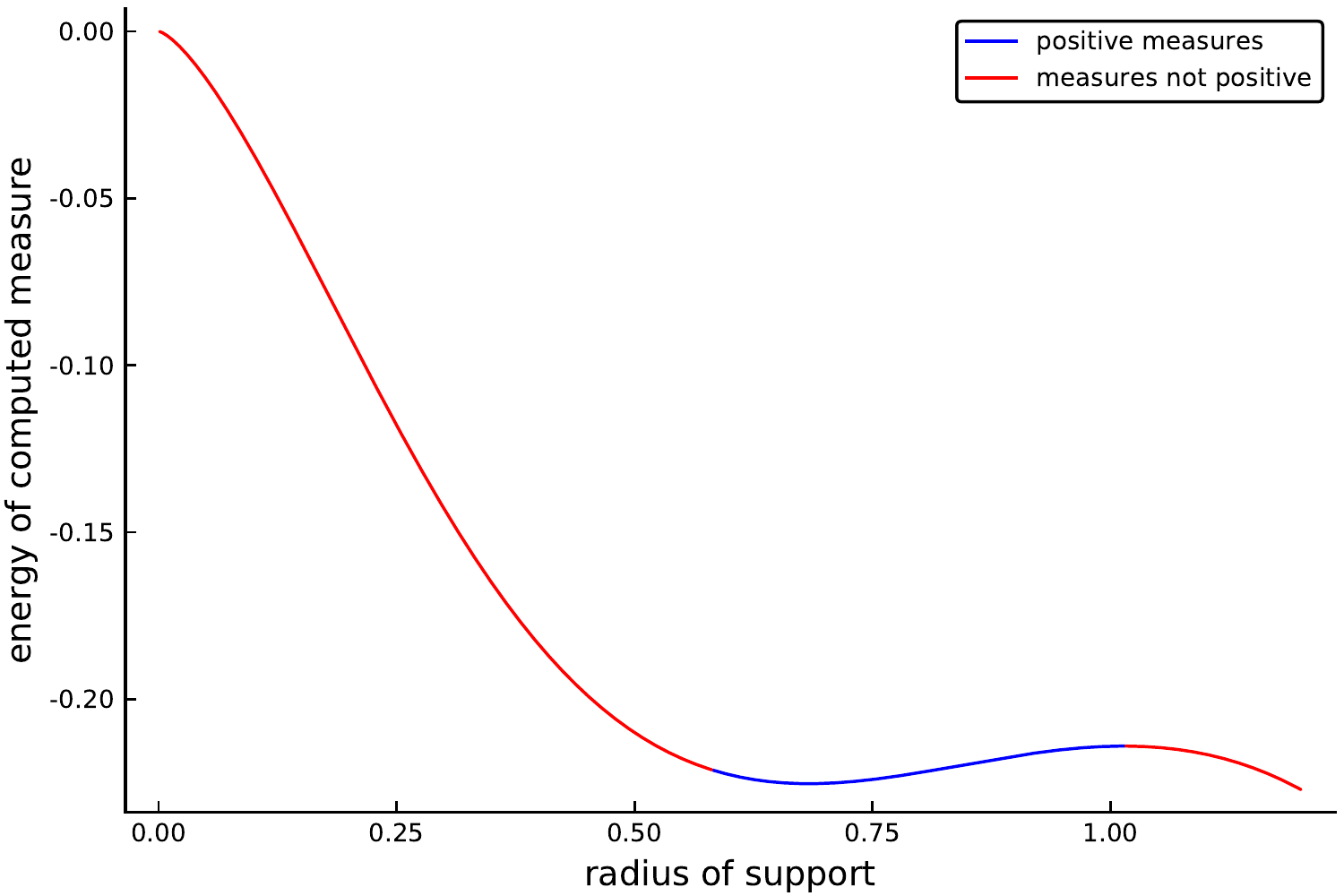} }}
    \caption{Energy as a function of the radius of computed measures with compact single interval support for different parameters. (B) is a close-up of the data in (A) to make the existence of a well-posed minimum more recognizable.}%
    \label{fig:uniquenesstests}%
\end{figure}
\begin{figure}[ht]
    \centering 
    \includegraphics[width=6cm]{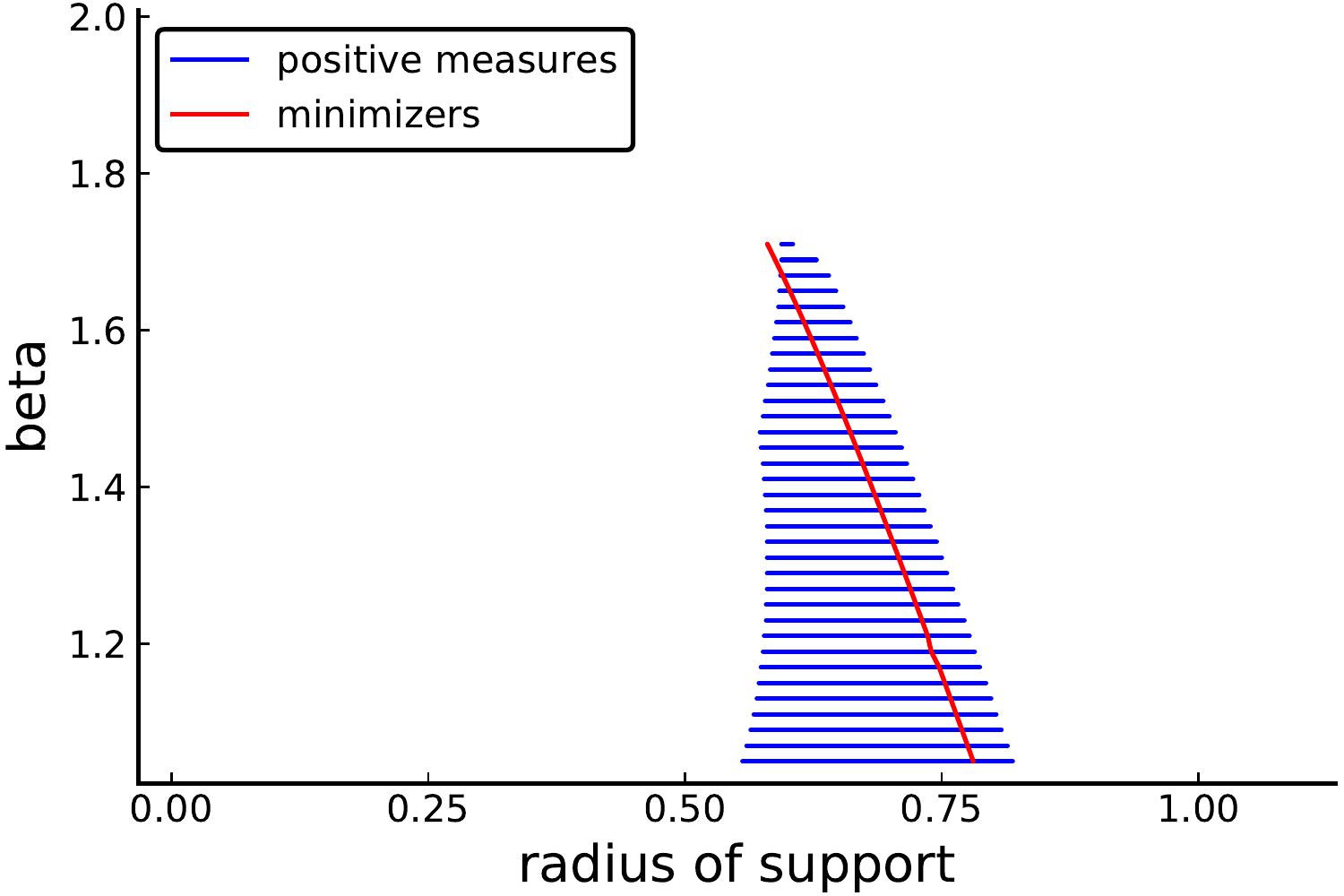}
    \caption{The existence of positive measures of single interval support for given radius is indicated in blue for $\alpha=3.5$ and $\beta \in (1,2)$ on the $y$-axis in steps of $0.02$. Positive measures stop existing approximately at $1.71$.}%
    \label{fig:existencesingleinterval}%
\end{figure}
\begin{figure}[ht]
     \subfloat[]
    {{ \centering \includegraphics[width=4cm]{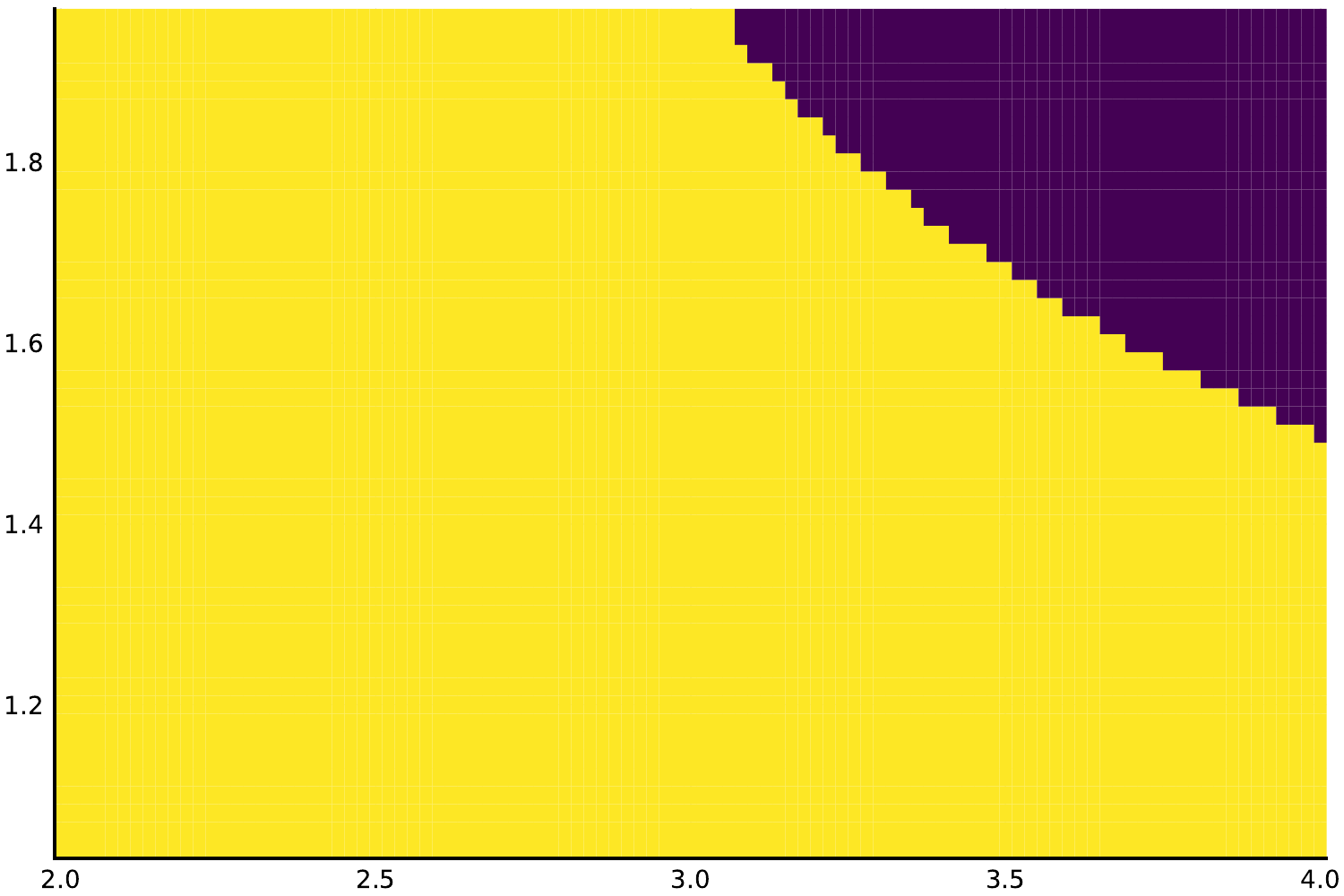} }}
     \subfloat[]
    {{ \centering \includegraphics[width=4cm]{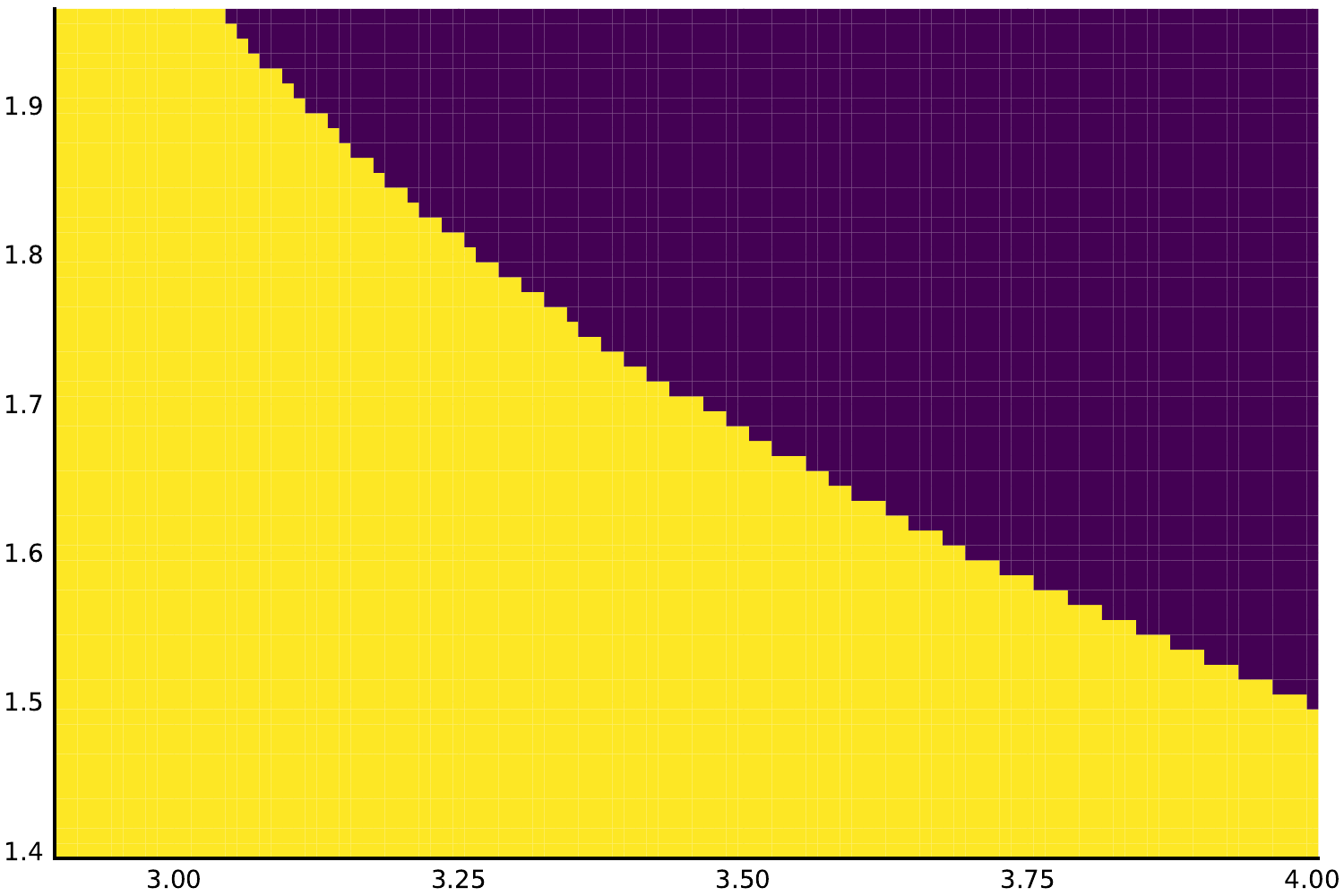} }}
         \subfloat[]
    {{ \centering \includegraphics[width=4cm]{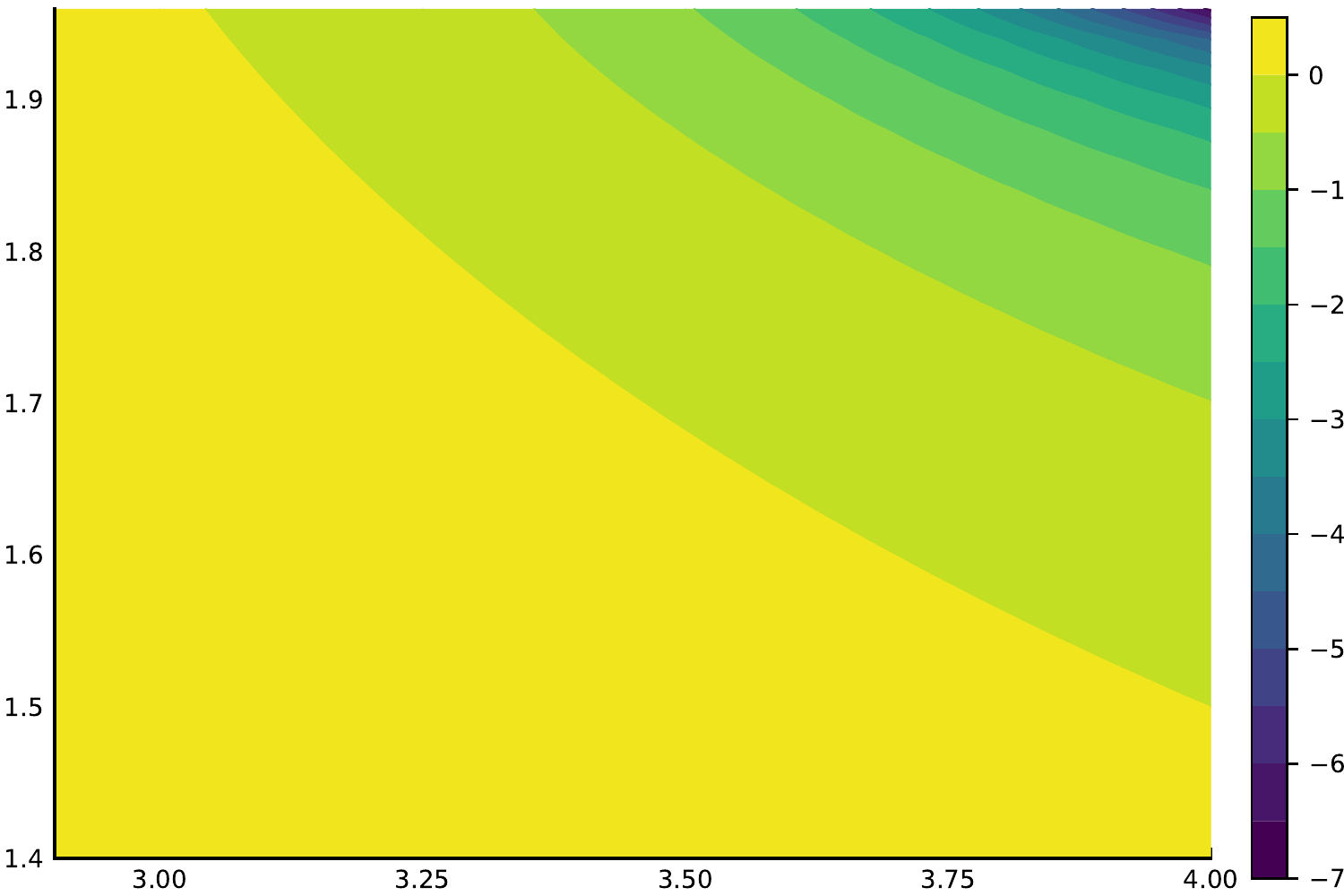} }}
    \caption{(A) shows gap presence with $\alpha$ on the $x$- and $\beta$ on the $y$-axis, while (B) shows higher resolution scan of the boundary. Bright region indicates a positive measure, dark region indicates presence of a gap. (C) shows a contour plot of the minimum values of the equilibrium measures obtained in the zoomed segment. Only the brightest contour represents positive equilibrium measures.}%
    \label{fig:gapformingbehaviour}%
\end{figure}

\subsection{Numerical experiments on two interval measures}\label{sec:twointervalexperiments}
We numerically explore equilibrium measure solutions in the range in which the single interval approach yields negative results and thus implies a gap at the origin, splitting the support into two intervals. We study three example parameter choices: 
\begin{align}
(\alpha_1,\quad \beta_1,\quad M_1) &= \left(4,\quad 1.61,\quad 1\right),\nonumber\\
(\alpha_2,\quad \beta_2,\quad M_2) &= \left(3.34,\quad 1.83,\quad 1\right),\label{eq:twointervalexp}\\
(\alpha_2,\quad \beta_2,\quad M_2) &= \left(3.5,\quad 1.6,\quad 1\right).\nonumber
\end{align}
The parameters with index $1$ involve an even integer value attractive force for which the negative values around the origin in the single interval case were analytically observed in \cite{carrillo_explicit_2016}, which prompted the authors to ask questions about the form of the true solution. The parameters with index $2$ are a more generic parameter combination for which we also find negative values in the single interval approach. The parameters with index $3$ represent a case where the single interval case is well-behaved and expected to be the true solution. It thus serves as a control, where the two interval approach should in principle guide us to the fact that a single interval approach may be more appropriate, by resulting in the lowest energy states for positive measures when the inner boundaries approach $a=-a=0$.\\ 
In Figure \ref{fig:energycontours} we show contour plots for varying values of $a$ and $b$ of the energy for positive measures on two interval support $[-b,-a]\cup[a,b]$. We plot the measures obtained via a constrained optimization over the admissible measures in Figure \ref{fig:twointervalsingleapproach}. Figure \ref{fig:twointervalsingleapproach} also compares the obtained measures to the results of a discrete particle simulation similar to what was described for two and three dimensions in \cite{carrillo_particle_2010} and single interval approach solutions as discussed above and in \cite{carrillo_explicit_2016}.\\
Both the single interval and the discrete particle computations allow us to make predictions about the two interval support, although in the case of the single interval approach this becomes less accurate as the gap increases in size. Conversely, the results of the two interval method can be used to predict whether a single interval solution exists, as is the case for example $3$ in this section where the method implies that the lowest energy positive measures are found near $a=0$, see Figure \ref{fig:energycontours}(C). In both cases in which the support splits into two components, $\rho(x)$ appears to increase as $x$ approaches $a$ and $-a$, with singularities at the inner boundary points. This also held true in numerous other numerical experiments in the two interval support parameter range and is reproduced in particle simulations.
\\
As in the single interval case, excessive regularization can lead to inadmissible solutions appearing admissible and thus obtaining a sensible result is dependent on choosing a good Tikhonov parameter. A good guiding principle is looking for the smallest Tikhonov parameter for which the method is stable as the order of polynomial approximation is increased. Future work on multiple interval methods should explore more specialized regularization approaches to try and minimize this error.
\begin{figure}[ht]
     \subfloat[$(\alpha_1,\beta_1) = \left(4, 1.61\right)$]
    {{ \centering \includegraphics[width=8cm]{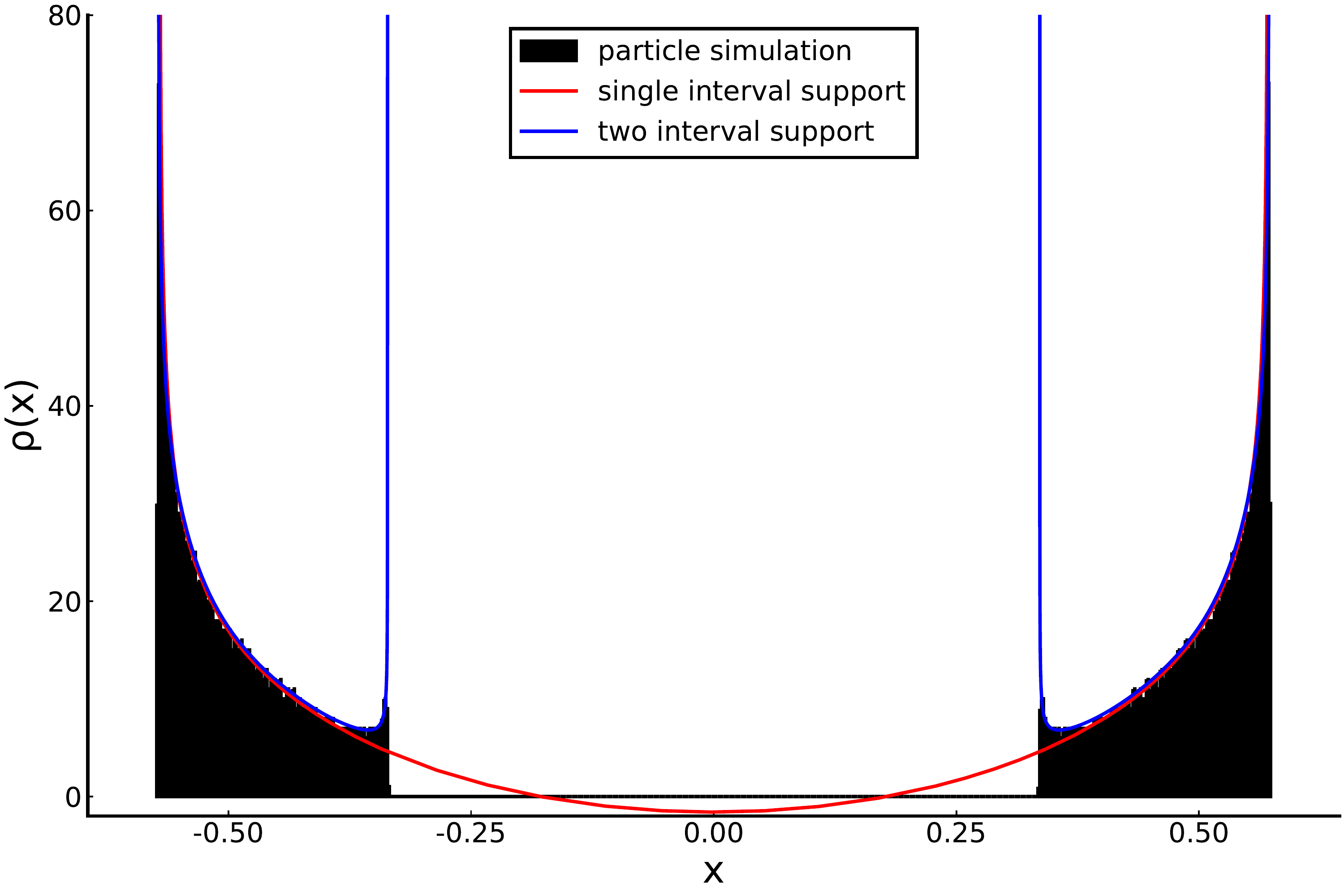} }}\\
     \subfloat[$(\alpha_2,\beta_2) = \left(3.34, 1.83\right)$]
    {{ \centering \includegraphics[width=8cm]{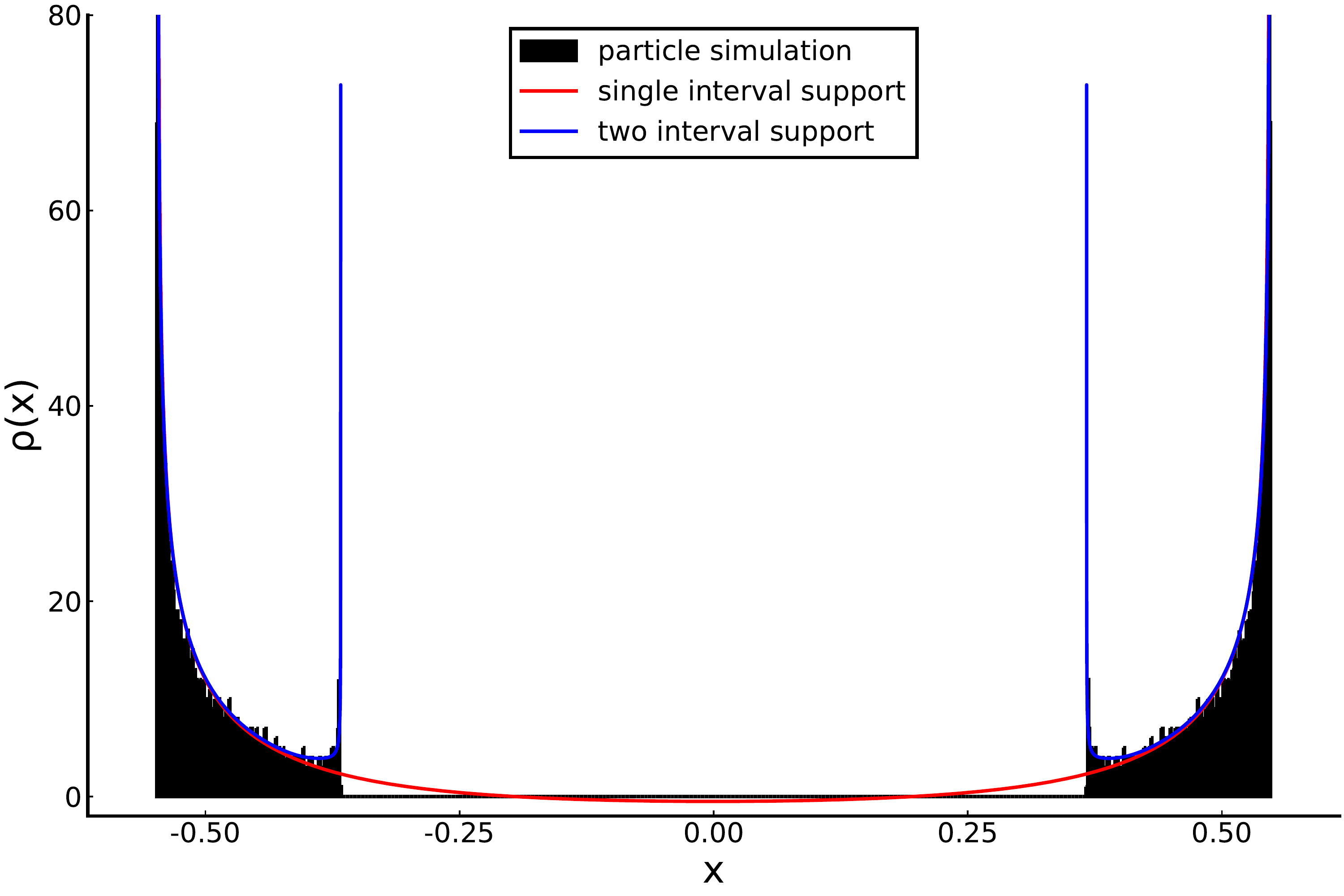} }}\\
     \subfloat[$(\alpha_3,\beta_3) = \left(3.5, 1.6\right)$]
    {{ \centering \includegraphics[width=8cm]{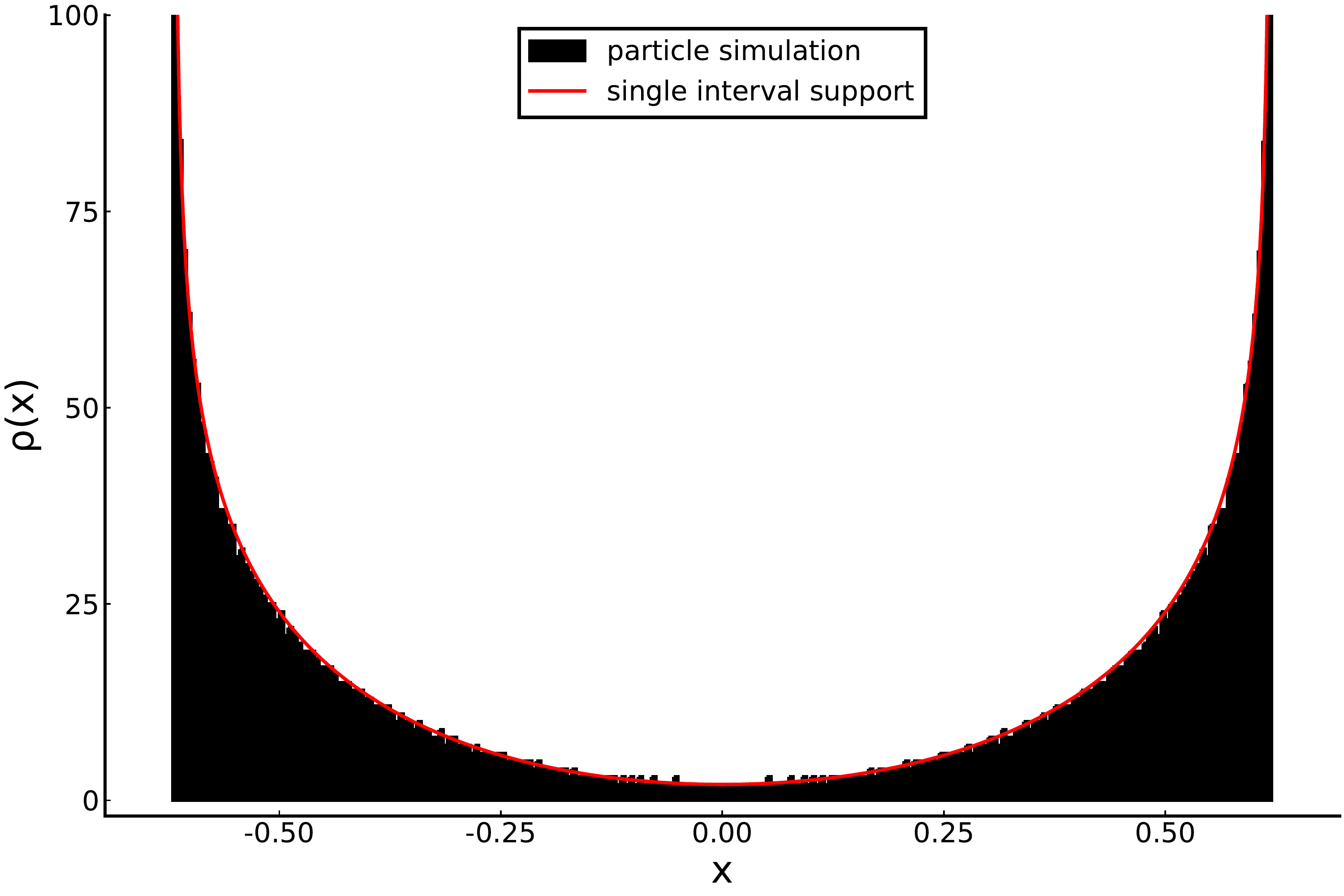} }}
    \caption{Combined plots comparing histograms of the results of a particle simulation, with 3000 particles initially randomly distributed on $(-1,1)$, with equilibrium measure candidates obtained using single and two interval support approaches respectively (measures standardized to histogram mass for ease of comparison), with parameters as in in \eqref{eq:twointervalexp}.} %
    \label{fig:twointervalsingleapproach}%
\end{figure}

\begin{figure}[ht]
     \subfloat[$(\alpha_1,\beta_1) = (4,1.61)$]
    {{ \centering \includegraphics[width=4.1cm]{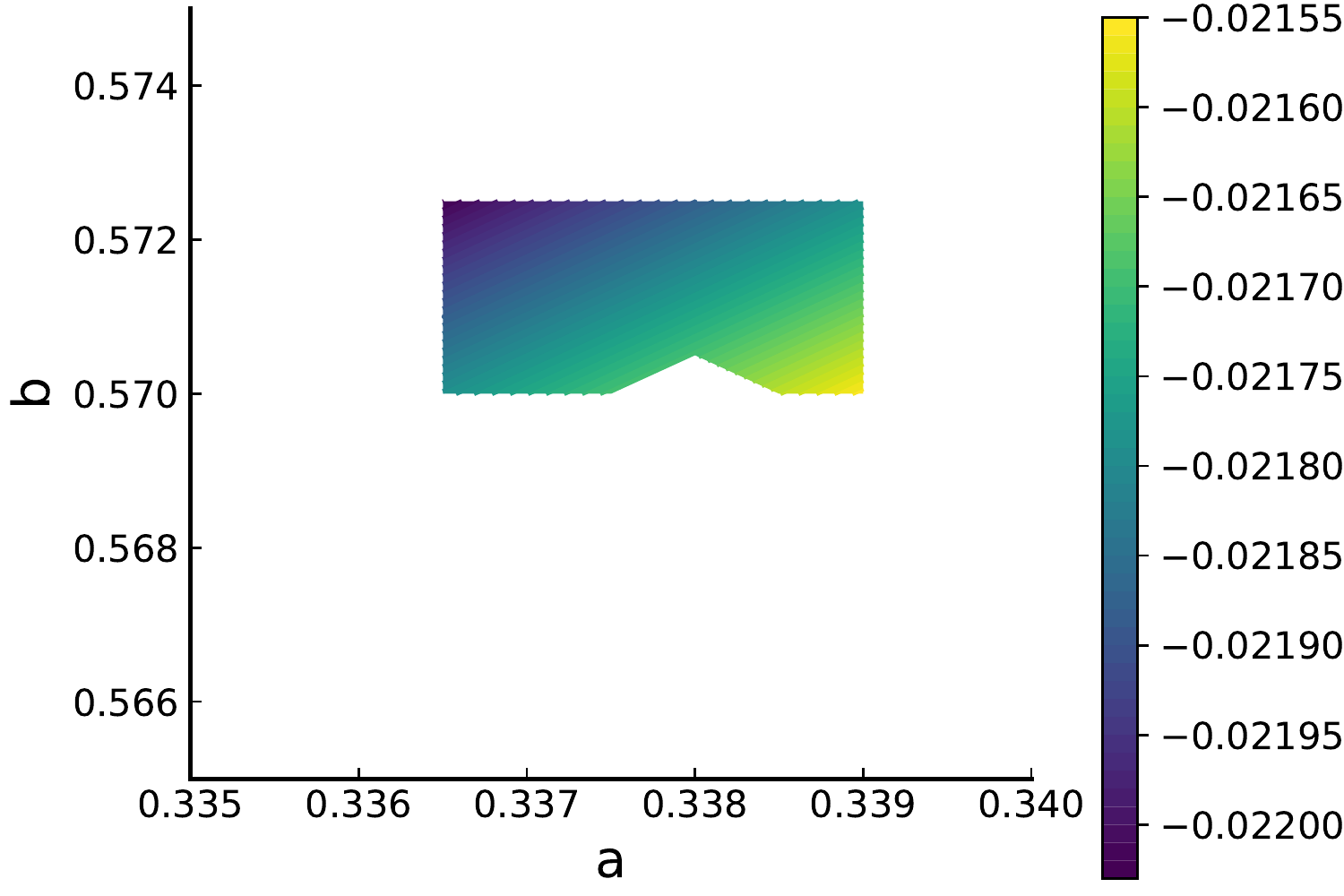} }}
     \subfloat[$(\alpha_2,\beta_2) = (3.34,1.83)$]
    {{ \centering \includegraphics[width=4.4cm]{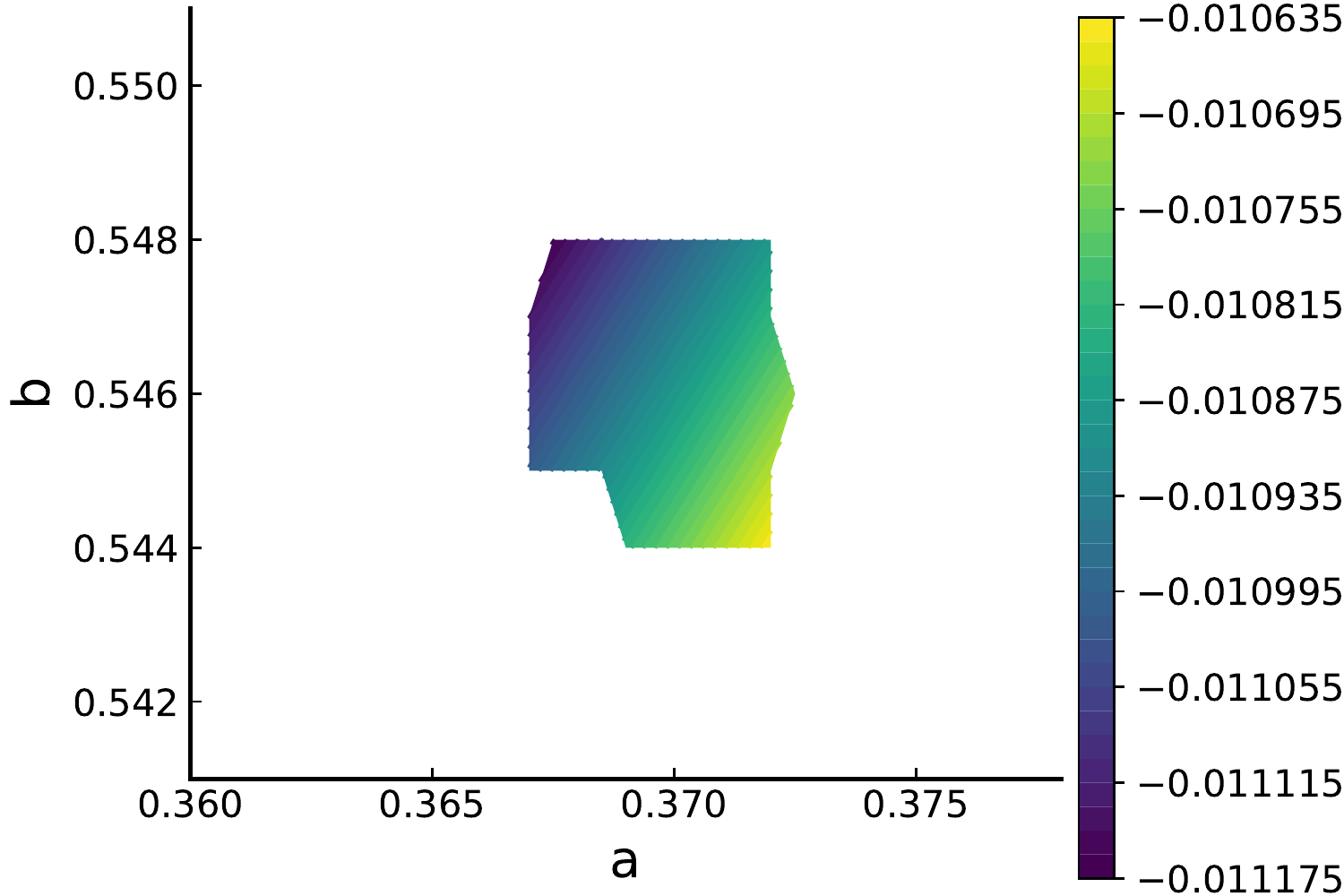} }}
         \subfloat[$(\alpha_3,\beta_3) = (3.5,1.6)$]
    {{ \centering \includegraphics[width=4.1cm]{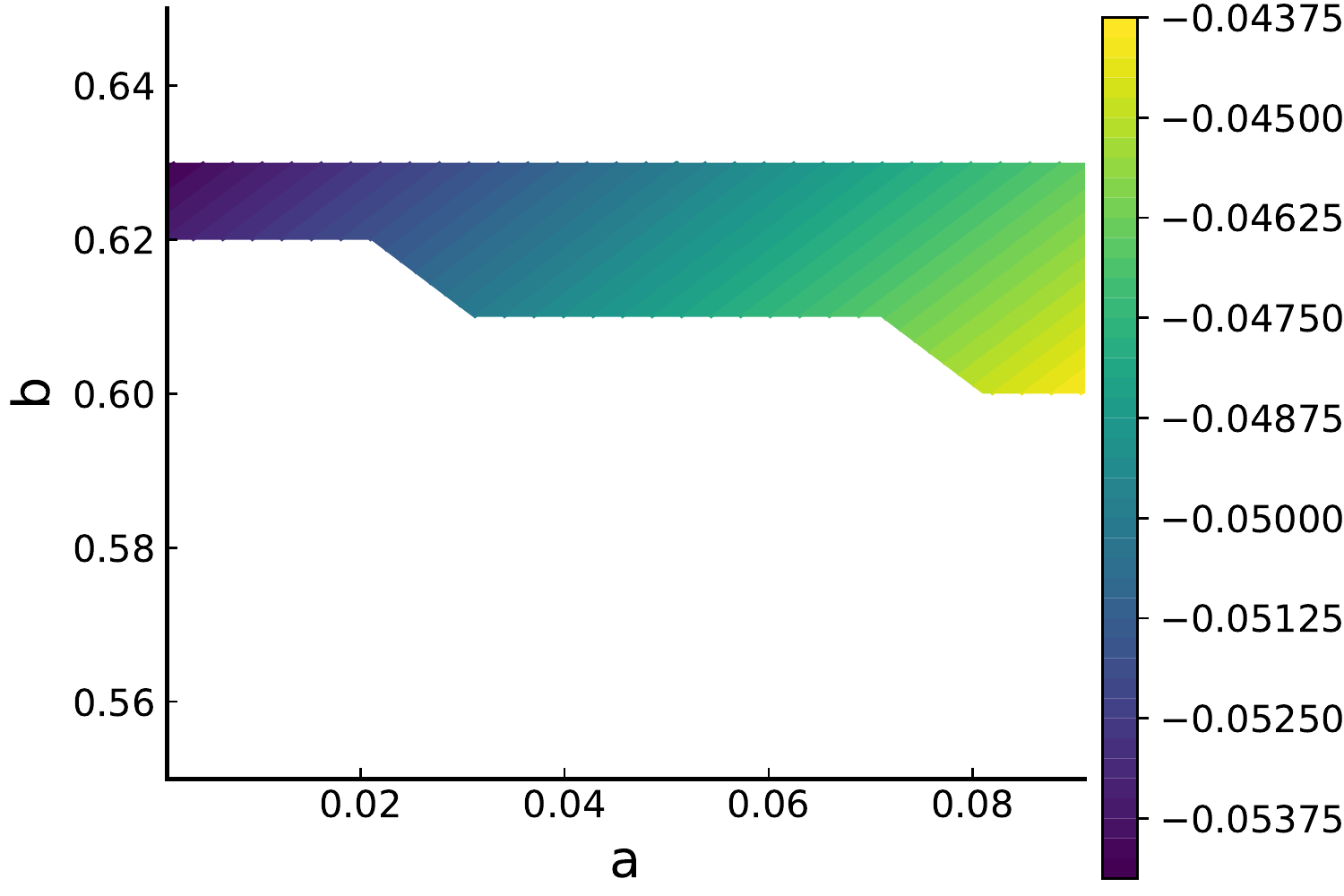} }}
    \caption{Energy contour plots with respect to $(a,b)$ for measures with two interval support $[-b,-a]\cup[a,b]$ for parameters in \eqref{eq:twointervalexp}, showing the approximate regions in which the obtained measures are positive and thus admissible. Note the implication of $a=0$ for the control case parameters $(\alpha_3,\beta_3)$, as in this case an admissible single interval support measure exists.}%
    \label{fig:energycontours}%
\end{figure}
\subsection{Transition from single to two interval support}\label{sec:twointervalexperiments2}
Finally, equipping an algorithm with a simple check for negative values allows it to automatically determine for which parameters $\beta$ it should perform a two interval approach. Using this as well as initial guesses for the optimizations obtained from small-scale particle simulations we can fully automate the procedure to study the transition from single to two interval support. Figure \ref{fig:phasetransition} shows three snapshots of this transition for $\alpha = 4$, where previous results \cite{carrillo_explicit_2016} indicate the split into two intervals should occur at $\beta = 1.5$. We have made an animation of the transition in steps of size $0.0001$ available at \cite{gutleb_1d_2020}. As we have seen above, the method we have introduced matches both analytically known as well as conjectured results for the single interval support parameter ranges as well as discrete particle simulation results in the high two interval support ranges. In the transition phase depicted in Figure \ref{fig:phasetransition} and the animation in \cite{gutleb_1d_2020}, no other method is currently known which could be used as independent evidence for the method's accuracy, as particle simulations in this range do not converge to a particular gap size as the number of particles is increased within reasonable computing times. As such, our results are the first which allow us to probe this transition phase of the support.
\begin{figure}[ht]
     \subfloat[]
    {{ \centering \includegraphics[width=6cm]{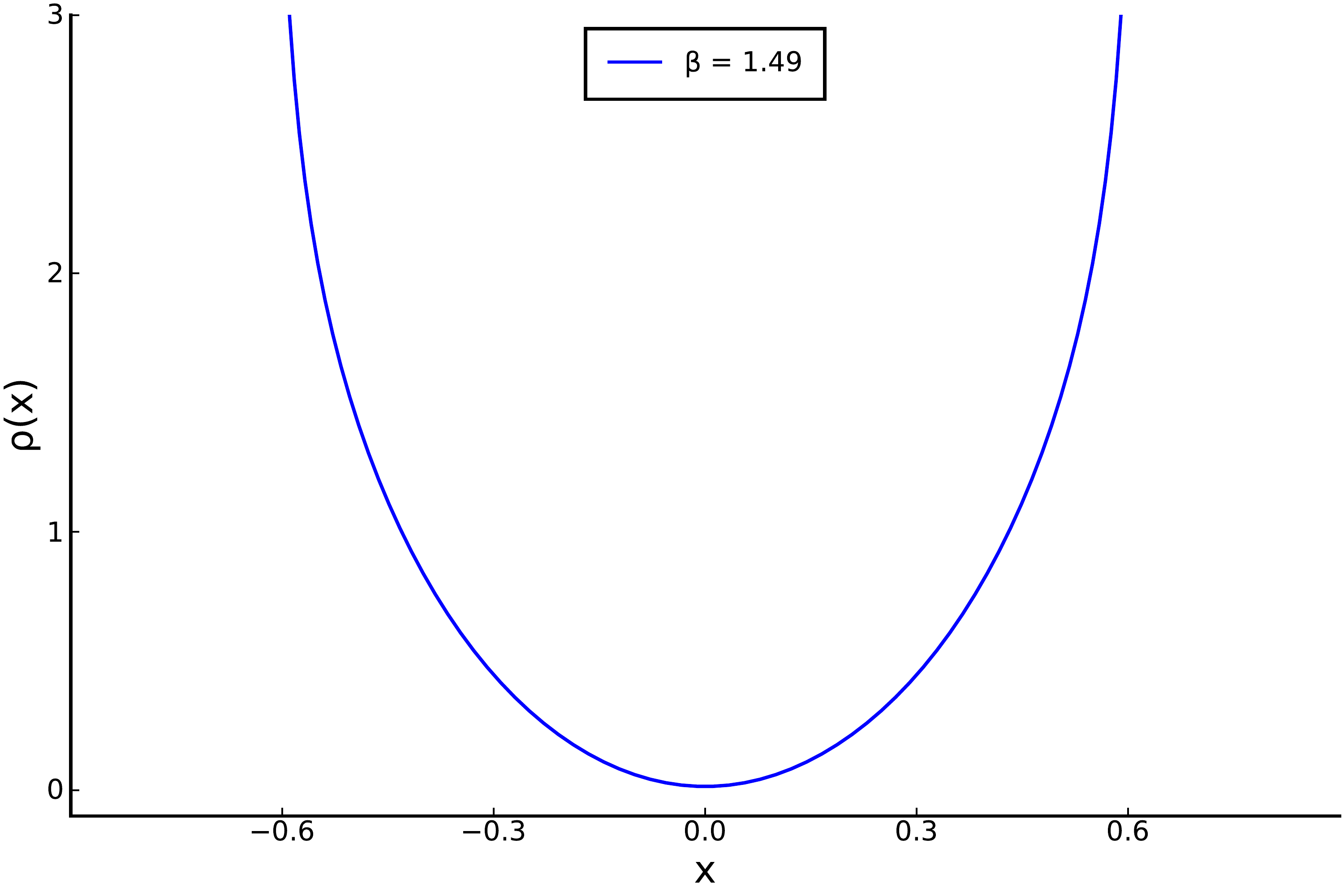} }}
     \subfloat[]
    {{ \centering \includegraphics[width=6cm]{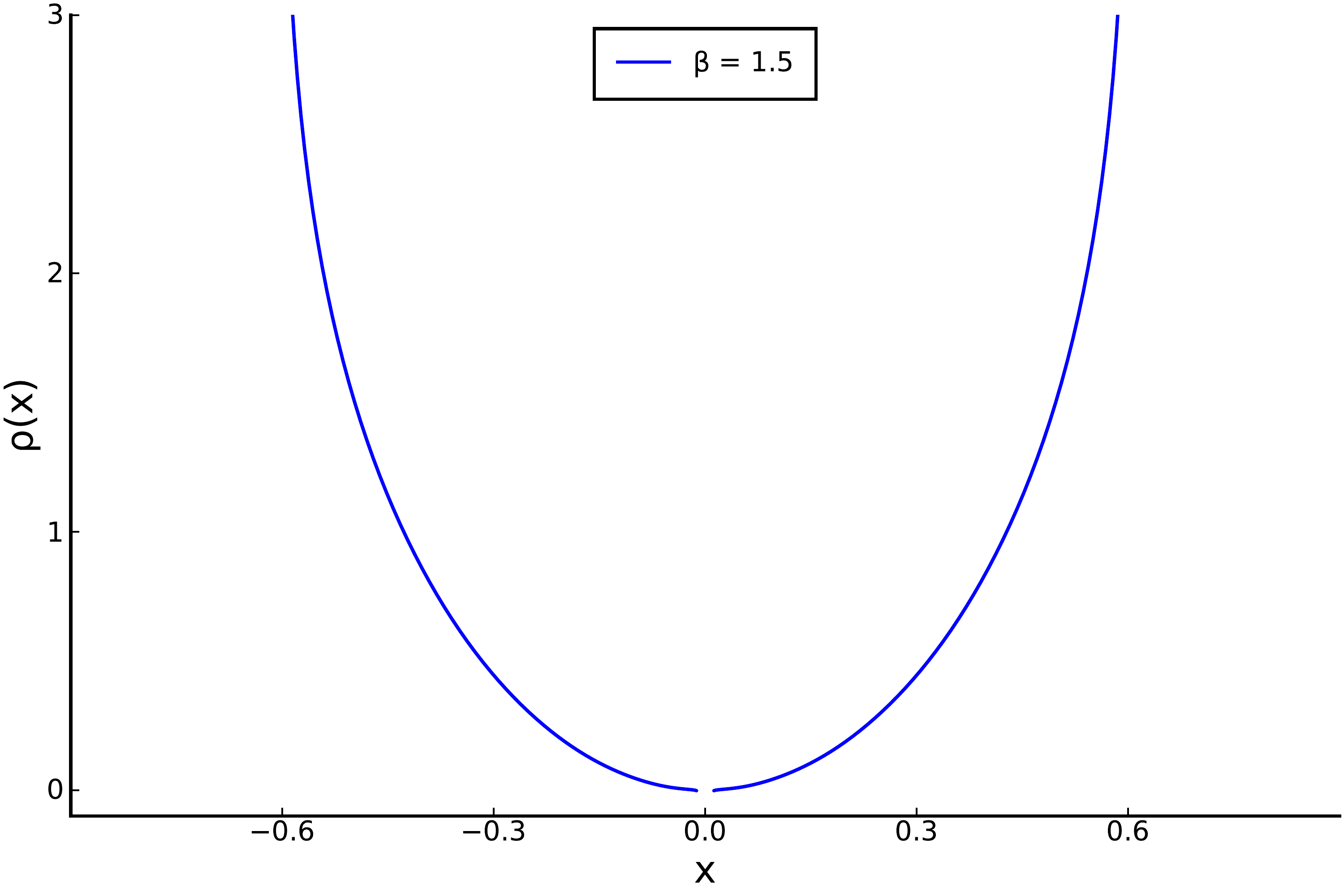} }}\\
         \subfloat[]
    {{ \centering \includegraphics[width=6cm]{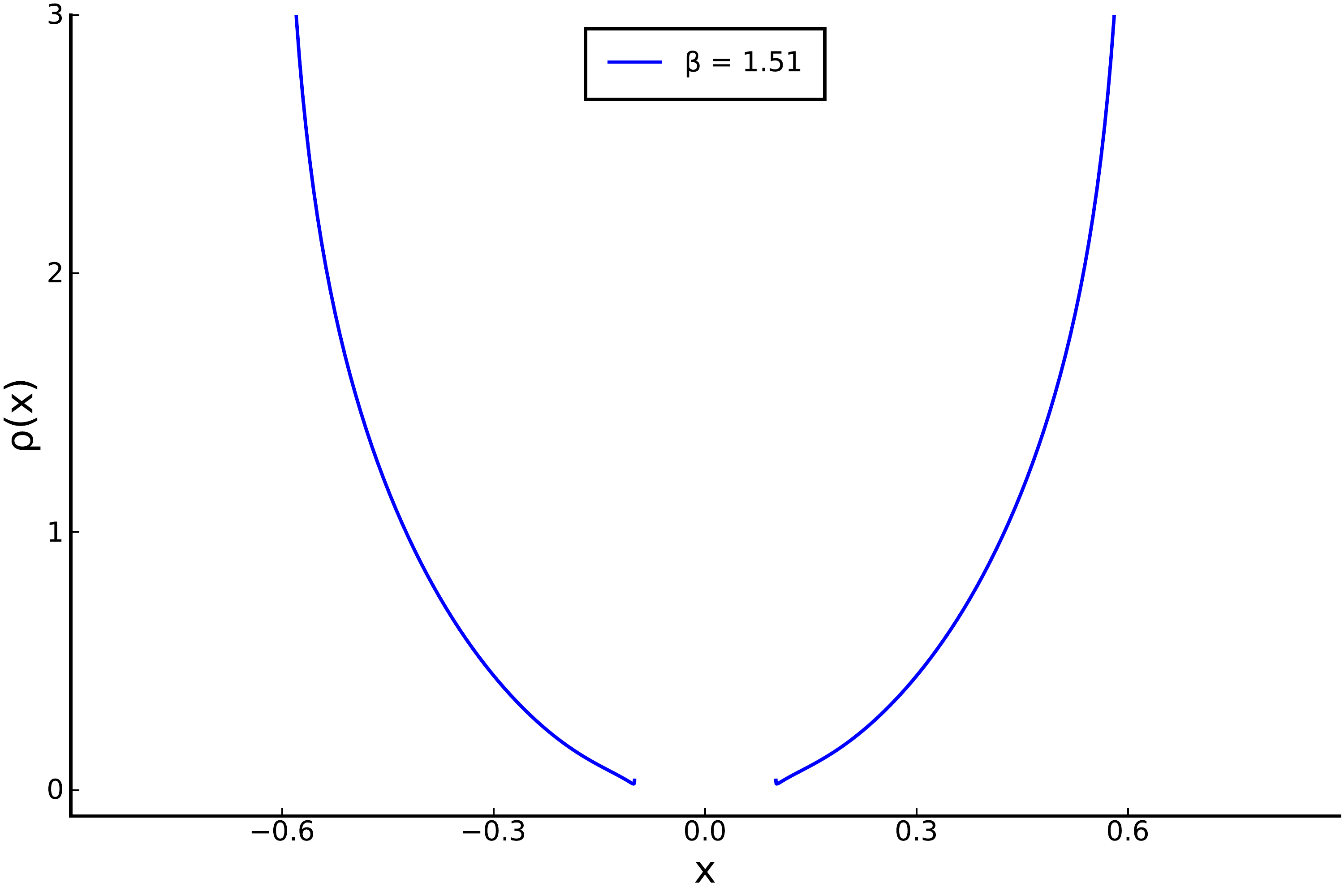} }}
             \subfloat[]
    {{ \centering \includegraphics[width=6cm]{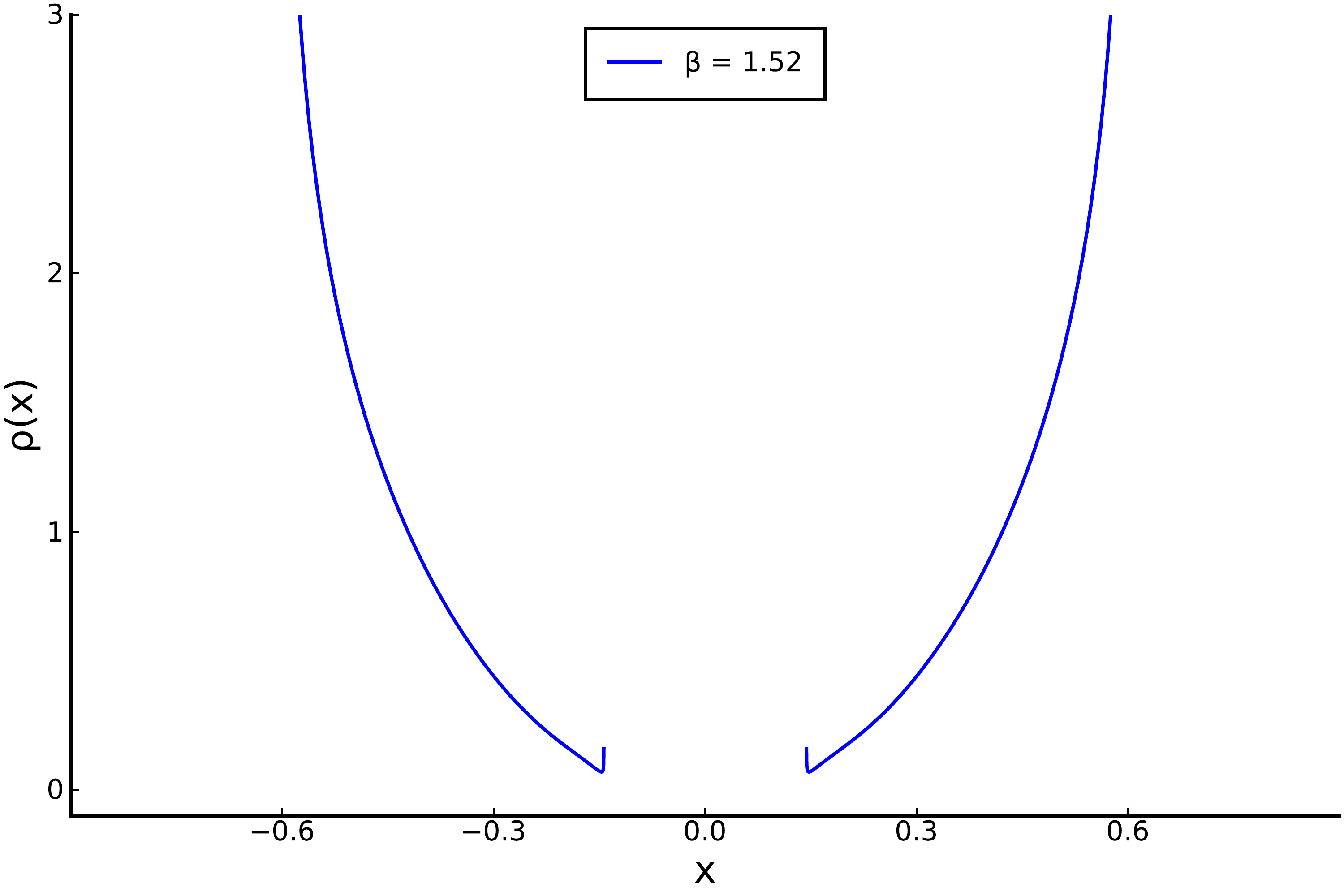} }}
    \caption{Equilibrium measures for $\alpha = 4$ around the transition from single to two interval support measures, which based on \cite{carrillo_explicit_2016} is expected to take place at $\beta = 1.5$. An animation of this process has been made available at \cite{gutleb_1d_2020}.}%
    \label{fig:phasetransition}%
\end{figure}
\section{Discussion} \label{sec:discussion}
We introduced a method combining ultraspherical sparse spectral and optimization approaches which provide a fast and accurate way to compute one-dimensional equilibrium measures for power law kernels of arbitrary parameters where solutions exist. The presented approach reduces what is na\"\i vely a difficult optimization problem on the space of positive measures to an optimization over the support boundary. We also considered cases where an external potential is present, resulting in a no longer translation invariant and possibly asymmetric boundary of support depending on the shape of the potential.\\
The present approach yields novel numerical insights into analytically open questions. In Section \ref{sec:alpha4} we showed numerical experimental evidence for uniqueness of the global minimizers in more general parameter ranges than has presently been proved and presented. We also discussed the phenomenon of gap formation for certain parameter ranges in Section \ref{sec:beta1}, \ref{sec:twointervalexperiments} and \ref{sec:twointervalexperiments2}, which has previously been suggested to occur in discrete particle swarm numerics but nothing was known about the continuous case. Methods to approximate the continuous model in this direct manner allow us to obtain some guidance and insights for results which may be provable in an analytic manner in the future.\\
Many of the analytic results cited in this paper such as \cite{carrillo_explicit_2016,lopes_uniqueness_2019,canizo_existence_2015}, also discuss similar results which hold for power law kernel equilibrium problems in higher dimensions. We expect that a higher dimensional extension of the method presented in this paper is feasible on the basis of multivariate orthogonal polynomials on appropriate domains and intend to address this in future work.

\section*{Acknowledgments}

We would like to thank Nick Hale for many discussions on two-sided fractional integrals. TSG would like to thank Jakob M\"oller for useful discussions on measure theory. JAC was supported by the Advanced Grant Nonlocal-CPD (Nonlocal PDEs for Complex Particle Dynamics: 	Phase Transitions, Patterns and Synchronization) of the European Research Council Executive Agency (ERC) under the European Union's Horizon 2020 research and innovation programme (grant agreement No. 883363). SO was supported by the Leverhulme Trust Research Project Grant  RPG-2019-144. JAC and SO were also supported by the Engineering and Physical Sciences Research Council (EPSRC) grant EP/T022132/1.

\bibliographystyle{siamplain}

\bibliography{references}

\end{document}